\documentclass[10pt]{article}

\usepackage{geometry} 
\usepackage{caption}
\usepackage[figurename=FIG.]{caption}

\usepackage{amsmath}
\usepackage{amsthm}
\usepackage{amssymb}

\usepackage{tikz}
\usepackage{tkz-euclide}
\usetkzobj{all}
\usetikzlibrary{positioning}
\usetikzlibrary{intersections}

\usepackage{pgfplots}
\usepgfplotslibrary{fillbetween}

\theoremstyle{plain}
\newtheorem{lem}{Lemma}[section]
\newtheorem{thm}{Theorem}[section]
\newtheorem{prp}{Proposition}[section]
\newtheorem{cor}{Corollary}[section]

\theoremstyle{definition}
\newtheorem{dfn}{Definition}[section]

\theoremstyle{remark}
\newtheorem{rem}{Remark}[section]
\newtheorem{bsp}{Example}[section]
\newtheorem{ntn}{Notation}[section]

\DeclareMathOperator*{\im}{im\hspace{0.025cm}}
\DeclareMathOperator*{\spec}{spec}
\DeclareMathOperator*{\FinRk}{FinRk}

\DeclareMathOperator*{\curv}{Curv}
\DeclareMathOperator*{\mcurv}{Curv_{m}}
\DeclareMathOperator*{\essinf}{ess\,inf}
\DeclareMathOperator*{\esssup}{ess\,sup}
\DeclareMathOperator*{\End}{End}

\usepackage{hyperref}

\begin{document}

\title{\raggedright\Large{\textbf{$L^{2}$-Wasserstein distances of tracial $W^{*}$-algebras\\
and their disintegration problem}}}
\date{}

\maketitle

\vspace{-1.8cm}

\begin{flushleft}
David F. Hornshaw\\
\textit{Institute for Applied Mathematics, University of Bonn\\
Endenicher Allee 60, 53115 Bonn, Germany\\
\textnormal{E-mail:} hornshaw@iam.uni-bonn.de}
\end{flushleft}

\begin{abstract}
\noindent We introduce $L^{2}$-Wasserstein distances on densities of tracial $W^{*}$-algebras based on a Benamou-Brenier formulation, replacing multiplication by densities with multiplication operators arising as the logarithmic mean under a functional calculus. Furthermore, we concern ourselves with $L^{2}$-Wasserstein distances induced by decomposed derivations on $C^{*}$-algebras of continuous sections of a $\mathcal{K}(H)$-bundle vanishing at infinity. We prove a distintegration theorem for such distances, introduce mean entropic curvature bounds in case $H$ is finite-dimensional and show control of these by the essential infimum of the entropic curvature bounds on the fibres. To conclude, we give sufficient conditions for disintegrating arbitrary $L^{2}$-Wasserstein distances for unital $C^{*}$-algebras that are Morita equivalent to a commutative unital $C^{*}$-algebra.
\end{abstract}

\tableofcontents

\newpage

\section*{Introduction}
We present our motivation and main results, after which we give an overview of the paper's content and structure. The overview is followed by a discussion concerning earlier results by Carlen and Maas for certain finite-dimensional cases.\newline\par

\textbf{Motivation.} Our motivation is to conduct geometric analysis on noncommutative spaces. For this, understanding of noncommutative curvature is essential. Doing so presents an ongoing challenge, one prominent approach being modular curvature for the special case of noncommutative tori. This was developed in \cite{CoMoModCur2T}, \cite{KhalkaliScalCurv4Tori} and \cite{LeMoModCurvME}. In general, an inability to access \textit{local} information prevents straight-forward generalisation of classical definitions to the noncommutative setting. An example of this is a lack of elementary ODE theory, even the notion of a chart, in the proper noncommutative setting. On the other hand, fruitful geometric analysis is possible in case we only have curvature bounds. A fundamental example is Bochner's inequality, see Li and Yau \cite{LiYauParKer}. Such bounds are \textit{global} information of the underlying space and can be expressed synthetically, hence we expect them to have a noncommutative analogue.\par
Rather than understanding curvature directly, we therefore seek to establish a noncommutative analogue of Ricci curvature bounds. Synthetic Ricci curvature bounds for metric measure spaces in the form of entropic curvature bounds were introduced by Sturm \cite{SturmGMMSI}. Utilising $L^{2}$-Wasserstein distances, this approach leads to rich metric geometry for metric measure spaces beginning in \cite{SturmGMMSI} and \cite{SturmGMMSII}. The paper \cite{ErKuStEquivalence} by Erbar, Kuwada and Sturm shows equivalence of various curvature bound conditions, as well as an analogue of Bochner's inequality for specific metric measure spaces. Other examples of Bochner's inequality applied to the geometric analysis of singular spaces are \cite{GiKuwHeatAlex}, \cite{OhStBeWeitz} and \cite{ZhaZhuAlexandrov}.\newline\par

\textbf{Main Results.} Our main results are twofold. Firstly, we introduce $L^{2}$-Wasserstein distances on the space of densities of a tracial $W^{*}$-algebra. A $C^{*}$-algebra $A$ equipped with a l.s.c.~semi-finite trace $\tau$ on $A$ yields a tracial $W^{*}$-algebra $L^{\infty}(A,\tau)$ represented over $L^{2}(A,\tau)$. Given a type of $A$-derivation $\partial$ from $L^{2}(A,\tau)$ to a submodule of a sum $\bigoplus_{k=1}^{m}L^{2}(A,\tau)$, we follow a Benamou-Brenier approach to define the $L^{2}$-Wasserstein distance. We will call such derivations symmetric gradients. The distance is given on the space of densities $\mathcal{D}:=\{p\in L_{+}^{1}(A,\tau)\ |\ \tau(p)=1\}$ by the minimisation problem

\begin{align*}
\mathcal{W}_{2}(p,q):=\inf_{(\rho_{t},v_{t})\in\mathcal{A}(p,q)}\sqrt{\frac{1}{2}\int_{0}^{1}||v_{t}||_{\rho_{t}}dt}
\end{align*}

\noindent where $\rho_{t}\in\mathcal{D}$, $\rho_{0}=p$, $\rho_{1}=q$, while $v_{t}$ lies in a tangent space constructed over each $\rho_{t}$. Furthermore, we demand a continuity equation $\frac{d}{dt}\tau(pa)=\langle v_{t},a\rangle_{\rho_{t}}$ to be satisfied. Here, $a$ is an element of an appropriate $^{*}$-subalgebra $\mathfrak{A}\subset A\cap D(\partial)$ with

\begin{align*}
||a||_{\rho_{t}}^{2}=\sum_{k=1}^{m}\int_{0}^{1}\tau(\rho_{t}^{1-\alpha}\partial a\rho_{t}^{\alpha}\partial a)d\alpha
\end{align*}

\noindent yielding the tangent space over $\rho_{t}$ via Hausdorff completion of $\mathfrak{A}$. Our choice of inner product arises naturally from a noncommutative chain rule if we wish to retain classical relationships between finiteness of $\mathcal{W}_{2}$ on bounded densities, the heat flow and relative entropy.\par
Secondly, we consider examples of form $C_{0}(X,\mathcal{K}(H))$ and give conditions for disintegrating an $L^{2}$-Wasserstein distance into $L^{2}$-Wasserstein distances for $(\mathcal{K}(H),\textrm{tr})$. Let $X$ be a locally compact Hausdorff space with Radon measure $\nu$ such that $(X,\mathcal{B}(X))$ is a separable measure space. Then $A=C_{0}(X,\mathcal{K}(H))$ equipped with trace

\begin{align*}
(\nu\otimes\textrm{tr})(F):=\int_{X}\textrm{tr}(F(x))d\nu
\end{align*}

\noindent provides a setting giving rise to well-behaved $L^{2}$-Wasserstein distances. A sufficiently regular decomposition $\partial=(\partial_{x})_{x\in X}$ into symmetric gradients $\partial_{x}$ for $(\mathcal{K}(H),\textrm{tr})$ induces a distance for which minimisers disintegrate. The disintegration theorem reads as follows:

\setcounter{section}{4}
\setcounter{thm}{0}

\begin{thm}
Let $\partial$ be a vertical gradient such that $\partial_{x}$ has continuous dependence of minimisers on start- and endpoints for a.e.~$x\in X$. For all $P,Q\in\mathcal{D}$ with finite distance, we have 

\begin{align*}
\mathcal{W}_{2}^{2}(P,Q)=\int_{X}\mathcal{W}_{2,x}^{2}(\theta_{P}(x)^{2}P(x),\theta_{P}(x)^{2}Q(x))d\nu_{P}
\end{align*}

\noindent and there exists a minimiser $\mu_{t}$ of $\mathcal{W}_{2}(P,Q)$ such that $\theta_{P}(x)^{2}\mu_{t}(x)\in\mathcal{M}(\theta_{P}(x)^{2}P(x),\theta_{P}(x)^{2}Q(x))$ for a.e.~$x\in X$. 
\end{thm}

\setcounter{section}{0}
\setcounter{thm}{0}

\noindent In the theorem's formulation, $P,Q\in\mathcal{D}\subset L^{1}(X,\mathcal{S}_{1}(H))$, $d\nu_{P}=\textrm{tr}(P(x))d\nu$ and $\mathcal{W}_{2,x}$ is the $L^{2}$-Wasserstein distance induced by $\partial_{x}$ on density matrices. Furthermore, we have 

\begin{align*}
\theta_{P}(x):= \begin{cases}\big(\textrm{tr}(P(x))\big)^{-\frac{1}{2}} & \textrm{if}\ P(x)\neq 0 \\
0 & \textrm{else}
\end{cases}
\end{align*}

\noindent for each $P\in\mathcal{D}$ and $\mathcal{M}(p,q)$ is the set of minimisers for density matrices w.r.t~the appropriate fibre-geometry. If $H$ is finite-dimensional, Theorem \ref{THM.Disint} implies existence of minimisers between all $P,Q\in\mathcal{D}$ with $\textrm{tr}(P(x))=\textrm{tr}(Q(x))$ for a.e.~$x\in X$. We extend the theorem to $\mathcal{K}(H)$-bundles and give sufficient conditions for viewing an arbitrary $L^{2}$-Wasserstein distance as one that disintegrates in the above sense. To decide if a disintegration is possible for a given $L^{2}$-Wasserstein distance is its \textit{disintegration problem}.\newline\par

\textbf{Content overview.} A Benamou-Brenier approach to $L^{2}$-Wasserstein distances \cite{BenBreW2Mech} requires us to introduce a concept of noncommutative gradient. In our case, these are symmetric derivations on a $C^{*}$-algebra $A$ taking values in a symmetric Hilbert $A$-bimodule. Such derivations appear in \cite{CiDrchltFrmsNCS} and \cite{CiSaDrivSqRtsDrchltFrms} as the natural noncommutative extension of gradients induced by Dirichlet forms. The strongly related notion of derivation on a $W^{*}$-algebra was studied in detail in \cite{WeavDerI} and \cite{WeavDerII}, with \cite{WeavDerII_Err} as list of errata. Symmetry and the Leibniz rule give rise to a noncommutative chain rule, yielding

\begin{align*}
\partial\log x =\big((L_{x}\otimes R_{x})(D\log)\big)(\partial x).
\end{align*}

\noindent Here, $x$ is an appropriate element of the domain $D(\partial)$, $D\log$ is the quantum derivative of the logarithm, and $L_{x}\otimes R_{x}$ is a $C^{*}$-representation of $C(\spec(x)\times\spec(x))$ over $H$ induced by the bimodule action of $A$. In general, we expect

\begin{align*}
(L_{x}\otimes R_{x})(D\log)=x^{-1}
\end{align*}

\noindent to be true if and only if $L_{x}=R_{x}$ holds. If we wish to maintain classical relations between relative entropy, heat flow and finiteness of $L^{2}$-Wasserstein distances on bounded densities, we are lead to replace multiplication by a density with multiplication by an operator involving functional calculus and the representation above. More precisely, we define 

\begin{align*}
M_{p}=(L_{p}\otimes R_{p})(M_{lm})
\end{align*}

\noindent for each density $p\in L^{\infty}(A,\tau)$. Here, $M_{lm}$ is the logarithmic mean. If $p$ is invertible, this implies $M_{p}=(L_{x}\otimes R_{x})(D\log)^{-1}$, assuring

\begin{align*}
M_{\rho_{t}}(\partial\log \rho_{t}) = \partial \rho_{t}
\end{align*}

\noindent to remain true for $\rho_{t}:=e^{-t\Delta}p_{0}$. Here, $\Delta=\partial^{*}\partial$ is the induced Laplace operator and $p_{0}$ a bounded density. Having constructed our multiplication operator, we define the tangent space at a bounded density analogous to the commutative case. From this, an $L^{2}$-Wasserstein distance on the space of \textit{bounded} densities $\mathcal{D}_{b}$ follows via minimisation over admissible paths.\par
Our finiteness result requires a setting similar to well-behaved commutative ones, for example a compact Riemannian manifold. Setting $P_{t}:=e^{-t\Delta}$, we first prove

\setcounter{section}{2}

\begin{thm}
If $\partial$ satisfies a Poincar\'e-type inequality, the distance between any two invertible bounded densities is finite. If $A$ is unital, $p\in\mathcal{D}_{b}$ and $P_{t}$ regularity improving, the distance between $p$ and $\rho_{t}:=P_{t}(p)$ is finite for each $t\in [0,1]$.
\end{thm}

\noindent We thus recover the classical case and thereby justify our initial definitions. Moreover, we lift a theorem first proved by Simon \cite{SiPosImpr} to the noncommutative setting showing

\begin{thm}
If $e^{tL}$ is a semigroup of self-adjoint, positivity preserving operators on $L^{2}(A,\tau)$, it is positivity improving if and only if it is ergodic.
\end{thm}

\setcounter{section}{0}

\noindent Ergodicity therefore becomes a necessary condition for $P_{t}$ to be regularity improving. Do note that so far, we have restricted ourselves to bounded densities.\par 
In order to extend the distance to unbounded densities, we require additional assumptions on domain and codomain of our symmetric gradients. To begin with, we assume $\partial$ to take values in a symmetric Hilbert $L^{\infty}(A,\tau)$-subbimodule of a sum $\bigoplus_{k=1}^{m}L^{2}(A,\tau)$ equipped with the canonical $L^{\infty}(A,\tau)$-bimodule action and Hilbert space structure. Next, we assume existence of an extension algebra $\mathfrak{A}\subset A\cap D(\partial)$ which lies dense and such that $\alpha\longmapsto p^{\alpha}\partial ap^{1-\alpha}$ is an element of $L^{1}([0,1],L^{1}(A,\tau))$ for each $a\in\mathfrak{A}$. Under these assumptions, the first statement of Proposition \ref{PRP.dlog} is sufficient to extend $\mathcal{W}_{2}$ to unbounded densities. Multiplication operators will be of form

\begin{align*}
\big(M_{p}(x)\big)_{k}=\int_{0}^{1}p^{\alpha}x_{k}p^{1-\alpha}d\alpha\in L^{1}(A,\tau)
\end{align*}

\noindent for $p\in\mathcal{D},x\in\bigoplus_{k=1}^{m}L^{\infty}(A,\tau)$ and $k\in\{1,...,m\}$. This yields the tangent space norm on elements of $\mathfrak{A}$ we wrote down at the very beginning. Knowing this, extending the distance becomes a straightforward task.\par
Turning to vertical gradients, our first task is to understand $L^{p}$-spaces that are defined by $(C_{0}(X,\mathcal{K}(H)),\nu\otimes\textrm{tr})$. Proposition \ref{PRP.Prd_Tr_LP} shows these to equal $L^{p}(X,\mathcal{S}_{p}(H))$ for each $p\in[0,\infty]$. We construct vertical gradients as $(\partial F)(x)=\partial_{x}F(x)$ for each $F\in D(\partial)$, where $\partial_{x}$ is a symmetric gradient for $(\mathcal{K}(H),\textrm{tr})$ mapping to some $\bigoplus_{k=1}^{m}\mathcal{S}_{2}(H)$ with $m\in\mathbb{N}$ fixed. We demand $C_{c}(X)\odot\FinRk(H)$ to be an extension algebra. As outlined above, this allows us to extend multiplication operators to densities. We have

\begin{align*}
\big(M_{P}(F)\big)_{k}(x)=\int_{0}^{1}P(x)^{\alpha}F_{k}(x)P(x)^{1-\alpha}d\alpha\in L^{1}(X,\mathcal{S}_{1}(H))
\end{align*}

\noindent for each $P\in\mathcal{D}, F\in\bigoplus_{k=1}^{m}L^{\infty}(X,\mathcal{B}(H))$ and $k\in\{1,...,m\}$. We impose other conditions ensuring mass preservation in each fibre, in particular assuming each $\partial_{x}$ to be a fibre gradient. The latter is a notion introduced at the very beginning of Subsection 3.2. Mass preservation in almost every fibre is necessary for showing the disintegration theorem and proved in Proposition \ref{PRP.Mass_Prsv_Unbd}. First consequences are given in Corollary \ref{COR.Mass_Prsv_Unbd}. For example, $\mathcal{D}$ disintegrates into subspaces whose elements have the same mass in almost every fibre. This gives rise to $L^{2}$-Wasserstein distances that do not metrisise the $w^{*}$-topology even as the underlying $C^{*}$-algebra is unital, see Remark \ref{REM.Mass_Prsv_Unbd}.\par
Before proving the disintegration theorem, we concern ourselves with symmetric gradients on $(\mathcal{K}(H),\textrm{tr})$. These present our fibre-geometries. Proposition \ref{PRP.Mass_Prsv} yields the aforementioned conditions for mass preservation along fibres. Furthermore, we introduce the notion of continuous dependence on start- and endpoints necessary for a measurable selection theorem used in the proof of Theorem \ref{THM.Disint}. The proof itself is divided into two parts. In the first, we show every admissible path to have a representative inducing an admissible path on almost every fibre. In the second, we utilise a measurable selection theorem to show existence of an integrable choice of fibre-wise minimisers. Key for the second step is approximation of marginals by well-chosen step functions and utilisation of continuous dependence of minimisers on start- and endpoints. This will enable us to show a condition required for applying the measurable selection theorem.\par
As an application, we consider mean entropic curvature bounds. If $H$ is finite, the relative entropy for any density $p$ in a fibre is $\textrm{tr}(p\log p)$. For bounded densities $P$ on a compact $X$, the noncommutative relative entropy becomes

\begin{align*}
\textrm{Ent}_{m}(P|\nu\otimes\textrm{tr})=\int_{X}\textrm{tr}(P(x)\log P(x))d\nu
\end{align*}

\noindent which also makes sense for unbouded densities. Declaring it to be the mean relative entropy, we consider synthetic Ricci curvature bounds in analogy to the commutative case. Doing so, we obtain global curvature bounds both on the fibres and the whole geometry. Adapting the notion of continuous dependence of minimisers on start- and endpoints for the proof, we obtain

\setcounter{section}{4}
\setcounter{thm}{1}

\begin{thm}
If $\partial$ is a vertical gradient for $(C_{0}(X,M_{n}(\mathbb{C})),\nu\otimes\textrm{tr})$, then 

\begin{align*}
\mcurv(\nu\otimes\textrm{tr},\partial)\geq \essinf_{x\in X}\curv(M_{n}(\mathbb{C}),\textrm{tr},\partial_{x}).
\end{align*}
\end{thm}

\setcounter{section}{0}
\setcounter{thm}{0}

\noindent showing control of global curvature bounds by those for the fibre-geometries. Theorem \ref{THM.Disint} and \ref{THM.NC_MCrv} taken together indicate reasonable control of the global geometry by that of the fibres. In the fifth section, we extend vertical gradients and Theorem \ref{THM.Disint} to the general $\mathcal{K}(H)$-bundle case. This will present no great challenge, as most of the work occurs locally.\par
Lastly, we introduce the disintegration problem. $A\cong\Gamma(\textrm{End}(V))$ for a finite-dimensional hermitian vector bundle $V$ over a compact Hausdorff space $X$ holds if and only if $A$ is Morita equivalent to $C(X)$. Given a symmetric gradient $\partial$ for $(A,\tau)$, we ask if it is possible to find a $C^{*}$-algebra isomorphism from $A$ to some $\Gamma(\textrm{End}(V))$ such that $\partial$ is a vertical gradient after push-forward. While we can show $\tau$ to have form $\nu\otimes\textrm{tr}$ locally, the same cannot be said for symmetric gradients. If for example $A=C(X)$ with non-zero gradient, $V$ must be one-dimensional and $1_{\mathbb{C}}$ the sole density in each fibre. If $\partial$ were vertical, $\partial$ vanishes by the Leibniz rule. Thus vertical gradients are a purely noncommutative phenomenon. We are able to provide sufficient conditions for disintegration in Corollary \ref{COR.Grd_Dcp_I}, after which we briefly outline plans to extend results to fields of elementary $C^{*}$-algebras and beyond.\par
A last word regarding our choice of Wasserstein distance is in order. While easier to handle, $L^{1}$-Wasserstein distances are unsuitable for our purposes. Even in the commutative case and irrespective of the underlying metric measure space, their geodesics are convex combinations of the marginal states. Thus their metric geometry on states is independent of the underlying space's metric geometry. We nevertheless recommend the discussion in \cite{MartiViewOptTrnspNCG} as an introduction and point to \cite{LatrQuanLocCmpct} for a noncommutative $L^{1}$-Wasserstein distance based on ideas first formulated by Rieffel in his paper on compact quantum metric spaces \cite{RiefQuanMetrSp}. Relations between noncommutative $L^{1}$- and $L^{2}$-Wasserstein distances for finite-dimensional $C^{*}$-algebras mirroring the commutative setting have recently been announced \cite{RouNilRelL2andL1LogSobolev}.\newline\par

The noncommutative continuity equation first presented by Carlen and Maas in \cite{CaMaW2RiemI} for the special case of the CAR algebra generated by $n$ bounded operators on a Hilbert space of dimension $n^{2}$, as well as its generalisation to other finite-dimensional cases in \cite{CaMaW2RiemII}, was essentially derived by the same reasoning we apply here. One minor difference is that Carlen and Maas aimed to replace multiplication by $p^{-1}$, rather than $p$, with a noncommutative analogue. While the noncommutative chain rule was not mentioned in either publication, the multiplication operator produced in both is indeed $(L_{\rho_{s}}\otimes R_{\rho_{s}})(D\log)$. This can be seen by applying the Proposition \ref{PRP.dlog}, i.e.~Pedersen's calculus, to the multiplication operators defined in \cite{CaMaW2RiemI} or \cite{CaMaW2RiemII}. As this replacement procedure is key to our approach, we view both papers as foundational.\par
Maas' own work \cite{MaGrdtFlwsEntrpFin} concerning an analogue of $L^{2}$-Wasserstein distances for discrete spaces already utilised a similar replacement procedure involving the logarithmic mean. This is something we expect to see since noncommutative geometry aims, among other things, to unify continuous and discrete geometries. Many of the difficulties we face in the noncommutative setting already arise in the discrete case. Various alternative techniques were used in \cite{ErMaRiCuFinMarkCh} and \cite{ErMaRiCuBL} to obtain analogues of classical curvature bounds, as well as Gromov-Hausdorff convergence for a discrete commutative setting in \cite{GiMaGromHausConvgDiscrtTrnsprt}.\par
Finally, we wish to point out that Wirth is developing an $L^{2}$-Wasserstein distance based on the same replacement procedure we engage in here. We stress that both Wirth and the author developed their approach independently of one another, only realising their ideas' similarity after they had matured. Notes were exchanged. In particular, Lemma \ref{LEM.L2_Pst_Prj} is a slight adaption of a lemma proved by Wirth in future work of his.\newline\par

\noindent\textbf{Structure.} This paper is divided into two major parts. In the first, we introduce noncommutative $L^{2}$-Wasserstein distances and prove finiteness results on bounded densities. These are the first two sections. The second part comprises the last three sections. In the third section, we introduce fibre-geometries in preparation of the vertical gradient case. We deal with vertical gradients on trivial $\mathcal{K}(H)$-bundles and prove Theorem \ref{THM.Disint} and \ref{THM.NC_MCrv} in the fourth section. We extend to the general bundle case and introduce the disintegration problem in the fifth section.

\smallskip

\noindent\textbf{Notation and conventions.} $\mathcal{B}(H)$ is the space of bounded linear operators on a Hilbert space while $\mathcal{K}(H)$ are the compact operators. Furthermore, $\mathcal{S}_{p}(H)$ denotes the Schatten ideals for $p\in [1,\infty]$, with $\mathcal{S}_{\infty}(H)=\mathcal{B}(H)$. For a suitable measure space $(X,\nu)$ and Banach space $E$, $L^{p}(X,E,d\nu)$ denotes the Bochner-$L^{p}$-space. We often drop $\nu$ from our notation. Given a $C^{*}$-algebra $A$, let $A_{h}$ be its self-adjoint and $A_{+}$ its positive elements. We call $\tau$ a trace on $A$ if it is a l.s.c.\@, semi-finite trace according to Definition 6.1.1 in \cite{DixC*Alg} and set $D(\tau):=\{\tau(|a|)<\infty\ |\ a\in A\}$.  In this case, we call $(A,\tau)$ a tracial $C^{*}$-algebra. By default, we consider its n.s.f. extension to $L^{\infty}(A,\tau)$ which we again denote by $\tau$. We write $A\subset L^{p}(A,\tau)$ when considering the image of $A$ under the canonical inclusion.

\smallskip

\noindent\textbf{Standard references.} General references concerning $C^{*}$- and $W^{*}$-algebras are \cite{DixC*Alg}, \cite{BK.PedC*Auto}, \cite{BK.SakC*W*}, and \cite{TakTOAI}. For noncommutative integration theory, the original paper \cite{SegNCInt} and its correction \cite{SegNCIntCorr} provide a detailed introduction, while both \cite{BK.HndbkGeomBS} and \cite{NelNotesNCInt} give streamlined ones. Broad introductions to noncommutative geometry are \cite{CoNCG} and \cite{KhalBasicNCG}, with \cite{VarilAnIntroNCG} focusing on differential geometric aspects from a functional analytic point of view. We recommend \cite{VilOptTrsp} as a reference for $L^{p}$-Wasserstein distances in the commutative case. For a Benamou-Bernier approach in the commutative setting, we refer to \cite{AmbGreenBook}. Results concerning the Bochner integral can be found in the usual works \cite{TreLocConvSpaces} and \cite{ZeiNLFAIIb}. A reference for vector bundles is \cite{BK.LeeSmthMfds}.

\smallskip

\noindent\textbf{Acknowledgements.} The author's position at time of writing was funded by the ERC Advanced Grant \textit{Metric measure spaces and Ricci curvature - analytic, geometric, and probabilistic challenges} awarded to K.-T. Sturm. The author wishes to express gratitude to K.-T. Sturm for his continued advice and support.

\section{Preliminaries}

We define symmetric gradients as noncommutative analogues of gradients into $L^{2}$-sections, provide a noncommutative chain rule, and describe a differential calculus developed by Pedersen \cite{PedOpDiffFct} useful when dealing with Fr\'echet derivatives on $\mathcal{B}(H)$ involving the continuous functional calculus. This provides a standard representation of our multiplication operator in case $H=L^{2}(A,\tau)$. 

\subsection{Gradients into bimodules over $C^{*}$-algebras}

We define bimodules over $C^{*}$-algebras and introduce symmetric gradients. A primary reference and source of examples is \cite{CiDrchltFrmsNCS}. In it, derivation rather than gradient is the preferred terminology. 

\begin{dfn}
Let $A$ be a $C^{*}$-algebra. We define a bimodule over $A$, or simply $A$-bimodule, to be a $C^{*}$-representation $\pi$ of $A\otimes_{max}A^{op}$ over a Hilbert space.
\end{dfn}

\begin{ntn}
A representation $\pi$ induces an $A$-bimodule structure on $H$ in the algebraic sense, with both actions bounded w.r.t. the Hilbert space topology. We use $\pi$ and $H$ interchangably.
\end{ntn}

\begin{dfn}\label{DFN.NCGrad}
Let $(A,\tau)$ be a tracial $C^{*}$-algebra. Furthermore, let $H$ be a bimodule over $A$. A gradient $\partial$ for $(A,\tau)$ is a densely defined, closed linear operator from $L^{2}(A,\tau)$ to $H$ such that
\begin{itemize}
\item[1)] $D(\partial)$ is closed under the $^{*}$-operation on $L^{2}(A,\tau)$,
\item[2)] $A_{\partial}:=A\cap D(\partial)$ is a dense $^{*}$-subalgebra of $A$ and a core of $\partial$,
\item[3)] $\partial$ is an algebra derivation from $A_{\partial}$ to $H$.
\end{itemize}
\end{dfn}

\begin{ntn}
We write $(A,\tau,\partial)$ for a tracial $C^{*}$-algebra $A$ with trace $\tau$ and gradient $\partial$.
\end{ntn}

We wish to make sense of expressions $\partial f(a)$ for self-adjoint $a\in A_{\partial}$ and sufficiently regular $f\in C(\spec(a))$. By necessity, we require $f(a)\in D(\partial)$ to hold. If this is true, we expect a noncommutative chain rule to apply. Such a chain rule exists if we have an involution on $H$ compatible with the gradient.

\begin{dfn}
A bimodule $H$ over $A$ is called symmetric if there exists an isometric, anti-linear involution $J$ on $H$ such that $J(ahb)=b^{*}ha^{*}$ for each $a,b\in A$ and $h\in H$.
\end{dfn}

\noindent We required the domain of our gradient to be closed under adjoining in $A$ by definition, hence compatibility of $\partial$ and $J$ as defined next makes sense. 

\begin{dfn}
If $\partial$ is a gradient for $(A,\tau)$, $\partial$ is symmetric if $\partial(a^{*})=J(\partial a)$ for each $a\in A_{\partial}$.
\end{dfn}

\begin{rem}\label{REM.Cmplxfy}
Replacing $A$ by $A_{h}$, we consider Definition \ref{DFN.NCGrad} for the real case by demanding $H$ to be a real Hilbert space with an $A_{h}$-bimodule structure. In this case, tensoring with $\mathbb{C}$ yields a symmetric gradient $\partial_{\mathbb{C}}:=\partial+i\partial$. 
\end{rem}

For $a\in A_{h}$, $C(\spec(a))$ is unital and thus \textit{not} equal to $C^{*}(a)\subset A$ for non-unital $A$. We thus cannot obtain a representation of $C(\spec(a))\otimes C(\spec(a))$ over $H$ by restricting $\pi$ to $C(\spec(a))\otimes C(\spec(a))^{op}$ if $A$ is non-unital. Instead, we first consider the left representation $L_{a}$ of $C(\spec(a))$ over $H$ uniquely determined by

\begin{align*}
L_{a}(f)(x) := \begin{cases} 
f(a).x & \textrm{if}\ f(0)=0 \\
x & \textrm{if}\ f=1 
\end{cases}
\end{align*}

\noindent We construct a right representation $R_{a}$ of $C(\spec(a))^{op}$ over $H$ analogously, replacing left by right action of $A$ on $H$ in our definition above. Tensoring both $L_{a}$ and $R_{a}$, we have

\begin{align*}
L_{a}\otimes R_{a}:C(\spec(a))\otimes C(\spec(a))^{op}\longrightarrow\mathcal{B}(H)
\end{align*}

\noindent with $L_{a}\otimes R_{a}$ depending on the bimodule $\pi$ by construction. We do not care about this, as $H$ will remain fix once chosen. Commutativity of $C(\spec(a))$ implies $C(\spec(a))\otimes C(\spec(a))$ and $C(\spec(a)\times\spec(a))$ to be isomorphic.

\begin{prp}\label{PRP.LR_Rpr}
Let $H$ be a symmetric bimodule over $A$, $a\in A_{h}$ and $I\subset\mathbb{R}$ a closed interval containing $\spec(a)$. Then for all $f\in C(I\times I)$, we have

\begin{align*}
||(L_{a}\otimes R_{a})(f)||_{\mathcal{B}(H)}\leq ||f||_{C(I\times I)}
\end{align*}

\end{prp}
\begin{proof}
$L_{a}\otimes R_{a}$ is a representation, hence a homomorphism of $C^{*}$-algebras. Thus its norm is less or equal to one and therefore $||(L_{a}\otimes R_{a})(f)||_{\mathcal{B}(H)}\leq ||f||_{C(\spec(a)\times\spec(a))}$, while $||f||_{C(\spec(a)\times\spec(a))}\leq ||f||_{C(I\times I)}$ follows at once from $\spec(a)\subset I$. 
\end{proof}

We are ready to discuss the noncommutative chain rule. It will involve the quantum derivative of a function. Outside of our immediate context, the quantum derivative is a natural analogue of the classical derivative in the discrete setting. Examples are discrete gradients on graphs. A useful introduction is provided by \cite{KacQuantumCalc}. Here, the quantum derivative is simply the correct object when searching for a chain rule involving gradients.

\begin{dfn}
Let $I\subset\mathbb{R}$ be a closed interval. For $f\in C^{1}(I)$, its quantum derivative is

\begin{align*}
Df(s,t):=\begin{cases} 
\frac{f(s)-f(t)}{s-t} & \textrm{if}\ s\neq t\\
f'(s) & \textrm{if}\ s=t 
\end{cases}
\end{align*}

\noindent with $(s,t)\in I\times I$.
\end{dfn}

\noindent $Df$ is continuous by hypothesis on $f$. 

\begin{prp}
Let $\partial$ be a symmetric gradient for $(A,\tau)$ and $a\in A_{h}\cap A_{\partial}$. If $f\in C^{1}(\spec(a))$ such that $f(0)=0$, then

\begin{itemize}
\item[1)] $f(a)\in D(\partial)$ with $\partial(f(a))=(L_{a}\otimes R_{a})(Df)(\partial(a))$,
\item[2)] $||\partial(f(a))||_{H}\leq ||f'||_{C(\spec(a))}||\partial(a)||_{H}$.
\end{itemize}

\noindent If we know $A$ and $H$ to be commutative, then $(L_{a}\otimes R_{a})(Df)(h)=f'(\partial(a)).h$ for each $h\in H$.
\end{prp}
\begin{proof}
The first and second statement are proved in \cite{CiDrchltFrmsNCS}, while the third can be checked immediately on polynomials vanishing at the origin. This extends to all $f$ we consider by density of such polynomials.
\end{proof}

\begin{rem}\label{REM.QD_Intvl}
If $f=g$ on $I$, then $Df=Dg$ in $C(\spec(a)\times \spec(a))$. We are thus able to compute the chain rule for elements $f\in C^{1}(\spec(a))$ even if they are not continuously differentiable outside of $I$, or in case $f(0)\neq 0$ but $0\notin I$. To do so, we simply replace $f_{|I}$ with an appropriate extension $g$ defined on $\mathbb{R}$. The result is independent of our choice.
\end{rem}

\subsection{Pedersen's differential calculus}
In \cite{PedOpDiffFct}, Pedersen developed a differential calculus based on usual Frech\'et differentiation yet well-behaved with respect to functional calculus. For a more thorough treatment, we refer to the original paper. 

\begin{dfn}
Let $H$ be a separable Hilbert space, $I\subset\mathbb{R}$ a closed interval. We denote the space of all self-adjoint, bounded operators over $H$ with spectra in $I$ by $B(H)_{h}^{I}$. We call a function $f:I\longrightarrow\mathbb{R}$ operator differentiable if the map

\begin{align*}
f:B(H)_{h}^{I}\longrightarrow B(H),\ T\mapsto f(T)
\end{align*}

\noindent is Fr\'echet differentiable, and denote its Fr\'echet derivative at $T$ by $df_{T}$.
\end{dfn}

Pedersen showed operator differentiable maps to form a Banach $^{*}$-algebra, denoted by $C_{op}^{1}(I)$. This notation is justified because every operator differentiable function $f$ is continuously Fr\'echet differentiable, cmpr.~Theorem 2.6 in \cite{PedOpDiffFct}. 

\begin{prp}\label{PRP.dlog}
If $H$ is a separable Hilbert space and $I\subset\mathbb{R}_{>0}$ a closed interval, $d\hspace{0.05cm}\log_{\hspace{0.05cm}T}(S)$ is the unique solution to the integral equation

\begin{align*}
\int_{0}^{1}T^{s}XT^{1-s}ds=S
\end{align*}

\noindent for each $S,T\in B(H)_{sa}^{I}$. Let furthermore $(A,\tau)$ be a tracial $C^{*}$-algebra represented over $H$ and $L^{1}(A,\tau)$ separable. If $T\in L^{\infty}(A,\tau)\cap\mathcal{B}(H)_{sa}^{I}$ and $S\in\mathcal{B}(H)_{sa}^{I}$ such that $d\log_{T}(S)\in L^{1}(A,\tau)$, then 

\begin{align*}
\tau(Td\log_{\hspace{0.05cm}T}(S))=\tau(S).
\end{align*}
\end{prp}
\begin{proof}
The first statement is proved on p.~155 of \cite{PedOpDiffFct}. For the second one, note $||T^{\alpha}||_{\infty}||T^{1-\alpha}||_{\infty}=||T||_{\infty}$ since $I\subset\mathbb{R}_{>0}$ closed by hypothesis while $\lambda\longmapsto\lambda^{\alpha}$ increases monotonically for $\lambda>0$, $\alpha\in (0,1]$. For all $R\in L^{1}(A,\tau)$, we therefore know

\begin{align*}
||T^{\alpha}RT^{1-\alpha}||_{L^{1}(A,\tau)}\leq ||T^{\alpha}||_{\infty}||T^{1-\alpha}||_{\infty}||R||_{L^{1}(A,\tau)}=||T||_{\infty}||R||_{L^{1}(A,\tau)}.
\end{align*}

\noindent Multiplication in $L^{\infty}(A,\tau)$ is $||.||_{\infty}$-continuous. Thus $\alpha\longmapsto\tau(T^{\alpha}RT^{1-\alpha}X)=\tau(RT^{1-\alpha}XT^{\alpha})$ is continuous, hence measurable, for arbitrary $X\in L^{\infty}(A,\tau)$. Hence $T^{\alpha}RT^{1-\alpha}$ is Bochner-integrable as path from $[0,1]$ to $L^{1}(A,\tau)$ by separability of the latter. Using continuity of $\tau$ w.r.t.~the $||.||_{L^{1}(A,\tau)}$-topology and $d\log_{T}(S)\in L^{1}(A,\tau)$, we are now able to calculate

\begin{align*}
\tau(Td\log_{\hspace{0.05cm}T}(S)) &= \int_{0}^{1}\tau(Td\log_{\hspace{0.05cm}T}(S))ds\\ \\
&=\int_{0}^{1}\tau(T^{s}d\log_{\hspace{0.05cm}T}(S)T^{1-s})ds\\ \\
&=\tau(\int_{0}^{1}T^{s}d\log_{\hspace{0.05cm}T}(S)T^{1-s}ds)\\ \\
&=\tau(S).
\end{align*}
\end{proof}

\begin{rem}\label{REM.dlog}
By definition of $d\log$ as Fr\'echet derivative, $d\log_{T}(S)\in L^{\infty}(A,\tau)$ if $T,S\in L^{\infty}(A,\tau)$. Hence finiteness of $\tau$ implies $d\log_{T}(S)\in L^{1}(A,\tau)$ whenever $S,T\in L^{\infty}(A,\tau)$.
\end{rem}

We assume $A$ to be a $C^{*}$-algebra representable over a separable Hilbert space for the remainder of this section. Pedersen developed a noncommutative chain rule involving his differential calculus and derivations on $A$. 

\begin{prp}
If $\partial:A\longrightarrow A$ is a closed derivation, then $f(a)\in D(\partial)$ and $\partial(f(a))=df_{a}(\partial(a))$ for each $f\in C_{op}^{1}(I)$ with $\spec(a)\subset I$.
\end{prp}
\begin{proof}
This is Theorem 3.7 in \cite{PedOpDiffFct}.
\end{proof}

\noindent If $H=L^{2}(A,\tau)$, $(L_{a}\otimes R_{a})(Df)$ reduces to Pedersen's derivative $df_{a}$. As such, Pedersen derived a special case of the noncommutative chain rule we discussed above.

\begin{prp}
Let $L^{2}(A,\tau)$ be equipped with the canonical $A$-bimodule structure. If $\partial$ is a symmetric gradient such that $\partial_{|A}$ is a closed derivation on $A$, then $(L_{a}\otimes R_{a})(Df)=df_{a}$ for each $f\in C_{op}^{1}(I)$.
\end{prp}
\begin{proof}
This is checked immediately on polynomials vanishing at the origin, and the general statement follows by density. 
\end{proof}

The case of $H=L^{2}(A,\tau)$ yields a most canonical setting for our extension problem as it will provide an integral representation of our multiplication operator $M_{p}$ given by

\begin{align*}
M_{p}(h)=\int_{0}^{1}p^{\alpha}hp^{1-\alpha}d\alpha.
\end{align*}

\noindent Here, $h\in L^{2}(A,\tau)$ and $p$ is a bounded density. $C^{*}$-dynamical systems induce gradients of this form. As such, even \textit{bounded} symmetric gradients arise canonically in infinite dimensions. All $i[y,\hspace{0.05cm}.\hspace{0.05cm}]$ with $y\in A_{h}$ are of this form. We obtain them by differentiating $\alpha_{t}(x):=e^{tiy}xe^{-tiy}$ at the origin.

\section{$L^{2}$-Wasserstein distances on noncommutative densities}

Starting from noncommutative relative entropy for unital $C^{*}$-algebras, we motivate our notion of multiplication operator. Natural definitions of tangent space, energy functional and $L^{2}$-Wasserstein distance for bounded densities follow. Our justification is completed by finiteness results emulating the compact Riemannian case. We extend $L^{2}$-Wasserstein distances to unbounded densities for symmetric gradients mapping into symmetric Hilbert $L^{\infty}(A,\tau)$-subbimodules and having an extension algebra.

\subsection{Noncommutative relative entropy}

For this subsection, we assume $(A,\tau)$ to be a unital tracial $C^{*}$-algebra with $\tau(1_{A})=1$ such that both $L^{1}(A,\tau)$ and $L^{2}(A,\tau)$ are separable. This occurs if $A$ is separable. Unitality is required when using Petz's variational description of Araki's noncommutative relative entropy. Set $M:=L^{\infty}(A,\tau)$. 

\begin{dfn}\label{DFN.Bd_Dst_Untl}
$\mathcal{D}_{b}:=\{p\in M_{+}\ |\ \tau(p)=1\}$ is the space of bounded densities.
\end{dfn}

The noncommutative relative entropy is defined as a Legendre 	. In the commutative case, this reduces to the familiar representation of the relative entropy as the Legendre dual of the logarithmic Laplace transform. 

\begin{dfn}
For all $p\in\mathcal{D}_{b}$, the noncommutative relative entropy is defined as

\begin{align*}
\textrm{Ent}(p|\tau):=\underset{x\in M_{sa}}{\sup}\{\tau(xp)-\log\tau(e^{x})\}
\end{align*}
\end{dfn}

\noindent We refer to Petz's paper \cite{PetzVarRelEnt} for the variational description we make use of here. See \cite{PetzPrpRelEnt} for a more general description of its operator algebraic properties. Convexity and hence lower semicontinuity of the noncommutative relative entropy in all relevant operator algebraic topologies follow immediately from the definition. The noncommutative relative entropy additionally takes the expected form $\tau(p\log p)$ on bounded densities.

\begin{prp}\label{PRP.Entrp_Rpr}
If $p\in\mathcal{D}_{b}$, then $\textrm{Ent}(p|\tau)=\tau(p\log p)<\infty$. 
\end{prp}
\begin{proof}
Assume $p$ to be invertible. In \cite{PetzVarRelEnt}, the relative entropy $S(\varphi,\omega)$ after Araki is discussed in full generality. Our case reduces to $\varphi=\tau$ and $\omega=\tau(\hspace{0.05cm}.\hspace{0.05cm}p)$. Both are faithful normal states. Traciality of $\tau$ implies equality of $\mathit{\Delta}$ and the identity. In the notation of \cite{PetzVarRelEnt}, this ensures $\varphi^{h}=\tau(\hspace{0.05cm}.\hspace{0.05cm}e^{h})$. Using this, the first proposition in \cite{PetzVarRelEnt} implies

\begin{align*}
S(\varphi,\omega)=\underset{x\in M_{h}}{\sup}\{\tau(xp)-\log\tau(e^{x})\}.
\end{align*}

\noindent The same proposition also  tells us that the supremum is reached if and only if 

\begin{align*}
\tau(\hspace{0.025cm}.\hspace{0.05cm}p)=\tau(\hspace{0.05cm}.\hspace{0.05cm}\frac{e^{x}}{\tau(e^{x})})
\end{align*}

\noindent holds, which is true for $x=\log p$. Thus $\tau(p\log p)-\log(\tau(p))=\tau(p\log p)$ equals the supremum.\par
For arbitrary $p$, we have $p\log p\in M$. This follows by continuity of $\lambda\log\lambda$ on $\mathbb{R}_{\geq 0}$, allowing us to apply Borel functional calculus. Any sequence of invertible operators $p_{i}$ converging to $p$ in the strong operator-topology implies 

\begin{align*}
\textrm{Ent}(p|\tau)\leq\liminf_{i}\textrm{Ent}(p_{i}|\tau)=\tau(p\log p)
\end{align*}

\noindent by lower semi-continuity of the noncommutative relative entropy, as well as continuity of the functional calculus under the strong operator-topology. For the converse, define $x_{\varepsilon}:=\min\{\log p,\varepsilon\}$. Then $\tau(px_{-\varepsilon})$ converges to $\tau(p\log p)$, resp. $\tau(e^{x_{-\varepsilon}})$ to $1$, for $\varepsilon\longrightarrow\infty$.
\end{proof}

Let $\partial$ be a symmetric gradient for $(A,\tau)$ and $\Delta:=\partial^{*}\partial$ its Laplacian. We examine an interaction between the noncommutative relative entropy and heat semigroup $P_{t}:=e^{-t\Delta}$. For our following statements, we require the heat semigroup to regularise elements in $M_{+}$ sufficiently well. The derivative in the upcoming definition is the Fr\'echet derivative w.r.t.~the $||.||_{M}$-topology.

\begin{dfn}
Set $D_{Fr}(\Delta):=\{x\in D(\Delta)\ |\ \frac{d}{dt}_{|t=0}P_{t}(x)\ \textrm{exists} \}$. We call $P_{t}$ regularity improving if for all $x\in M_{+}$ and all $t\in (0,1]$, we have $P_{t}(x)\in GL(M)\cap D_{Fr}(\Delta)$.
\end{dfn}

\noindent By the semigroup property, $x\in D_{Fr}(\Delta)$ if and only if $t\longmapsto P_{t+s}(x)$ is Fr\'echet differentiable at the origin for each $s\in (0,1)$. An example from commutative geometry is the heat semigroup on a compact Riemannian manifold. Uniform convergence of the heat kernel yields the required property, see Chapter 8 of \cite{GrigHeatKernel}. In the finite-dimensional case, $\Delta$ having one-dimensional kernel implies the heat semigroup to be regularity improving. We will prove this at the end of this Section 2.3. The next lemma shows the logarithm's quantum derivative appearing when differentiating the relative entropy evaluated at the heat semigroup.

\begin{lem}\label{LEM.Entrp_Mlt}
Let $p\in\mathcal{D}_{b}$ and set $\rho_{t}:=P_{t}(p)$. If $P_{t}$ is regularity improving, then

\begin{align*}
\frac{d}{dt}_{|t=s}Ent(\rho_{t}\hspace{0.05cm}|\hspace{0.05cm}\tau)=-\langle \big((L_{\rho_{s}}\otimes R_{\rho_{s}})(D\log)\big)(\partial \rho_{s}),\partial\rho_{s}\rangle_{H}
\end{align*}

\noindent for each $s\in (0,1)$. 
\end{lem}
\begin{proof}
For all $t\in (0,1]$, we have $\frac{d}{dt}\rho_{t}=-\Delta\rho_{t}$ and we know the limit on the left-hand side to lie in $M$. Moroever, $\rho_{t}$ is a bounded \textit{density} for each $t\in [0,1]$ since $P_{t}$ preserves mass by unitality of $A$. Finally, $\rho_{t}$ is invertible for each $t\in (0,1]$. Taken together, this implies $\log(\rho_{t})$ to be Fr\'echet differentiable on $(0,1)$. Application of the chain rule allows us to express its derivative using Pedersen's differential calculus.\par
We calculate

\begin{align*}
\frac{d}{dt}_{|t=s}\textrm{Ent}(\rho_{t}\hspace{0.05cm}|\hspace{0.05cm}\tau)=-\tau(\Delta\hspace{0.05cm}\rho_{s}\log \rho_{s})+\tau(\rho_{s}d\log_{\rho_{s}}(-\Delta\hspace{0.05cm}\rho_{s}))
\end{align*}

\noindent where we used Pedersen's chain rule and the derivative of a bounded bilinear map. The second summand equals $-\tau(\Delta\hspace{0.05cm}\rho_{s})$ by the second statement of Proposition \ref{PRP.dlog} and Remark \ref{REM.dlog}. Once more, $\tau(\Delta\rho_{s})=0$ as $1_{A}\in\ker\partial$ by unitality. We obtain

\begin{align*}
\frac{d}{dt}_{|t=s}\textrm{Ent}(\rho_{t}\hspace{0.05cm}|\hspace{0.05cm}\tau)=-\tau(\Delta\hspace{0.05cm}\rho_{s}\log\rho_{s})
\end{align*}

\noindent for each $s\in (0,1)$. We have $\tau(\Delta\hspace{0.05cm}\rho_{s}\log\rho_{s})=\langle\Delta\hspace{0.05cm}p_{s},\log\rho_{s}\rangle_{L^{2}(A,\tau)}=\langle \partial\rho_{s},\partial\log\rho_{s}\rangle_{H}$. From this and $\spec(\rho_{s})\subset\mathbb{R}_{>0}$ being bounded from below by invertibility of $\rho_{s}$, the noncommutative chain rule for symmetric gradients shows 

\begin{align*}
\frac{d}{dt}_{|t=s}\textrm{Ent}(\rho_{t}\hspace{0.05cm}|\hspace{0.05cm}\tau)=-\langle\partial\rho_{s},\big((L_{\rho_{s}}\otimes R_{\rho_{s}})(D\log)\big)(\partial\rho_{s})\rangle_{H}.
\end{align*}

\noindent $D\log$ is real-valued, hence $(L_{\rho_{s}}\otimes R_{p_{s}})(D\log)$ is self-adjoint. The statement follows by shifting the operator to the left-hand side of the inner product.
\end{proof}

\begin{rem}
If $H$ is commutative, acting by $(L_{\rho_{s}}\otimes R_{\rho_{s}})(D\log)$ reduces to multiplication by $\rho_{s}^{-1}$.
\end{rem}

If $P_{t}$ is regularity improving and $\rho_{t}$ as above, $\rho_{t}$ should not only solve a noncommutative equivalent of the continuity equation but have finite energy. Under this condition, Lemma \ref{LEM.Entrp_Mlt} shows how noncommutativity leads us to replace $\rho_{t}^{-1}$ by $(L_{\rho_{s}}\otimes R_{\rho_{s}})(D\log)$. We seek to generalise multiplication by $\rho_{s}$, thus we consider the inverse of $(L_{p_{s}}\otimes R_{\rho_{s}})(D\log)$. As $\rho_{s}$ is invertible for $s>0$, this implies

\begin{align*}
(L_{p_{s}}\otimes R_{\rho_{s}})(D\log)^{-1}=(L_{p_{s}}\otimes R_{\rho_{s}})(D\log^{-1})=(L_{p_{s}}\otimes R_{\rho_{s}})(M_{lm})
\end{align*}

\noindent Here, $M_{lm}:=D\log^{-1}$ is the logarithmic mean. It is defined on all of $\mathbb{R}_{\geq 0}^{2}$, hence we obtained a candidate for a multiplication operator even if $p$ is not invertible. This presents our starting point for defining the noncommutative $L^{2}$-Wasserstein distance.

\subsection{Energy functional and definition for bounded densities}

For the remainder of this section, let $(A,\tau)$ be a tracial $C^{*}$-algebra and set $M:=L^{\infty}(A,\tau)$. We demand the actions of $A$ on $H$ to extend to bounded actions of $M$. Such actions will be classified as part of future work by Wirth. For direct summands of $L^{2}(A,\tau)$, the canonical $M$-actions clearly extend those of $A$.\par
We define a multiplication operator given an element in $M_{+}$. The logarithmic mean $M_{lm}$ is continuous on all of $\mathbb{R}_{\geq 0}$, vanishing at the boundary. For motivation, we refer to the previous subsection.

\begin{dfn}
If $x\in M_{+}$, then $M_{x}:=(L_{x}\otimes R_{x})(M_{lm})$ is the multiplication operator for $x$. 
\end{dfn}

\begin{bsp}
If $H=L^{2}(A,\tau)$, then $M_{x}(h)=\int_{0}^{1}x^{\alpha}hx^{1-\alpha}d\alpha$ by Proposition \ref{PRP.dlog}. This extends to direct sums of $L^{2}(A,\tau)$, as well as appropriate submodules defined in Subsection 2.4.
\end{bsp}

\noindent Since $M_{lm}$ is positive, we have $M_{x}\in B(H)_{+}$ in general. Furthermore, $M_{x}$ is invertible if $x$ is and for all $x\in M_{+}\cap GL(M)$, we have

\begin{align*}
M_{x}=(L_{x}\otimes R_{x})(M_{lm})=(L_{x}\otimes R_{x})(D\log)^{-1}
\end{align*}

\noindent because $L_{x}\otimes R_{x}$ is an algebra homomorphism. If we are in the situation of Lemma \ref{LEM.Entrp_Mlt} and set $v_{s}:=(L_{\rho_{s}}\otimes R_{\rho_{s}})(D\log)(\partial\rho_{s})$, the same lemma implies

\begin{align*}
-\frac{d}{dt}_{|t=s}\textrm{Ent}(\rho_{t}\hspace{0.05cm}|\hspace{0.05cm}\tau)=\langle v_{s},M_{\rho_{s}}v_{s}\rangle_{H}=||M_{\rho_{s}}^\frac{1}{2}v_{s}||_{H}^{2}.
\end{align*}

\noindent We view $||M_{p}^{\frac{1}{2}}h||_{H}$ as our analogue of the tangent space norm. Before defining admissible paths and the energy functional, we prove a crucial statement bounding the norm of $M_{x}\in B(H)$ by that of $x\in M$. 

\begin{prp}\label{PRP.Nrm_Mlt_Op}
For $x\in M_{+}$, we have $||M_{x}||_{B(H)}\leq ||x||_{M}$. Equality holds if $\pi$ is faithful.
\end{prp}
\begin{proof}
For $C>0$, consider $f(s,t):=M_{lm}(s,t)$ on $[0,C]\times [0,C]$. We claim $||f||_{\infty}=C$. To see this, first observe $f(C,C)=C$ since $D\log(C,C)=C^{-1}$. For the converse, note how $f(s,0)=0$ for each $s\in [0,\infty)$. Thus finding and comparing maxima of the differentiable functions $f_{s}(t):=f(s,t)$ on $(0,C)$ for each fix $s\in (0,C)$ is sufficient for our purposes. A calculation shows 

\begin{align*}
\frac{d}{dt}f_{s}(t)=0 \iff t=f_{s}(t)
\end{align*}

\noindent to hold. Since $s,t>0$, the right hand side is equivalent to $\log(\frac{s}{t})=\frac{s}{t}-1$. Yet, $\log(x)=x-1$ implies $x=1$ as the functions $\log(x)$ and $x-1$ intersect tangentially while $\log$ is concave. Hence $f_{s}$ has an extrema on $(0,C)$ if any only if $s=t$. Knowing $s<C$, this implies $f_{s}(s)=f(s,s)=s>0=f_{s}(0)$ to be an extreme point. It must therefore be a global maximum as $f_{s}$ is continuous on $[0,C]$. If $s=C$, there is no extreme point on $(0,C)$ but $f_{C}(C)=C>0=f_{C}(0)$ still holds. Thus the global maximum of $f$ is given by $C$.\par
The claim is trivial for $x=0$, hence let $x\in M_{+}$ be non-zero. As $M_{x}$ is given by the image of $f_{|\spec(x)\times\spec(x)}$ under $\pi$, we have 

\begin{align*}
||M_{x}||_{B(H)}\leq\sup_{s,t\in\spec(x)}|f(s,t)|\leq ||x||_{M}
\end{align*}

\noindent where the first inequality stems from $\pi$ being a $^{*}$-homomorphism of $C^{*}$-algebras. We used the first part of this proof to obtain the second estimate. However, $||x||\in\spec(x)$ holds by self-adjointness of $x$. The statement follows from $f(s,s)=s$. 
\end{proof}

\begin{dfn}\label{DFN.Dst}
Let $\mathcal{D}:=\{p\in L_{+}^{1}(A,\tau)\ |\ \tau(p)=1\}$ and $\mathcal{D}_{b}:=\{p\in\mathcal{D}\ |\ p\in L^{\infty}(A,\tau)\}$ be the space of densities, resp. the space of bounded densities.
\end{dfn}

\begin{rem}
We defined $\mathcal{D}_{b}$ before in Definition \ref{DFN.Bd_Dst_Untl}. It was merely a question of exposition.
\end{rem}

We are ready to define tangent spaces, admissible paths, the energy functional and finally the $L^{2}$-Wasserstein distance. In the following, assume $\partial$ to be a symmetric gradient for $(A,\tau)$. 

\begin{dfn}
For all $p\in\mathcal{D}_{b}$ and all $a,b\in A_{\partial}$, set

\begin{align*}
\langle a,b\rangle_{p}:=\langle M_{p}\partial a,\partial b\rangle_{H}.
\end{align*}
\end{dfn}

\begin{rem}
Each $\langle \ ,\hspace{0.05cm} \rangle_{p}$ is a semi-definite, positive bilinear form on $A_{\partial}$ by positivity of $M_{p}$.
\end{rem}

\begin{dfn}
The tangent space $T_{p}\mathcal{D}_{b}$ at $p\in\mathcal{D}_{b}$ is defined to be the Hausdorff completion of $A_{\partial}$ w.r.t.~$\langle \ ,\hspace{0.05cm} \rangle_{p}$. The tangent bundle is defined as  $T\mathcal{D}_{b}:=\underset{p\in\mathcal{D}_{b}}{\coprod}\ \{p\}\times T_{p}\mathcal{D}_{b}$.
\end{dfn}

\begin{ntn}
A path $\mu_{t}$ in $T\mathcal{D}_{b}$ splits into a pair of paths $\mu_{t}=(\rho_{t},v_{t})$ with unique $\rho_{t}\in\mathcal{D}_{b}$ and $v_{t}\in T_{\rho_{t}}\mathcal{D}_{b}$. We always use this or analogous notation when decomposing a path in the tangent bundle.
\end{ntn}

\begin{dfn}\label{DFN.W2_Bd}
Let $\mu_{t}:[0,1]\longrightarrow T\mathcal{D}_{b}$ such that $t\longmapsto\tau(\rho_{t}a)$ is absolutely continuous for each $a\in A_{\partial}$. We say that $\mu_{t}$ satisfies the noncommutative continuity equation if 

\begin{align*}
\frac{d}{dt}\tau(\rho_{t}a)=\langle v_{t},a\rangle_{\rho_{t}}
\end{align*}

\noindent for each $a\in A_{\partial}$ and a.e.~$t\in [0,1]$. 
\end{dfn}

\begin{ntn}
We drop the adjective ''noncommutative'' in the future.
\end{ntn}

We are able to represent any $v\in T_{p}\mathcal{D}_{b}$ in $H$. Given $v$, choose a sequence of $a_{i}\in A_{\partial}$ converging to $v$. From this, we obtain

\begin{align*}
M_{p}^{\frac{1}{2}}\partial a_{i}\longrightarrow w
\end{align*}

\noindent in $H$. In the above, $w\in H$ is independent of our choice of $a_{i}$ by definition of the inner product. This defines a bounded linear map from $(T_{p}\mathcal{D}_{b},||.||_{p})$ to $H$, sending $v$ to $w$. It is an isometry by construction. In particular, the image of $T_{p}\mathcal{D}_{b}$ in $H$ is closed. We thereby view each $T_{p}\mathcal{D}_{b}$ as a closed subspace of $H$, and $T_{p}D_{b}$ as a subspace of $D_{b}\times H$. Using this, we rewrite the continutiy equation as

\begin{align*}
\frac{d}{dt}\tau(\rho_{t}a)=\langle w_{t},M_{\rho_{t}}^{\frac{1}{2}}\partial a\rangle_{H}.
\end{align*}

\begin{ntn}
For a given path $\mu_{t}$ satisfying the continuity equation, we consider $v_{t}$ and $w_{t}$ interchangably from now on. Furthermore, we denote the projection from $H$ to $T_{\rho_{t}}\mathcal{D}_{b}$ by $R_{t}$.
\end{ntn}

\begin{dfn}\label{DFN.Bd_Adm}
Let $p,q\in\mathcal{D}_{b}$. An admissible path from $p$ to $q$ is a $\mu_{t}:[0,1]\longrightarrow T\mathcal{D}_{b}$ such that 

\begin{itemize}
\item[1)] $\mu_{t}$ satisfies the continuity equation,
\item[2)] $\rho_{0}=p$ and $\rho_{1}=q$,
\item[3)] $t\longmapsto ||v_{t}||_{\rho_{t}}^{2}=||w_{t}||_{H}^{2}\in L^{1}([0,1])$.
\end{itemize}

\noindent We denote the set of all admissible paths between $p$ and $q$ by $\mathcal{A}(p,q)$.
\end{dfn}

\noindent Let $\varphi$ be a linear reparametrisation and $\mu_{t}$ satisfy the continuity equation. We decompose $\mu_{\varphi(t)}$ into $\mu_{\varphi(t)}=(\rho_{\varphi(t)},v_{\varphi(t)}\dot{\varphi
}(t))$. Hence precomposition by $t\longmapsto -t$ maps admissible paths to admissible paths. The decomposition additionally shows that concatenating two admissible paths, in the canonical topological sense, again yields an admissible path. From this we obtain symmetry, resp.~the triangle-inequality for our distance candidate once we have defined the latter. 

\begin{dfn}\label{DFN.Bd_L2WDst}
We define the energy functional on admissible paths as

\begin{align*}
E(\mu_{t}):=\frac{1}{2}\int_{0}^{1}||v_{t}||_{\rho_{t}}^{2}dt
\end{align*}

\noindent and the noncommutative $L^{2}$-Wasserstein distance on bounded densities by

\begin{align*}
\mathcal{W}_{2}(p,q)=\inf_{\mu_{t}\in\mathcal{A}(p,q)}\sqrt{E(\mu_{t})}.
\end{align*}
\end{dfn}

\begin{ntn}
As before, we drop ''noncommutative'' in the above description.
\end{ntn}

We prove $\mathcal{W}_{2}$ to be a distance. By the discussion just prior to Definition \ref{DFN.Bd_L2WDst} and $E\geq 0$, we only need to check definiteness. To do so, we assume existence of a function $g$ allowing control of $\langle M_{p}\partial a,\partial a\rangle_{H}$ on a sufficiently large subset $S\subset A_{\partial}$.

\begin{dfn}\label{DFN.Spfct}
Let $S\subset A_{\partial}$ and $g:S\longrightarrow\mathbb{R}_{\geq 0}$ such that for all $p,q\in\mathcal{D}_{b}$, we have

\begin{itemize}
\item[1)] $\tau(pa)=\tau(qa)$ for each $a\in S$ if and only if $p=q$,
\item[2)] $||a||_{p}^{2}\leq g(a)$ for each $a\in S$.
\end{itemize}

\noindent Then $g$ is called a separating function.
\end{dfn}

\begin{prp}
If there exists a separating function $g$, then $\mathcal{W}_{2}$ is a distance.
\end{prp}
\begin{proof}
We only need to show definiteness. For all admissible paths $\mu_{t}$ and all $a\in S$, we have

\begin{align*}
\tau((\rho_{1}-\rho_{0})a)&=\int_{0}^{1}\frac{d}{dt}\tau(\rho_{t}a)dt\\ 
&=\int_{0}^{1}\langle v_{t},a\rangle_{\rho_{t}}dt\\
&\leq \sqrt{2g(a)E(\mu_{t})}
\end{align*}

\noindent where we used $2)$ in Definition \ref{DFN.Spfct} for the last estimate. By $1)$ in the same definition and construction of $\mathcal{W}_ {2}$, $\mathcal{W}_{2}(p,q)=0$ if and only if $p=q$.
\end{proof}

\begin{ntn}
Distances can be infinite in metric geometry. While we use this convention, distances with infinite value are also called extended distances, or extended metrics.
\end{ntn}

\begin{rem}
Our definition is compatible with the commutative case if the underlying metric measure space $(X,d,m)$ satisfies the reduced curvature-dimension condition $CD^{*}(K,N)$, see 2) of Theorem 1.2 in \cite{RajLInfBound}. Other examples are measured-length spaces, defined in \cite{GiHaContMetrBdDens}. We will see in Example \ref{BSP.Unbd_Comm_Case} how to recover the $C^{\infty}$-manifold setting in general after having extended to all densities for particular gradients in Subsection 2.4.
\end{rem}

\begin{bsp}\label{BSP.Bd_Comm_Case}
Let $(X,h)$ be a smooth Riemannian manifold with density $d|\omega|$ and connection $\nabla$. Set $A=C_{0}(X)$, $\tau=d|\omega|\otimes\mathbb{C}$, $\partial:=\nabla\otimes\mathbb{C}$ and $H$ to be the space of $L^{2}$-sections of $TX\otimes\mathbb{C}$ w.r.t.~$hd|\omega|$. If $S:=C_{c}^{\infty}(X)$, we have

\begin{align*}
||a||_{p}^{2}=\int_{X}ph(\partial a,\partial a)d|\omega|\leq ||h(\partial a,\partial a)||_{\infty}
\end{align*}

\noindent for each $p\in D_{b}$ and $a\in S$. Hence $g(a):=||h(\partial a,\partial a)||_{\infty}$ is a separating function.
\end{bsp}

\begin{bsp}\label{BSP.SF_Fct_I}
Let $H$ be separable. Consider $A=\mathcal{K}(H)$ and $\tau=\theta\textrm{tr}$ for $\theta>0$ fix. Then $\theta^{-1}L^{1}(A,\tau)$ equals $\mathcal{S}_{1}(H)$, thus $||x||_{M}\leq\theta^{-1}||x||_{L^{1}(A,\tau)}$ for each $x\in L^{1}(A,\tau)$. For $S=A_{\partial}$, we have

\begin{align*}
\langle M_{p}\partial a,\partial a\rangle_{H}\leq\theta^{-1}||\partial a||_{H}^{2}
\end{align*}

\noindent by Proposition \ref{PRP.Nrm_Mlt_Op} and $p\in\mathcal{D}_{b}$. Hence $g(a):=\theta^{-1}||\partial a||_{H}^{2}$ is a separating function. The fourth section deals with a wide generalisation of this example.
\end{bsp}

\begin{bsp}
All symmetric gradients of type considered in Subsection 2.4 have a canonical separating function, see Proposition \ref{PRP.Unbd_Sep}.
\end{bsp}

We end this subsection with a lemma useful when discussing vertical gradients.

\begin{lem}\label{LEM.Msrbl_Bd}
Assume there exists a separating function and let $\mu_{t}$ be an admissible path. If $L^{1}(A,\tau)$ is separable, then $\rho_{t}\in L^{1}([0,1],L^{1}(A,\tau))$.
\end{lem}
\begin{proof}
$A_{\partial}$ lies dense in $A$, hence dense in $M$ w.r.t.~the $w^{*}$-operator topology. We already know $\tau(\rho_{t}a)\in C([0,1])$ for each $a\in A_{\partial}$ by hypothesis. If on the other hand $a_{i}\in A_{\partial}$ converges to $x\in M$ in the $w^{*}$-topology, we know that $\tau(\rho_{t}a_{i})$ converges to $\tau(\rho_{t}x)$. Thus $\tau(\rho_{t}x)$ is approximated pointwise by measurable functions $\tau(\rho_{t}a_{i})$, where $x\in M$ was arbitrary but fix. Since $L^{1}(A,\tau)$ was separable and $L^{1}(A,\tau)^{*}=M$, Pettis' theorem shows strong measurability of $\rho_{t}$. Thus Bochner-integrability follows from $||\rho_{t}||_{L^{1}(A,\tau)}=1$. 
\end{proof}

\subsection{Finiteness on bounded densities for unital $C^{*}$-algebras}

For this subsection, let $\partial$ be a symmetric gradient for $(A,\tau)$ and assume existence of a separating function $g$. We show finiteness of $\mathcal{W}_{2}$ if $A$ is unital, $\partial$ satisfies a Poincar\'e-type inequality and the heat semigroup $P_{t}:=e^{-t\Delta}$ is regularity improving. For the latter, we show ergodicity to be a necessary condition.

\begin{dfn}
We say that $\partial$ satisfies a Poincar\'e-type inequality if there exists some $C>0$ such that $||a||_{L^{2}(A,\tau)}\leq C||\partial a||_{H}$ for each $a\in(\ker\partial)^{\bot}\cap A_{\partial}$.
\end{dfn}

\begin{prp}\label{PRP.Pncr_Inq}
Let $\partial$ satisfy a Poincar\'e-type inequality. For all $x\in M\cap L_{sa}^{2}(A,\tau)\cap\ker\tau$, there exists an $h\in H$ such that $\tau(xa)=\langle h,\partial a\rangle_{H}$ for each $a\in A_{\partial}$.
\end{prp}
\begin{proof}
This is proved in Theorem 9.2. of \cite{ZaevTopicsNCSpaces} for general $x\in L^{2}(A,\tau)$. 
\end{proof}

To show finiteness, we first prove that a Poincar\'e-type inequality suffices to have finite distance between invertible elements. After this, we use the regularity improving property of the heat semigroup to connect non-invertible elements to invertible ones. Finite energy of these paths will follow from Lemma \ref{LEM.Entrp_Mlt}.

\begin{thm}\label{THM.Fin_Dst}
If $\partial$ satisfies a Poincar\'e-type inequality, the distance between any two invertible bounded densities is finite. If $A$ is unital, $p\in\mathcal{D}_{b}$ and $P_{t}$ regularity improving, the distance between $p$ and $\rho_{t}:=P_{t}(p)$ is finite for each $t\in [0,1]$.
\end{thm}
\begin{proof}
We begin with the first statement. Thus let $p,q$ be invertible bounded densities and set $C:=\min\{\inf\spec (p),\inf\spec (q)\}$. We have $C>0$ because $p$ and $q$ are invertible in $M$. Writing $\rho_{t}:=(1-t)p+tq$, we have $C\langle x,x\rangle_{L^{2}(A,\tau)}\leq \langle\rho_{t}x,x\rangle_{L^{2}(A,\tau)}$ for each $x\in L^{2}(A,\tau)$. Hence $\rho_{t}$ is invertible for each $t\in [0,1]$. As $\partial$ satisfies a Poincar\'e-type inequality, Proposition \ref{PRP.Pncr_Inq} allows us to choose an $h\in H$ such that 

\begin{align*}
\tau((q-p)a)=\langle h,\partial a\rangle_{H}=\langle M_{\rho_{t}}^{-\frac{1}{2}}h,M_{\rho_{t}}^{\frac{1}{2}}\partial a\rangle_{H}
\end{align*}

\noindent for each $a\in A_{\partial}$. By Proposition \ref{PRP.LR_Rpr}, $
||M_{\rho_{t}}^{-1}||_{B(H)}\leq ||D\log||_{C([C,||\rho_{t}||_{M}]\times [C,||\rho_{t}||_{M}])}$ with the right-hand term bounded on $[0,1]$ by continuity of $\rho_{t}$. Hence $a\longmapsto \tau((q-p)a)=\tau(\dot{\rho}_{t}a)$ are bounded linear functionals on $T_{\rho_{t}}\mathcal{D}$ for each $t\in [0,1]$, represented by a unique $v_{t}\in T_{\rho_{t}}\mathcal{D}$. As an element in $H$, $v_{t}$ is given by 

\begin{align*}
w_{t}=R_{t}(M_{\rho_{t}}^{-\frac{1}{2}}h).
\end{align*}

\noindent $M_{p_{t}}^{\frac{1}{2}}h$ is continuous by the $||.||_{M}$-continuity of $p_{t}$, and $R_{t}$ a projection for each $t\in [0,1]$. Thus $w_{t}$ is strongly measurable in $H$, and $||w_{t}||^{2}$ lies in $L^{1}([0,1])$. Hence $\mu_{t}:=(\rho_{t},v_{t})$ is an admissible path from $p$ to $q$. Since $p$ and $q$ were arbitrary, the first statement follows.\par
For the second statement, let $p\in\mathcal{D}_{b}$ and note that we now assume $A$ to be unital. Without loss of generality, we norm $\tau$ to one. By Proposition \ref{PRP.Entrp_Rpr}, $p$ has finite relative entropy. Since $P_{t}$ is regularity improving, $\rho_{t}:=P_{t}(p)$ is an invertible bounded density for each $t\in (0,1]$. To see

\begin{align*}
-\frac{d}{dt}\tau(\rho_{t}a)=\tau(\Delta\rho_{t}a)=\langle \partial\rho_{t},\partial a\rangle_{H}=\langle M_{\rho_{t}}^{\frac{1}{2}}\partial\log \rho_{t},M_{\rho_{t}}^{\frac{1}{2}}\partial a\rangle_{H}
\end{align*} 

\noindent we expand by $M_{\rho_{t}}^{-1}$ and apply the noncommutative chain rule as in the proof of Lemma \ref{LEM.Entrp_Mlt}. Analogous to the first statement's proof, this induces a bounded linear functional represented by some $v_{t}$, for each $t\in (0,1]$. In $H$, $v_{t}$ is given by

\begin{align*}
w_{t}=-M_{\rho_{t}}^{\frac{1}{2}}\partial\log \rho_{t}.
\end{align*}

\noindent which gives a vector field on $[0,1]$. Frech\'et differentiability of $\rho_{t}$ on $(0,1)$ implies $||.||_{M}$-continuity of $\rho_{t}$, thus $w_{t}$ is strongly measurable in $H$ as $\partial$ is linear. To show $||w_{t}||^{2}\in L^{1}([0,1])$, observe that $||R_{t}||=1$ implies

\begin{align*}
||v_{t}||_{\rho_{t}}\leq ||M_{\rho_{t}}^{\frac{1}{2}}\partial\log \rho_{t}||_{H}
\end{align*}

\noindent for each $t\in (0,1]$. Using this, we estimate

\begin{align*}
\int_{0}^{1}||v_{t}||_{\rho_{t}}^{2}dt \leq \int_{0}^{1}||M_{\rho_{t}}^{\frac{1}{2}}\partial\log \rho_{t}||_{H}^{2}dt=\textrm{Ent}(p|\tau)-\textrm{Ent}(\rho(1)|\tau)<\infty
\end{align*}

\noindent We applied Lemma \ref{LEM.Entrp_Mlt} for the last equality. It follows that $\mu_{t}:=(\rho_{t},v_{t})$ is an admissible path.
\end{proof}

\begin{cor}
Let $A$ be unital. If $\partial$ satisfies a Poincar\'e-type inequality and has regularity improving heat semigroup, $\mathcal{W}_{2}$ is finite.
\end{cor}

To end this subsection, we provide a necessary condition for $P_{t}$ to be regularity improving. In this, we lift Simon's original proof \cite{SiPosImpr} to the noncommutative setting. We make use of notations and results immediately leading up to and found on p.~204-205 in \cite{CiDrchltFrmsNCS}. 

\begin{lem}\label{LEM.L2_Pst_Prj}
If $x\in L^{2}(A,\tau)$ is self-adjoint, then $\max\{x,0\}$ is given by the metric projection $x_{+}$ of $x$ onto the self-polar cone $L_{+}^{2}(A,\tau)$ in $L^{2}(A,\tau)$. 
\end{lem}
\begin{proof}
We know $\max\{x,0\}\in L_{+}^{2}(A,\tau)$ by construction of the positive elements. A general metric projection $P_{C}$ onto a closed convex set $C$ in a Hilbert space can be characterised uniquely by satisfying $\textrm{Re}\langle x-P_{C}(x),y-P_{C}(x)\rangle_{H}\leq 0$ for each $y\in C$. A calculation in our setting using any $y\geq 0$ yields

\begin{align*}
\tau((x-\max\{x,0\})(y-\max\{x,0\}))&=\tau((-\min\{x,0\})^{\frac{1}{2}}(\max\{x,0\}-y)(-\min\{x,0\})^{\frac{1}{2}})\\
&\leq \tau((-\min\{x,0\})^{\frac{1}{2}}\max\{x,0\}(-\min\{x,0\})^{\frac{1}{2}})\\
&=-\tau(\min\{x,0\}\max\{x,0\})\\
&=0.
\end{align*}
\end{proof}

\begin{rem}
Using the above characterisation of the metric projection to show the result was pointed out to the author by Wirth in a personal communication as a derivative of a lemma in future work of his.
\end{rem}

\noindent We turn to a second lemma that closely orients itself along Lemma 3.4 of Simon's proof.

\begin{lem}\label{LEM.Sim_NC}
Let $T$ be a positivity preserving operator on $L^{2}(A,\tau)$. If $x,y\in L_{+}^{2}(A,\tau)$ with $\langle x,y\rangle_{L^{2}(A,\tau)}\neq 0$, then $\langle Tx,Ty\rangle_{L^{2}(A,\tau)}\neq 0$.
\end{lem}
\begin{proof}
$L^{2}(A,\tau)=L^{2}(M,\tau)$ is a standard form of $M$ with cyclic vector $1_{A}$. We are thus able to use analogues of the pointwise supremum and infimum operations. Assume $x\wedge y=0$, where $x\wedge y$ is our analogue of the infimum of $x$ and $y$.\par
Point five of Lemma 2.50 in \cite{CiDrchltFrmsNCS} yields $x+y=|x-y|$. The fifth and third points of the same lemma together give $|x-y|=y+(x-y)_{+}$. From this, $x+y=y+(x-y)_{+}$ follows by Lemma \ref{LEM.L2_Pst_Prj} above. All in all, $x=\max\{x-y,0\}\in L_{+}^{2}(A,\tau)$ holds. Evoking Lemma 2.50 one last time, we have $y=\min\{x-y,0\}$ and thus $xy=yx=0$. This shows $x\wedge y\neq 0$ for $x,y\geq 0$. By definition, $x\wedge y\leq x,y$ holds for positive $x$ and $y$. From here on out, we nearly proceed verbatim as Simon did in the first lemma of \cite{SiPosImpr}. We only need to replace the minimum of $x$ and $y$ by $x\wedge y$.
\end{proof}

\begin{dfn}
A positive semigroup $e^{tL}$ is ergodic if for all $x,y\in L^{2}_{+}(A,\tau)$ with $x,y\neq 0$, there exists a $t>0$ such that $\tau(xe^{tL}y)>0$. A semigroup $e^{tL}$ on $L^{2}(A,\tau)$ is called positivity improving if $e^{tL}x$ has strictly positive spectrum for each $x\in L_{+}^{2}(A,\tau)$ and each $t\in (0,\infty]$.
\end{dfn}

\begin{rem}
An operator $T$ has strictly positive spectrum if $\spec (T)\subset\mathbb{R}_{>0}$. It does not imply existence of a uniform lower bound. We follow the commutative terminology in this, where a function $f$ is strictly positive if $f>0$ almost everywhere.
\end{rem}

\begin{thm}\label{THM.Erg}
If $e^{tL}$ is a semigroup of self-adjoint, positivity preserving operators on $L^{2}(A,\tau)$, it is positivity improving if and only if it is ergodic.
\end{thm}
\begin{proof}
After replacing Simon's first lemma with Lemma \ref{LEM.Sim_NC}, the proof is given verbatim to the one of Theorem 1 in \cite{SiPosImpr}.
\end{proof}

In \cite{CiDrchltFrmsNCS}, Cipriani provides necessary and sufficient conditions for ergodicity of a semigroup. If $A$ is unital, Corollary 2.48 in \cite{CiDrchltFrmsNCS} implies the heat semigroup to be ergodic if and only if $1_{A}$ is a \textit{simple} eigenvector of $\Delta$. 

\begin{cor}\label{COR.Erg}
If $A$ is unital, then $P_{t}$ is positivity improving if and only if $1_{A}$ is a simple eigenvector of $\Delta$.
\end{cor}
\begin{proof}
In the notation of \cite{CiDrchltFrmsNCS}, $(M,L^{2}(A,\tau),L_{+}^{2}(A,\tau),^{*})$ is a standard form of $M$ with cyclic vector $1_{A}$. Applying the equivalence between the first and third statement of Corollary 2.48 in \cite{CiDrchltFrmsNCS}, as well as Theorem \ref{THM.Erg} above, we obtain the statement.
\end{proof}

\begin{bsp}\label{BSP.FinDim_RgImprSG}
Let $A$ be a finite-dimensional $C^{*}$-algebra. If $\partial$ has one-dimensional kernel, so does $\Delta$. Moreover, $\partial$ satisfies a Poincar\'e-type inequality since $\Delta$ becomes a positive operator on the orthogonal complement of the kernel. Then Corollary \ref{COR.Erg} implies $P_{t}$ to be positivity improving. Since all relevant operator topologies are equivalent and a strictly positive spectrum implies invertibility of the operator in finite-dimensions, $P_{t}$ is regularity improving.
\end{bsp}		

\subsection{Extending to unbounded densities}

So far, we required densities to be bounded. We now extend $\mathcal{W}_{2}$ to all densities. To do so, we impose conditions on the domain and codomain of the gradient. In particular, we consider multiplication operators given by 

\begin{align*}
M_{p}(x)=\int_{0}^{1}p^{\alpha}xp^{1-\alpha}d\alpha
\end{align*}

\noindent on summands. For this, $\partial$ will have to take values in some $\bigoplus_{k=1}^{m}L^{2}(A,\tau)$ equipped with the canonical symmetric $L^{\infty}(A,\tau)$-bimodule structure and Hilbert space norm. All $L^{\infty}(A,\tau)$-subbimodules $H\subset\bigoplus_{k=1}^{m}L^{2}(A,\tau)$ are assumed to be closed subspaces throughout the paper. Furthermore, we assume $L^{1}(A,\tau)$ and $L^{2}(A,\tau)$ to be separable in this subsection.

\begin{dfn}\label{DFN.Can_Set}
If $H\subset\bigoplus_{k=1}^{m}L^{2}(A,\tau)$ is an $L^{\infty}(A,\tau)$-bimodule closed under adjoining of operators, we call $H$ a symmetric Hilbert $L^{\infty}(A,\tau)$-subbimodule. 
\end{dfn}

\begin{rem}
Note the important assumptions made at the beginning of this subsection.
\end{rem}

\noindent For the remainder of the subsection, let $\partial$ map into a symmetric Hilbert $L^{\infty}(A,\tau)$-subbimodule $H$. Morally, we view $H$ as a module of $L^{2}$-sections embedded in the $L^{2}$-sections of the trivial $m$-bundle over the space $A$ models. We do not assume $H$ to be a finitely generated, projective module.

\begin{ntn}\label{NTN.Dcp_Grd_Std}
As $\partial$ maps into $\bigoplus_{k=1}^{m}L^{2}(A,\tau)$, we view each $\partial_{k}$ as a symmetric gradient in itself.
\end{ntn}

We will have to replace $A_{\partial}$ by more suitable $^{*}$-subalgebra $\mathfrak{A}$. One can think of $\mathfrak{A}$ as playing a r\^ole similar to that of smooth functions with compact support, but the analogy is not too strict.

\begin{dfn}
Let $\mathfrak{A}\subset A_{\partial}$ be a dense $^{*}$-subalgebra of $A$ such that it is again a core for $\partial$. If furthermore $\partial a\in\bigoplus_{k=1}^{m}\big(L^{2}(A,\tau)\cap L^{\infty}(A,\tau)\big)$ for each $a\in\mathfrak{A}$, we call $\mathfrak{A}$ an extension algebra.
\end{dfn}

\begin{rem}
In the above definition, $L^{2}(A,\tau)\cap L^{\infty}(A,\tau)$ is viewed as an $L^{\infty}(A,\tau)$-bimodule in the algebraic sense. No topology is being considered.
\end{rem}

\begin{bsp}\label{BSP.Unbd_Comm_Case}
In example \ref{BSP.Bd_Comm_Case}, let $X$ be embedded isometrically into some $\mathbb{R}^{m}$. Then $TX\otimes\mathbb{C}\subset X\times\mathbb{C}^{m}$ and we set $\mathfrak{A}:=C_{c}^{\infty}(X)$ as an extension algebra. Thus we capture the smooth Riemannian setting with our formalism.  
\end{bsp}

\begin{bsp}
By Definition \ref{DFN.Vrt_Grd}, $C_{c}(X)\odot\FinRk(H)$ is an extension algebra for each vertical gradient.
\end{bsp}

\begin{bsp}
Let $(A,\mathbb{R},\alpha_{t})$ be a $C^{*}$-dynamical system such that the $^{*}$-algebra 

\begin{align*}
\mathfrak{A}:=\{x\in A\ |\ \partial(x):=\frac{d}{dt}_{|t=0}\alpha_{t}(x)\in A\cap L^{2}(A,\tau)\}
\end{align*}

\noindent lies dense in $A$ and is a core for $\partial$. If $A$ is unital and $\alpha_{t}(x)$ Fr\'echet differentiable at the origin for each $x\in A$, $\mathfrak{A}=A$ is an extension algebra.
\end{bsp}

\begin{lem}\label{LEM.Ext_Mlt}
For all $p\in L_{+}^{1}(A,\tau)$ and all $x\in L^{\infty}(A,\tau)$, we have

\begin{align*}
||p^{\alpha}xp^{1-\alpha}||_{L^{1}(A,\tau)}\leq ||p||_{\mathcal{S}_{1}(H)}||x||_{L^{\infty}(A,\tau)}
\end{align*}

\noindent and $p^{\alpha}xp^{1-\alpha}\in L^{1}([0,1],L^{1}(A,\tau))$.

\end{lem}
\begin{proof}
We show $p^{\alpha}xp^{1-\alpha}\in L^{1}(A,\tau)$ by applying the generalised H\"older inequality twice. Since $\alpha\in [0,1]$, we know $\alpha^{-1},(1-\alpha)^{-1}\in [1,\infty]$. Using H\"older for $1=\alpha+(1-\alpha)$ and $\alpha=\alpha+0$, we obtain

\begin{align*}
||p^{\alpha}xp^{1-\alpha}||_{1}&\leq ||p^{\alpha}x||_{\alpha^{-1}}||p^{1-\alpha}||_{(1-\alpha)^{-1}}\\
&\leq ||p^{\alpha}||_{\alpha^{-1}}||x||_{\infty}||p^{1-\alpha}||_{(1-\alpha)^{-1}}\\
&=\tau(p)^{\alpha}||x||_{\infty}\tau(p)^{1-\alpha}\\
&=||p||_{\mathcal{S}_{1}(H)}||x||_{\infty}.
\end{align*}

Once we know $\alpha\longmapsto p^{\alpha}xp^{1-\alpha}$ to be strongly measurable, the above yields Bochner-integrability. Measurability is clear if $p$ is bounded. Choose a strictly monotonically increasing sequence of $C_{i}\geq 1$ diverging to infinity, and set $p_{i}:=\min\{p,C_{i}\}$. Arguing by functional calculus shows $p_{i}^{\alpha}$ to approximate $p^{\alpha}$ in $L^{\alpha^{-1}}(A,\tau)$ for each $\alpha\in (0,1]$.\par
We claim that $p_{i}^{\alpha}xp_{i}^{1-\alpha}$ $||.||_{L^{1}(A,\tau)}$-converges to $p^{\alpha}xp^{1-\alpha}$ for each fixed $\alpha\in [0,1]$. To see this, we have to show convergence of $p^{\alpha}xp^{1-\alpha}-p_{i}^{\alpha}xp_{i}^{1-\alpha}$ to zero. We do so by using the triangle inequality and then applying H\"older as above to $p^{\alpha}x(p^{1-\alpha}-p_{i}^{1-\alpha})$, resp. $(p^{\alpha}-p_{i}^{\alpha})xp_{i}^{1-\alpha}$. Hence our path is a pointwise limit of strongly measurable ones, therefore strongly measurable itself. 
\end{proof}

\begin{dfn}
For all $p\in\mathcal{D}$ and $x\in\bigoplus_{k=1}^{m}L^{\infty}(A)$, we set 

\begin{align*}
M_{p}(x):=\Big(\int_{0}^{1}p^{\alpha}x_{k}p^{1-\alpha}d\alpha\Big)_{k=1}^{m}\in\bigoplus_{k=1}^{m}L^{1}(A,\tau)
\end{align*}
\end{dfn}

\begin{prp}\label{PRP.Mlt_Ctrc_Apprx}
For all $p\in\mathcal{D}$ and all $x\in\bigoplus_{k=1}^{m}L^{\infty}(A,\tau)$, the linear operator

\begin{align*}M_{p}:\bigoplus_{k=1}^{m}L^{\infty}(A,\tau)\longrightarrow\bigoplus_{k=1}^{m}L^{1}(A,\tau)
\end{align*}

\noindent is a contraction and $M_{p_{i}}(x)$ $||.||_{L^{1}(A,\tau)}$-converges to $M_{p}(x)$ if $(p_{i})_{i\in\mathbb{N}}\subset L_{+}^{\infty}(A,\tau)$ is defined as in the proof of Lemma \ref{LEM.Ext_Mlt}.
\end{prp}
\begin{proof}
This immediately follows from Lemma \ref{LEM.Ext_Mlt} above, resp. the last part of its proof.
\end{proof}

Assuming existence of an extension algebra $\mathfrak{A}$, we define the norm of a tangent space in analogy to the bounded case. Indeed, each summand of $\partial a$ lies in $L^{\infty}(A,\tau)$. Hence we are able to apply $M_{p}$ by Proposition \ref{PRP.Mlt_Ctrc_Apprx}. Furthermore, we have

\begin{align*}
\sum_{k=1}^{m}\tau((M_{p}\partial a)^{*}y)=\int_{0}^{1}\sum_{k=1}^{m}\tau(p^{1-\alpha}\partial x^{*}p^{\alpha}y)d\alpha
\end{align*}

\noindent for each $y\in\bigoplus_{k=1}^{m}L^{\infty}(A,\tau)$ by boundedness of $\tau$ on $L^{1}(A,\tau)$, as well as continuity of multiplication by $y$ from the right viewed as a linear operator on $L^{1}(A,\tau)$. For bounded $p$ and $y=\partial x$, we recover $||a||_{p}$ by construction. Approximation by $M_{p_{i}}$ shows the formula above to define a semi-definite, positive bilinear form on $\mathfrak{A}$. 

\begin{dfn}
Let $\mathfrak{A}$ be an extension algebra. For all $p\in\mathcal{D}$ and $a,b\in\mathfrak{A}$, we define

\begin{align*}
\langle a,b\rangle_{p}:=\int_{0}^{1}\sum_{k=1}^{m}\tau(p^{1-\alpha}\partial a^{*}p^{\alpha}\partial b)d\alpha
\end{align*}

\noindent and let $T_{p}\mathcal{D}$ be the Hausdorff completion of $\mathfrak{A}$ w.r.t.~$\langle \ ,\hspace{0.05cm} \rangle_{p}$. The tangent bundle is defined as before by $T\mathcal{D}:=\underset{p\in\mathcal{D}}{\coprod}\ \{p\}\times T_{p}\mathcal{D}$.
\end{dfn}

\begin{rem}
Each $T_{p}\mathcal{D}$ is a Hilbert space by construction. 
\end{rem}

We extend the rest of our relevant notions, beginning with admissible paths on all densities. Compatibility with the bounded case has to be proved since we replace $A_{\partial}$ by a potentially smaller $^{*}$-subalgebra $\mathfrak{A}$.

\begin{dfn}
Let $\mu_{t}:[0,1]\longrightarrow T\mathcal{D}$ such that $t\longmapsto\tau(\rho_{t}a)$ is absolutely continuous for each $a\in\mathfrak{A}$. We say that $\mu_{t}$ satisfies the (noncommutative) continuity equation if 

\begin{align*}
\frac{d}{dt}\tau(\rho_{t}a)=\langle v_{t},a\rangle_{\rho_{t}}
\end{align*}

\noindent for each $a\in\mathfrak{A}$ and a.e.~$t\in [0,1]$. 
\end{dfn}

\begin{dfn}\label{DFN.Unbd_Adm}
Let $p,q\in\mathcal{D}$. An admissible path from $p$ to $q$ is a $\mu_{t}:[0,1]\longrightarrow T\mathcal{D}$ such that 

\begin{itemize}
\item[1)] $\mu_{t}$ satisfies the continuity equation,
\item[2)] $\rho_{0}=p$ and $\rho_{1}=q$,
\item[3)] $t\longmapsto ||v_{t}||_{\rho_{t}}^{2}\in L^{1}([0,1])$.
\end{itemize}

\noindent We denote the set of all admissible paths between $p$ and $q$ by $\mathcal{A}(p,q)$.
\end{dfn}

\begin{prp}\label{PRP.Ext_Cmptbl}
If $\mu_{t}:[0,1]\longrightarrow T\mathcal{D}_{b}$, then $\mu_{t}$ is an admissible path w.r.t.~Definition \ref{DFN.Bd_Adm} if and only if it is one w.r.t.~Definition \ref{DFN.Unbd_Adm}. 
\end{prp}
\begin{proof}
$M_{\rho_{t}}$ reduces to the multiplication operator for bounded densities if $\rho_{t}$ is bounded. Thus density of $\mathfrak{A}\subset A_{\partial}$ w.r.t.~$||.||_{\partial}$, which we have by $\mathfrak{A}$ being a core, implies both constructions of $||.||_{p}$ to yield the same tangent space at $\rho_{t}$. We are left to show that absolute continuity w.r.t.~$\mathfrak{A}$ implies absolute continuity w.r.t.~$A_{\partial}$ as well. This follows from $\tau(p)=1$ and $\mathfrak{A}\subset A$ being dense. 
\end{proof}

\noindent We copy Definition \ref{DFN.Bd_L2WDst} verbatim to define the $L^{2}$-Wasserstein distance on $\mathcal{D}$ associated to $(A,\tau,\partial)$ and $\mathfrak{A}$, using the wider class of admissible paths defined just above. Proposition \ref{PRP.Ext_Cmptbl} shows this to be compatible with our previous definition on bounded densities. 

\begin{ntn}
We denote the $L^{2}$-Wasserstein distance on $\mathcal{D}$ obtained from the above extension procedure by $\mathcal{W}_{2}$ in analogy to the bounded case.
\end{ntn} 

\begin{prp}\label{PRP.Unbd_Sep}
$\mathcal{W}_{2}$ defines a distance on $\mathcal{D}$. 
\end{prp}
\begin{proof}
Setting $B:=\mathfrak{A}$ and $g(a):=||\partial a||_{\infty}$, we obtain a separating function as in the bounded case by density of $\mathfrak{A}$ and the second statement of Lemma \ref{LEM.Ext_Mlt}. Definiteness follows exactly as in the bounded case.
\end{proof}																

\section{Symmetric gradients for $(\mathcal{K}(H),\textnormal{tr})$}

We discuss symmetric gradients for $(\mathcal{K}(H),\textrm{tr})$ in preparation of the fourth section, in particular fibre gradients. Significance of Theorem \ref{THM.Disint} is ensured by showing continuous dependence of minimisers on start- and endpoints if $H$ is finite-dimensional. This includes existence of minimisers in finite dimensions.

\subsection{Existence of minimisers for finite-dimensional $H$}

In this subsection, we assume $\partial$ to be a symmetric gradient for $(M_{n}(\mathbb{C}),\textrm{tr})$ with $n\in\mathbb{N}$ arbitrary. Without loss of generality, we assume $\partial$ to map into a finite-dimensional space. The first step is to show finiteness of $\mathcal{W}_{2}$. Example \ref{BSP.FinDim_RgImprSG} shows this to be true if $\partial$ has one-dimensional kernel. For the general case, we introduce invertible operators $S_{p}$ associated to each $p\in\mathcal{D}_{b}$. This will allow us to write $\dot{\rho}_{t}=S_{\rho_{t}}v_{t}$ for each admissible path. Continuous dependence of $S_{p}$ on $p$ will imply $\rho_{t}:=(1-t)p+tq$ to be an admissible path even between non-invertible densities.\par
Choose $p\in\mathcal{D}_{b}$ and decompose 

\begin{align*}
M_{n}(\mathbb{C})=T_{p}\mathcal{D}_{b}\bigoplus\ker M_{p}^{\frac{1}{2}}\partial
\end{align*}

\noindent orthogonally. By finite-dimensionality, we avoid a completion procedure when constructing the tangent space. We have $\im (\partial^{*}M_{p}\partial)_{|T_{p}\mathcal{D}_{b}}\subset T_{p}\mathcal{D}_{b}$, where $(\partial^{*}M_{p}\partial)_{|T_{p}\mathcal{D}_{b}}$ is injective and positive by construction of the tangent space. Hence $(\partial^{*}M_{p}\partial)_{|T_{p}\mathcal{D}_{b}}$ is an invertible operator on $T_{p}\mathcal{D}_{b}$ . Let $R_{p}$ be the projection onto $T_{p}\mathcal{D}_{b}$ in $M_{n}(\mathbb{C})$. We construct an operator on $M_{n}(\mathbb{C})$ by 

\begin{align*}
S_{p}:=(\partial^{*}M_{p}\partial)_{|T_{p}\mathcal{D}_{b}}R_{p}\oplus (1_{M_{n}(\mathbb{C})}-R_{p}).
\end{align*}

\noindent $S_{p}$ is invertible by construction and depends continuously on its base point since $M_{p}$ and $R_{p}$ do. The continuity equation and $v_{t}\in T_{p}\mathcal{D}_{b}$, as well as boundedness of $\partial$, imply $\dot{\rho}_{t}=(\partial^{*}M_{p}\partial)_{|T_{p}\mathcal{D}_{b}}v_{t}$ for any admissible path. We summarise our construction in a lemma.

\begin{lem}
For all $p\in\mathcal{D}_{b}$, there exists a positive invertible operator $S_{p}\in \mathcal{B}(M_{n}(\mathbb{C}))$ depending continuously on $p$ such that $\dot{\rho}_{t}=S_{\rho_{t}}(v_{t})$ for each $\mu_{t}=(\rho_{t},v_{t})\in\mathcal{A}(p,q)$.
\end{lem}

For all $p\in\mathcal{D}_{b}$ and all $x,y\in M_{n}(\mathbb{C})$, we consider $S_{p}^{-1}(x-y)$. This expression is jointly continuously w.r.t.~all three variables. Thus if $x$ and $y$ lie in a bounded set $K$, then $||S_{p}^{-1}(x-y)||$ is bounded on $\mathcal{D}_{b}\times K\times K$ by continuity.

\begin{prp}\label{PRP.FinDim_WkMetr}
If $p,q\in\mathcal{D}_{b}$, then $((1-t)p+tq,S_{(1-t)p+tq}^{-1}(p-q))\in\mathcal{A}(p,q)$. Furthermore, $\mathcal{W}_{2}$ has finite diameter and metrisises the $w^{*}$-topology on $\mathcal{D}_{b}$. 
\end{prp}
\begin{proof}
The path $\rho_{t}:=(1-t)p+tq$ is continuously differentiable with $\dot{\rho_{t}}=p-q=S_{\rho_{t}}(S_{\rho_{t}}^{-1}(p-q))$. Moreover, we have $||v_{t}||_{\rho_{t}}=||S_{\rho_{t}}^{-1}(p-q)||_{H}<C$ for some $C>0$ independent of $p,q$, by continuity and $\mathcal{D}_{b}\subset M_{n}(\mathbb{C})$ being bounded. Thus $(\rho_{t},v_{t})\in\mathcal{A}(p,q)$, and therefore $\mathcal{W}_{2}$ finite.\par
It is immediate that convergence in $\mathcal{W}_{2}$ implies convergence in the weak topology. For the converse, we use the uniform bound $C>0$ above and continuous dependence on $p$. This allows use of dominated convergence to obtain

\begin{align*}
\lim_{i}E((1-t)p_{i}+tp)=\lim_{i}\frac{1}{2}\int_{0}^{1}||S_{(1-t)p_{i}+tp}^{-1}(p_{i}-p)||_{H}^{2}dt=0
\end{align*}

\noindent for each sequence $p_{i}\in\mathcal{D}_{b}$ weakly convergent to $p$.
\end{proof}

\begin{prp}\label{PRP.FinDim_Min}
For all $p,q\in\mathcal{D}_{b}$, there exists a $\mu_{t}\in\mathcal{A}(p,q)$ such that $\mathcal{W}_{2}(p,q)=\sqrt{E(\mu_{t})}$.
\end{prp}
\begin{proof}
$\mathcal{A}(p,q)\neq\emptyset$ for all $p,q\in\mathcal{D}_{b}$ by Proposition \ref{PRP.FinDim_WkMetr}. Let $\mu_{t}^{i}\in\mathcal{A}(p,q)$ be a sequence such that $\sqrt{E(\mu_{t}^{i})}$ strictly decreases to $\mathcal{W}_{2}(p,q)$. $(E(\mu_{t}^{i}))_{i\in\mathbb{N}}$ is bounded, and we select a weakly convergent subsequence of $w_{t}^{i}\in L^{2}([0,1],\mathcal{H})$ by Banach-Alaoglu. By compactness of $\mathcal{D}_{b}$ in the $w^{*}$-topology, absolute continuity of $\rho_{t}^{i}$, and again boundedness of $(E(\mu_{t}^{i}))_{i\in\mathbb{N}}$, we choose a subsequence $\rho_{t}^{i}$ that $w^{*}$-converges uniformly to a path $\rho_{t}\in\mathcal{D}_{b}$ between $p$ and $q$ using Arzel\'a-Ascoli.\par
Finite-dimensionality of $H$ implies uniform convergence of $\rho_{t}^{i}$ to $\rho_{t}$ in norm. This implies

\begin{align*}
\lim_{i}||M_{\rho_{t}^{i}}^{\frac{1}{2}}-M_{\rho_{t}}^{\frac{1}{2}}||_{\mathcal{B}(\mathcal{H})}=0
\end{align*}

\noindent for each $t\in [0,1]$. Since all $\rho_{t}^{i},\rho_{t}$ are densities, Proposition \ref{PRP.Nrm_Mlt_Op} shows $||M_{\rho_{t}^{i}}^{\frac{1}{2}}-M_{\rho_{t}}^{\frac{1}{2}}||_{\mathcal{B}(\mathcal{H})}\leq 2$ and we apply dominated convergence to obtain

\begin{align*}
\lim_{i}|\int_{0}^{t}||(M_{\rho_{s}^{i}}^{\frac{1}{2}}-M_{\rho_{s}}^{\frac{1}{2}})\partial a||_{\mathcal{H}}^{2}ds|=0
\end{align*}

\noindent for each $a\in M_{n}(\mathbb{C})$. Using this and that $E(\mu_{t}^{i})$ is strictly decreasing, we calculate

\begin{align*}
\lim_{i}|\int_{0}^{t}\langle w_{s}^{i},M_{\rho_{s}^{i}}^{\frac{1}{2}}-M_{\rho_{t}}^{\frac{1}{2}}\partial a\rangle_{\mathcal{H}}ds|\leq E(\mu_{t}^{0})\lim_{i}|\int_{0}^{t}||(M_{\rho_{s}^{i}}^{\frac{1}{2}}-M_{\rho_{s}}^{\frac{1}{2}})\partial a||_{\mathcal{H}}^{2}ds|=0.
\end{align*}

\noindent This proves

\begin{align*}
\lim_{i}\int_{0}^{t}\langle w_{s}^{i},M_{\rho_{s}^{i}}\partial a\rangle_{\mathcal{H}}ds=\int_{0}^{t}\langle w_{t},M_{\rho_{t}}^{\frac{1}{2}}\partial a\rangle_{\mathcal{H}}ds
\end{align*}

\noindent for each $a\in M_{n}(\mathbb{C})$. Here, $w_{t}$ is the weak limit of $w_{t}^{i}$ in $L^{2}([0,1],\mathcal{H})$. All of this implies

\begin{align*}
\textrm{tr}(\rho_{t}a)=\lim_{i}\textrm{tr}(\rho_{t}^{i}a)=\lim_{i}\Big(\int_{0}^{t}\langle w_{s}^{i},M_{\rho_{s}^{i}}\partial a\rangle_{\mathcal{H}}ds\Big)+\textrm{tr}(pa)=\Big(\int_{0}^{t}\langle w_{t},M_{\rho_{t}}^{\frac{1}{2}}\partial a\rangle_{\mathcal{H}}ds\Big)+\textrm{tr}(pa)
\end{align*}

\noindent for each $a\in M_{n}(\mathbb{C})$. Thus $(\rho_{t},w_{t})\in\mathcal{A}(p,q)$. Moreover, l.s.c.~of $||.||_{L^{2}([0,1],\mathcal{H})}$ coupled with weak convergence of $w_{t}^{i}$ to $w_{t}$ yields $E(\mu_{t})\leq\liminf E(\mu_{t}^{i})=\mathcal{W}_{2}^{2}(p,q)$. Hence $\mu_{t}$ is a minimiser.
\end{proof}

\subsection{Fibre gradients and mass preservation}

In order to have mass preservation along fibres when dealing with vertical gradients, we require the latter to decompose into symmetric gradients for $(\mathcal{K}(H),\textrm{tr})$ with additional properties. The notion of fibre gradient encompasses precisely these properties. Proposition \ref{PRP.Mass_Prsv} is the result we need to show mass preservation in the fourth section.

\begin{dfn}\label{DFN.Fbr_Grd}
A symmetric gradient $\partial$ for $(\mathcal{K}(H),\textrm{tr})$ mapping to $\mathcal{S}_{2}(H)$ is a fibre gradient if 

\begin{itemize}
\item[1)] $\mathcal{S}_{1}(H)\subset D(\partial)$ and $\FinRk(H)$ is a core,
\item[2)] $\partial(\mathcal{S}_{1}(H))\subset\mathcal{S}_{1}(H)$,
\item[3)] $\partial^{*}=-\partial$,
\item[4)] $\partial$ extends to a bounded operator on $\mathcal{K}(H)$.
\end{itemize}
\end{dfn}

\begin{rem}\label{REM.Curveball}
In our setting, $\FinRk(H)$ being a core implies it being an extension algebra. This will be relevant in the fourth section exactly once [Link!].
\end{rem}

\begin{bsp}
For all $T\in\mathcal{B}(H)_{h}$, $i\textrm{Ad}_{T}$ is a fibre gradient. In particular, all symmetric gradients are fibre gradients if $H$ is finite-dimensional.
\end{bsp}

\begin{prp}\label{PRP.Bd_S1toS1}
If $\partial\in\mathcal{B}(\mathcal{S}_{2}(H))$ is a symmetric gradient for $(\mathcal{K}(H),\textrm{tr})$, then $\partial(\mathcal{S}_{1}(H))\subset\mathcal{S}_{1}(H)$. 
\end{prp}
\begin{proof}
Without loss of generality, we assume $x\in\mathcal{S}_{1}(H)_{h}$ by symmetry of $\partial$. Any element of $\mathcal{S}_{1}(H)_{h}$ can be split into positive and negative parts again lying in $\mathcal{S}_{1}(H)_{h}$. We therefore reduce to the case of positive $x\in\mathcal{S}_{1}(H)$ and write $x=y^{2}$ for a $y\in\mathcal{S}_{2}(H)_{+}$. By construction, $||x||_{\mathcal{S}_{1}(H)}=||y||_{\mathcal{S}_{2}(H)}^{2}$. For all $z\in\mathcal{S}_{2}(H)$, we have

\begin{align*}
|\textrm{tr}(\partial x z)|&\leq |\textrm{tr}(y\partial y z )|+|\textrm{tr}(\partial y y z)|\\
&=\langle\partial y,zy\rangle_{\mathcal{S}_{2}(H)}+\langle\partial y,yz\rangle_{\mathcal{S}_{2}(H)}\\
&\leq ||\partial y||_{\mathcal{S}_{2}}\big(\sqrt{\textrm{tr}(xz^{*}z)}+\sqrt{\textrm{tr}(xzz^{*})}\ \big)\\
&=||\partial||_{\mathcal{B}(\mathcal{S}_{2}(H)}\sqrt{||x||_{\mathcal{S}_{1}(H)}}\big(2\sqrt{||x||_{\mathcal{S}_{1}(H))}||z||_{\mathcal{K}(H)}^{2}}\ \big)\\
&=2||\partial||_{\mathcal{B}(\mathcal{S}_{2}(H))}||x||_{\mathcal{S}_{1}(H)}||z||_{\mathcal{K}(H)}.
\end{align*}

\noindent Since $\mathcal{S}_{2}(H)\subset\mathcal{K}(H)$ densely and $\mathcal{K}(H)^{*}=\mathcal{S}_{1}(H)$, $\partial x\in\mathcal{S}_{1}(H)$.
\end{proof}

\begin{rem}\label{REM.Mass_Prsv}
If $\partial\in\mathcal{B}(\mathcal{S}_{2}(H))$ is a bounded symmetric gradient for $(\mathcal{K}(H),\textrm{tr})$, then $1)$ in Definition \ref{DFN.Fbr_Grd} is satisfied by hypothesis and $2)$ is satisfied by Proposition \ref{PRP.Bd_S1toS1}. Thus $\partial$ is a fibre gradient if and only if $3)$ and $4)$ are satisfied. If $H$ is finite-dimensional, all symmetric gradients on $\mathcal{S}_{2}(H)$ are fibre gradients. In general, $\textrm{Ad}_{T}$ is a fibre gradient for each $T\in\mathcal{K}(H)$. 
\end{rem}

\begin{prp}\label{PRP.Mass_Prsv}
If $\partial$ is a fibre gradient and $(\eta_{i})_{i\in\mathbb{N}}\subset\mathcal{S}_{1}(H)$ an approximate identity in $\mathcal{B}(H)$, then $\eta_{i}\longrightarrow 0$ weakly in $T_{p}\mathcal{D}_{b}$ for each $p\in\mathcal{D}_{b}$. 
\end{prp}
\begin{proof}
Let $p\in\mathcal{D}_{b}$ be fix but arbitrary. From Lemma \ref{LEM.Ext_Mlt}, we know $p^{\alpha}\partial xp^{1-\alpha}\in\mathcal{S}_{1}(H)$ for each $x\in A_{\partial}$ and each $\alpha\in [0,1]$. Using the integral representation of $M_{p}$ we have by $\partial$ mapping into $L^{2}(\mathcal{K}(H),\textrm{tr})=\mathcal{S}_{2}(H)$, we obtain

\begin{align*}
\langle \eta_{i},\eta_{i}\rangle_{p}=\int_{0}^{1}||p||_{\mathcal{S}_{1}}||\partial\eta_{i}||_{\mathcal{B}(H)}^{2}d\alpha\leq(\sup_{i} ||\partial\eta_{i}||_{\mathcal{B}(H)})^{2}.
\end{align*}

\noindent As $\eta_{i}$ is an approximate identity, it $w^{*}$-converges in $\mathcal{B}(H)$. Thus $\sup_{i}||\eta_{i}||_{\mathcal{K}(H)}$ is finite. Hence $||\partial\eta_{i}||_{\mathcal{K}(H)}\leq||\partial||_{\mathcal{B}(\mathcal{K}(H))}\sup_{i}||\eta_{i}||_{\mathcal{K}(H)}$ is, where we used $4)$. Using our estimate just above, we see that $\sup_{i}||\eta_{i}||_{p}$.\par
We know $-\partial(p^{\alpha}\partial xp^{1-\alpha})\in\mathcal{S}_{1}(H)$ by $2)$, implying $\lim\textrm{tr}(p^{\alpha}\partial x p^{1-\alpha}\partial\eta_{i})=-\textrm{tr}(\partial(p^{\alpha}\partial xp^{1-\alpha})\eta_{i})=-\textrm{tr}(\partial(p^{\alpha}\partial xp^{1-\alpha}))$ since $\eta_{i}$ is an approximate identity. We claim the last term vanishes. By the Leibniz rule and $3)$, $\textrm{tr}(\partial(TS))=0$ for each $T,S\in\mathcal{S}_{1}(H)$. Since $\mathcal{S}_{1}(H)_{h}=\mathcal{S}_{1}(H)_{+}-\mathcal{S}_{1}(H)_{+}$, the claim follows. We apply dominated convergence to obtain

\begin{align*}
\langle x,\eta_{i}\rangle_{p}=\int_{0}^{1}\textrm{tr}(p^{\alpha}\partial x p^{1-\alpha}\partial\eta_{i})d\alpha\longrightarrow 0
\end{align*}

\noindent for each $x\in D(\partial)$. The latter is a dense subset of $T_{p}\mathcal{D}_{b}$ and $\sup_{i\in\mathbb{N}}||\eta_{i}||_{p}$ is finite. Together, this implies the statement.
\end{proof}

\subsection{Continuous dependence of minimisers on start- and endpoints}

We introduce the notion of continuous dependence of minimisers on start- and endpoints, a property we require of almost every fibre in order to apply a measurable selection theorem in our proof of Theorem \ref{THM.Disint}. For the remainder of this section, we assume $H$ to be separable, identify $H=\mathcal{S}_{2}(H)$ and let $\partial$ be a symmetric gradient for $(\mathcal{K}(H),\textrm{tr})$.\par
Choose countable $T_{k}\in\FinRk(H)\cap B_{\leq 1}(\mathcal{K}(H))$ lying densely in $B_{\leq 1}(\mathcal{K}(H))$. Then

\begin{align*}
d(S,R):=\sum_{k=0}^{\infty}\frac{1}{2^{k+1}}|\textrm{tr}((S-R)T_{k})|
\end{align*}

\noindent metrisises the $w^{*}$-topology on $\mathcal{S}_{cl}(\mathcal{K}(H)):=\overline{\mathcal{S}(\mathcal{K}(H))}=B_{\leq 1}(\mathcal{S}_{1}(H),||.||_{\mathcal{S}_{1}(H)})$. The latter is compact in $(\mathcal{S}_{1}(H),w^{*})$ by Banach-Alaoglu, hence $(\mathcal{S}_{cl}(\mathcal{K}(H)),d)$ is a compact metric space. For finite-dimensional $H$, $\mathcal{S}_{cl}(\mathcal{K}(H))$ is the unit sphere. 

\begin{ntn}
$\mathcal{S}_{cl}(\mathcal{K}(H)):=(\mathcal{S}_{cl}(\mathcal{K}(H)),d)$
\end{ntn}

\noindent We define a distance on $C([0,1],\mathcal{S}_{cl}(\mathcal{K}(H)))$ by setting

\begin{align*}
D(f,g):=\sup_{t\in [0,1]}d(f(t),g(t))
\end{align*} 

\noindent turning $(C([0,1],\mathcal{S}_{cl}(\mathcal{K}(H))),D)$ into a complete, separable metric space.

\begin{ntn}
$C([0,1],\mathcal{S}_{cl}(\mathcal{K}(H))):=(C([0,1],\mathcal{S}_{cl}(\mathcal{K}(H))),D)$.
\end{ntn}

\begin{ntn}
Let $\otimes_{\varepsilon}$ denote the injective tensor product of locally convex topological vector spaces.
\end{ntn}	

\begin{rem}
To see separability, first note that $[0,1]$ is a Kelly space and $(\mathcal{S}_{1}(H),w^{*})$ a locally convex Hausdorff space. Thus $C([0,1],(\mathcal{S}_{1}(H),w^{*}))\cong C([0,1])\otimes_{\varepsilon} (\mathcal{S}_{1}(H),w^{*})$ w.r.t.~the topology of uniform convergence. The latter space is immediately seen to be separable by separability of $C([0,1])$ and $(\mathcal{S}_{1}(H),w^{*})$. We have

\begin{align*}
C([0,1],\mathcal{S}_{cl}(\mathcal{K}(H)))\subset C([0,1],(\mathcal{S}_{1}(H),w^{*}))
\end{align*}

\noindent and uniform convergence is equivalent to convergence w.r.t.~the distance $D$. As subspace of a separable space, we thereby know $C([0,1],\mathcal{S}_{cl}(\mathcal{K}(H)))$ to be separable itself.
\end{rem}

\begin{dfn}\label{DFN.MinSet}
For all $p,q\in\mathcal{D}_{b}$, $\mathcal{M}(p,q):=\{\mu_{t}\in\mathcal{A}(p,q)\ |\ \mathcal{W}_{2}(p,q)=\sqrt{E(\mu_{t})}\}$ is the set of minimisers between $p$ and $q$.
\end{dfn}

\begin{lem}\label{LEM.FinDim_CntDpd_I}
If $H$ is finite-dimensional, then $\mathcal{M}(p,q)\neq\emptyset$ and $\mathcal{M}(p,q)\subset C([0,1],\mathcal{S}_{cl}(\mathcal{K}(H)))$ is closed for each $p,q\in\mathcal{D}_{b}$.
\end{lem}
\begin{proof}
Let $p_{i}\in\mathcal{D}_{b}$ be a sequence $||.||_{\mathcal{S}_{1}(H)}$-approximating $p$, resp.~$q_{i}\in\mathcal{D}_{b}$ a sequence $||.||_{\mathcal{S}_{1}(H)}$-approximating $q$. Choose minimisers $\mu_{t}^{i}\in\mathcal{A}(p_{i},q_{i})$, which exist by Proposition \ref{PRP.FinDim_Min}. We know that $\mathcal{W}_{2}$ metrisises the $w^{*}$-topology by Proposition \ref{PRP.FinDim_WkMetr}, hence $\lim_{i}\sqrt{E(\mu_{t}^{i})}=\mathcal{W}_{2}(p,q)$. In particular, we obtain boundedness of $(E(\mu_{t}^{i}))_{i\in\mathbb{N}}\subset\mathbb{R}$. Then the argument used in our proof of Proposition \ref{PRP.FinDim_Min} for extracting a minimiser works the same for varying but $||.||_{\mathcal{S}_{1}(H)}$-converging start- and endpoints, modulo obvious minor modifications. We thus extract a subsequence $\mu_{t}^{i}$ of minimisers in order to obtain a minimisers from $p$ to $q$. This shows $\mathcal{M}(p,q)$ to be non-empty.\par
Given a converging sequence $\mu_{t}^{i}\in\mathcal{M}(p,q)$ in $C([0,1],\mathcal{S}_{cl}(\mathcal{K}(H)))$, $E(\mu_{t}^{i})=\mathcal{W}_{2}(p,q)$ allows us to extract a subsequence converging to a minimiser $\mu_{t}\in\mathcal{M}(p,q)$ as before. As $\mu_{t}^{i}$ converges by hypothesis, $\mu_{t}$ must be the limit of the whole sequence. 
\end{proof}

\begin{dfn}\label{DFN.Cnt_Dpd}
A symmetric gradient $\partial$ for $(\mathcal{K}(H),\textrm{tr})$ has continuous dependence of minimisers on start- and endpoints if for all $p,q\in\mathcal{D}_{b}$ and all $(p_{i})_{i\in\mathbb{N}},(q_{i})_{i\in\mathbb{N}}\subset\mathcal{D}_{b}$ with $p_{i}\longrightarrow p$, resp.~$q_{i}\longrightarrow q$ in the $||.||_{\mathcal{S}_{1}(H)}$-topology, we know that

\begin{itemize}
\item[1)] there exist $\mu_{t}\in\mathcal{M}(p,q)$ and $\mu_{t}^{i_{k}}\in\mathcal{M}(p_{i_{k}},q_{i_{k}})$ with $\lim_{k\in\mathbb{N}}D(\mu_{t}^{i_{k}},\mu_{t})=0$,
\item[2)] the limit of each $D$-converging sequence of $\mu_{t}^{i}\in\mathcal{M}(p_{i},q_{i})$ lies in $\mathcal{M}(p,q)$. 
\end{itemize}
\end{dfn}

\begin{rem}
We expect continuous dependence of minimisers on start- and endpoints if $E$ is lower semi-continuous. Moreover, if we know $2)$ and have existence of a $D$-converging sequence of $\mu_{t}^{i}$, $1)$ follows immediately. 
\end{rem}

\begin{prp}\label{PRP.Cnt_Dpd}
If $\partial$ has continuous dependence of minimisers on start- and endpoints, then $\mathcal{M}(p,q)$ is non-empty and closed w.r.t.~$D$ for each $p,q\in\mathcal{D}_{b}$.
\end{prp}
\begin{proof}
Set $p_{i}=p, q_{i}=q$ and apply $1)$ to see non-emptiness, $2)$ for closedness.
\end{proof}

\begin{lem}\label{LEM.FinDim_CntDpd_II}
If $H$ is finite-dimensional, all symmetric gradients have continuous dependence of minimisers on starting- and endpoints.
\end{lem}
\begin{proof}
Our argument proving non-emptiness of $\mathcal{M}(p,q)$ in Lemma \ref{LEM.FinDim_CntDpd_I} makes no assumption on the sequences $(p_{i})_{i\in\mathbb{N}},(q_{i})_{i\in\mathbb{N}}\in\mathcal{D}_{b}$ used. As we extract a minimising subsequence by Arzel\'a-Ascoli, $\lim D(\mu_{t}^{i_{k}},\mu_{t})=0$ follows. We thus have $1)$. If we already have $D$-convergence, we argue as in the proof of Lemma \ref{LEM.FinDim_CntDpd_I} \textit{after} having applied Arzel\'a-Ascoli to show $2)$.
\end{proof}

\section{Vertical gradients for trivial $\mathcal{K}(H)$-bundles}

We establish our setting, prove the disintegration theorem and consider mean entropic curvature bounds as an application. For the remainder of this section, let $X$ be a locally compact Hausdorff space, $\mathcal{B}(X)$ its Borel $\sigma$-algebra, $(X,\mathcal{B}(X))$ a separable measure space and $H$ a separable Hilbert space. Separability of $(X,\mathcal{B}(X))$ ensures $L^{p}(X.\nu)$ to be separable as Banach space for each Radon measure $\nu$.

\subsection{Product traces, their $L^{p}$-spaces and vertical gradients}

As before, $\otimes_{\varepsilon}$ denotes the injective tensor product. We have $C_{c}(X,E)=C_{c}(X)\otimes_{\varepsilon}E$ for each Banach space $E$ since $X$ is a Kelly space by local compactness. Thus $C_{c}\odot E\subset C_{c}(X,E)$ densely, while $C_{c}(X,E)\subset C_{0}(X,E)$ holds in any case. 
	
\begin{dfn}
If $\tau$ is a trace on $C_{0}(X,\mathcal{K}(H))$ such that

\begin{itemize}
\item[1)]$C_{c}(X,\mathcal{S}_{1}(H))\subset D(\tau)$,
\item[2)]$T\longmapsto \tau(f\odot T)=:\tau_{f}(T)$ is a bounded linear functional on $\mathcal{S}_{1}(H)$ for each $f\in C_{c}(X)$,
\end{itemize}

\noindent then $\tau$ is called a product trace.
\end{dfn}

\begin{prp}\label{PRP.Prd_Tr}
If $\nu$ is a Radon measure on $X$, the functional $\nu\odot\textrm{tr}$ on $C_{c}(X)\odot\mathcal{S}_{1}(H)$ induces a unique product trace denoted by $\nu\otimes\textrm{tr}$. Moreover and for each product trace $\tau$, there exists a unique Radon measure $\nu$ on $X$ such that $\tau=\nu\otimes\textrm{tr}$. 
\end{prp}
\begin{proof}
Consider $L^{1}(X,\mathcal{S}_{1}(H),d\nu)$ and define the subspace

\begin{align*}
D(\nu\otimes\textrm{tr}):=C_{0}(X,\mathcal{S}_{1}(H))_{+}\cap L^{1}(X,\mathcal{S}_{1}(H),d\nu).
\end{align*}

\noindent We set 

\begin{align*}
(\nu\otimes\textrm{tr})(F) := \begin{cases} 
\int_{X}\textrm{tr}(F(x))d\nu & \textrm{if}\ F\in D(\tau) \\
\infty & \textrm{else}
\end{cases}
\end{align*}

\noindent yielding a trace $\nu\otimes\textrm{tr}$ on $C_{0}(X,\mathcal{K}(H))$. As $C_{c}(X,\mathcal{S}_{1}(H))\subset D(\nu\otimes\textrm{tr})$ holds by construction, $\nu\otimes\textrm{tr}$ is a product trace. Furthermore, we have $(\nu\otimes\textrm{tr})(|F|)=\int_{X}\textrm{tr}(|F(x)|)d\nu$ for each $F\in D(\tau)$ because $|F|(x)=|F(x)|$ for all $x\in X$. $C_{c}(X)\odot\mathcal{S}_{1}(H)$ lies dense in both $L^{1}(A,\tau)$ and $L^{1}(X,\mathcal{S}_{1}(H))$ since it lies dense in $D(v\otimes\textrm{tr})$ w.r.t.~either topology. Thus the $L^{1}$-space defined by $\nu\otimes\textrm{tr}$ is $L^{1}(X,\mathcal{S}_{1}(H),d\nu)$ by construction, and $\nu\odot\textrm{tr}$ uniquely determines the product trace $\nu\otimes\textrm{tr}$.\par
Let $\tau$ be a product trace. For all positive $f\in C_{c}(X)$, set $S_{f}\in B(H)=\mathcal{S}_{1}(H)^{*}$ for the unique element such that $\tau_{f}=\textrm{tr}(S_{f}\hspace{0.05cm}.\hspace{0.05cm})$. Positivity and traciality of $\tau$ imply the same to hold for each $\tau_{f}$. Thus $S_{f}=L(f)1_{\mathcal{B}(H)}$ for a unique $L(f)\in [0,\infty)$. Here, positivity of $L(f)$ follows from positivity of $\tau_{f}$. A similar argument applies for negative $f$, with $L(f)\in(-\infty,0]$. We decompose $f=f_{+}+f_{-}$, for $f_{+}=\max\{f,0\}$ and $f_{-}=\min\{f,0\}$. Linearity of $\tau$ implies $\tau_{f}=((L(f_{+})+L(f_{-}))\textrm{tr}$ and $\tau_{f+g}=((L(f)+L(g))\textrm{tr}$. We obtain $L(f)=L(f_{+})+L(f_{-})$ and $L(f+g)=L(f)+L(g)$. Hence $L$ is a positive linear functional on $C_{c}(X)$, and there exists a unique Radon measure $\nu$ on $X$ representing $L$. We have $\tau_{|C_{c}(X)\odot\mathcal{S}_{1}(H)}=\nu\odot\textrm{tr}$ by construction of $L$. The second statement follows by uniqueness of $\nu\otimes\textrm{tr}$. 
\end{proof}

\begin{cor}\label{COR.Prd_Tr}
Let $H$ be finite-dimensional. If $\tau$ is a finite trace on $C_{0}(X,\mathcal{K}(H))$, then it is a product trace with finite Radon measure.
\end{cor}
\begin{proof}
Since $H$ is finite-dimensional, we only need to show $C_{c}(X,\mathcal{S}_{1}(H))\subset D(\tau)$. This is true by hypothesis, as $\tau$ is defined on all of $C_{0}(X,\mathcal{K}(H))$. Finiteness of $\nu$ is implied by finiteness of $\tau=\nu\otimes\textrm{tr}$.
\end{proof}

Our proof of Proposition \ref{PRP.Prd_Tr} shows each product trace $\tau=\nu\otimes\textrm{tr}$ to have $C_{0}(X,\mathcal{S}_{1}(H))\cap L^{1}(X,\mathcal{S}_{1}(H),d\nu)$ as domain in $C_{0}(X,\mathcal{K}(H))$. Furthermore, we saw the $L^{1}$-space of $\tau$ to equal $L^{1}(X,\mathcal{S}_{1}(H),d\nu)$. We generalise this relation to arbitrary $p\in [1,\infty]$.

\begin{ntn}
We fix a product trace $\tau=\nu\otimes\textrm{tr}$ for the remainder of this section and drop all references to $\nu$ from our $L^{p}$-space notation in the future. 
\end{ntn}

\begin{prp}\label{PRP.Prd_Tr_LP}
If $\tau$ is a product trace, then $L^{p}(C_{0}(X,K(H)),\tau)=L^{p}(X,\mathcal{S}_{p}(H))$ for each $p\in [1,\infty]$.
\end{prp}
\begin{proof}
Let $p\in [1,\infty)$. For all $F\in D(\tau)$, we know $|F|^{p}\in D(\tau)$. Then $\tau=\nu\otimes\textrm{tr}$ implies 

\begin{align*}
\tau(|F|^{p})=\int_{X}\textrm{tr}(|F(x)|^{p})d\nu
\end{align*}

\noindent for each $F\in D(\tau)$. We argue by density to obtain our statement for general $p\in [1,\infty)$ in direct analogy to Proposition \ref{PRP.Prd_Tr}.\par
Let $p=\infty$. We have $L^{\infty}(X,\mathcal{B}(H))\subset L^{1}(X,\mathcal{S}_{1}(H))^{*}$ isometrically via the map

\begin{align*}
F\longmapsto\Big(G\longmapsto\int_{X}\textrm{tr}(F(x)G(x))d\nu\hspace{0.05cm}\Big).
\end{align*}

\noindent We already saw $L^{1}(X,\mathcal{S}_{1}(H))=L^{1}(C_{0}(X,\mathcal{K}(H),\tau)$, thus 

\begin{align*}
C_{0}(X,\mathcal{K}(H))\subset L^{\infty}(X,\mathcal{B}(H))\subset L^{1}(C_{0}(X,\mathcal{K}(H)),\tau)^{*}=L^{\infty}(C_{0}(X,\mathcal{K}(H)),\tau)
\end{align*}

\noindent where the last object is the $W^{*}$-algebra generated by $C_{0}(X,\mathcal{K}(H)$ represented over $L^{2}(X,\mathcal{S}_{2}(H))$. Moreover, we used $L^{1}(A,\omega)^{*}=L^{\infty}(A,\omega)$ for each trace $\omega$ on any $C^{*}$-algebra $A$. Our statement follows from $C_{0}(X,\mathcal{K}(H))\subset L^{\infty}(C_{0}(X,\mathcal{K}(H)),\tau)$ densely in the strong operator-topology should $L^{\infty}(X,\mathcal{B}(H))$ be closed w.r.t.~the strong operator-topology.\par
Choose a countable subset $N\subset\mathcal{B}(H)$ dense in the $w^{*}$-topology and let $T\in L^{\infty}(X,\mathcal{B}(H))^{\prime}$. $T$ commutes with every $1_{X}\otimes S\in L^{\infty}(X,\mathcal{B}(H))$, $S\in N$ arbitrary. As $N$ is countable, we have $T(x)\in N^{\prime}$ for a.e.~$x\in X$. By density of $N$ and $\mathcal{B}(H)$ being a factor, $T\in L^{\infty}(X)$. Hence $L^{\infty}(X,\mathcal{B}(H))^{\prime}=L^{\infty}(X)$. Theorem IV.7.10 in \cite{TakTOAI} states that $L^{\infty}(X)^{\prime}=L^{\infty}(X,\mathcal{B}(H))$, and $L^{\infty}(X,\mathcal{B}(H))^{\prime\prime}=L^{\infty}(X,\mathcal{B}(H))$ is a $W^{*}$-algebra. In particular, it is closed in the strong operator-topology.
\end{proof}

\begin{rem}\label{REM.Prd_Tr_LP}
We have $L^{1}(X,\mathcal{S}_{1}(H))\cong L^{1}(X)\otimes_{\pi}\mathcal{S}_{1}(H)$. Furthermore, $L^{1}(X)$ and $\mathcal{S}_{1}(H)$ are separable. Thus $L^{1}(X,\mathcal{S}_{1}(H))$ is separable. The same holds true for $L^{2}(X,\mathcal{S}_{2}(H))$, where we use the tensor product of Hilbert spaces rather than the projective tensor product $\otimes_{\pi}$ above.
\end{rem}

\begin{cor}\label{COR.L1_Pstv}
Each class in $L_{+}^{1}(X,\mathcal{S}_{1}(H))$ can be represented by an integrable function $F$ such that $F(x)\geq 0$ for every $x\in X$. 
\end{cor}
\begin{proof}
Each $F\in L_{+}^{1}(X,\mathcal{S}_{1}(H))$ can be expressed as $F=G^{*}G$ by unbounded Borel functional calculus. Any representative in $G$ thus induces a representative in $F$ as required.
\end{proof}

We establish an appropriate setting for a symmetric gradient. To do so, we have to define an action of  $L^{\infty}(X,B(H))\otimes_{max} L^{\infty}(X,B(H))^{op}$ on a Hilbert space equipped with an appropriate involution $J$. In light of Definition \ref{DFN.Fbr_Grd} and Proposition \ref{PRP.Mass_Prsv}, we focus on $\bigoplus_{k=1}^{m}L^{2}(X,\mathcal{S}_{2}(H))=L^{2}(X,\bigoplus_{k=1}^{m}\mathcal{S}_{2}(H))$ as our Hilbert space. This will enable us to prove Proposition \ref{PRP.Mass_Prsv_Unbd}, i.e.~mass preservation in almost every fibre. 

\begin{rem}
Each $\bigoplus_{k=1}^{m}L^{2}(X,\mathcal{S}_{2}(H))=L^{2}(X,\bigoplus_{k=1}^{m}\mathcal{S}_{2}(H))$ is equipped with the canonical left and right action of $L^{\infty}(X,\mathcal{B}(H))$ induced by pointwise multiplication and pointwise adjoining. For each $F\in L^{\infty}(X,\mathcal{B}(H))$ and $G\in\bigoplus_{k=1}^{m}L^{2}(X,\mathcal{S}_{2}(H))$, we thus have $(F.G)_{k}(x)=F(x).G_{k}(x)$ and $(G.F)_{k}(x)=G_{k}(x).F(x)$. 
\end{rem}

We next discuss vertical gradients. Assume we are given $\mathcal{H}:=\bigoplus_{k=1}^{m}\mathcal{S}_{2}(H)$. We consider a family $(\partial_{x})_{x\in X}$ of symmetric gradients for $(\mathcal{K}(H),\textrm{tr})$ mapping to $\mathcal{H}$ such that

\begin{itemize}
\item[$\bullet$] $(\partial_{x})_{k}$ is a fibre gradient for all $k\in\{1,...,m\}$ for a.e.~$x\in X$,
\item[$\bullet$] $x\longmapsto\partial_{x}T$ is measurable for each $T\in\FinRk(H)$.
\end{itemize}

\noindent For $f\odot T\in C_{c}(X)\odot\FinRk(H)$, set $\partial(f\odot T)(x):=f(x)\partial_{x}T$ and consider

\begin{align*}
D_{cb}(\partial):=\{F\in C_{c}(X)\odot\FinRk(H)\ |\ ||\partial_{x}F(x)||_{\mathcal{H}}\in L^{2}(X)\cap L^{\infty}(X)\}
\end{align*}

\noindent to obtain a densely defined, unbounded operator $(\partial,D_{cb}(\partial))$ from $L^{2}(X,\mathcal{S}_{2}(H))$ to $L^{2}(X,\mathcal{H})$. The operator is closable since each $\partial_{x}$ is closable and $L^{1}$-convergence implies a.e.~pointwise convergence for a subsequence. We denote this closure by $\partial$ as well. Observe that $(\partial F)(x)=\partial_{x}F(x)$ a.e.~for each $F\in D(\partial)$ by construction as each $\partial_{x}$ is closed.

\begin{dfn}\label{DFN.Vrt_Grd}
Let $(\partial_{x})_{x\in X}$ be a family as above. We call $\partial$ the induced operator of $(\partial_{x})_{x\in X}$ and call it a vertical gradient if 

\begin{itemize}
\item[1)] $C_{c}(X)\odot\FinRk(H)=D_{cb}(\partial)$,
\item[2)] there exists an approximate identity $(\eta_{i})_{i\in\mathbb{N}}\subset\FinRk(H)$ such that $\sup_{x\in K,i\in\mathbb{N}}||\partial_{x}\eta_{i}||_{\mathcal{B}(H)}$ is finite for each compact set $K\subset X$.
\end{itemize}
\end{dfn}

\begin{prp}\label{PRP.Vrt_Grd_Prop}
If $\partial$ is a vertical gradient, it is a symmetric gradient for $(C_{0}(X,\mathcal{K}(H)),\tau)$ such that $(\partial F)(x)=\partial_{x}F(x)$ a.e.~for each $F\in D(\partial)$. Furthermore, $C_{c}(X)\odot\FinRk(H)$ is an extension algebra.
\end{prp}
\begin{proof}
By construction, $\partial$ is a closed, densely defined, unbounded operator such that $(\partial F)(x)=\partial_{x}F(x)$ for each $F\in D(\partial)$. The latter implies $\partial$ to be a derivation since each $\partial_{x}$ is one. As $C_{c}(X)\odot\FinRk(H)\subset D(\partial)$, $\partial$ is a gradient. Symmetry follows from $(\partial F)(x)=\partial_{x}F(x)$ and $\partial_{x}$ being symmetric. We already used density of $C_{c}(X)\odot\FinRk(H)$ in $C_{0}(X,\mathcal{K}(H))$, while it is a core by construction of $\partial$. It therefore is an extension algebra by definition of $D_{cb}(\partial)$.
\end{proof}

\begin{bsp}
If $H$ is finite-dimensional such that $x\longmapsto||\partial_{x}||$ is locally bounded, then $\partial$ is a vertical gradient by Remark \ref{REM.Mass_Prsv}.
\end{bsp}

\begin{bsp}
Let $\partial_{0}$ be a symmetric gradient for $(\mathcal{K}(H),\textrm{tr})$ mapping to $\mathcal{H}$ such that each $\partial_{k}$ is a fibre gradient. For all $f\in L_{loc}^{\infty}(X,\mathcal{K}(H))$, the measurable family given by $\partial_{x}:=f(x)\partial_{0}$ induces a vertical gradient.
\end{bsp}

\begin{rem}
All derivations from $M_{n}(\mathbb{C})$ to itself are of form $\textrm{Ad}_{T}$. Hence if $H$ is finite-dimensional, then $\partial_{x}=(\textrm{Ad}_{T_{1}(x)},\hdots,\textrm{Ad}_{T_{m}(x)})$ for a measurable family with $(T_{1}(x),...,T_{m}(x))\in M_{n}(\mathbb{C})^{m}$. 
\end{rem}

For the remainder of this section, let $\partial$ be a vertical gradient. As described in Subsection 2.4, having an extension algebra allows us to extend to unbounded densities. It furthermore implies existence of a separating function suitable for our extension. Here, $\mathfrak{A}:=C_{c}(X)\odot\FinRk(H)$ is the extension algebra we consider. From what we have seen at the beginning of this subsection, a density is an element $P\in L_{+}^{1}(X,\mathcal{S}_{1}(H))$ such that $\int_{X}\textrm{tr}(P(x))d\nu=1$.

\begin{dfn}For all $P\in\mathcal{D}$, we define
\begin{align*}
\theta_{P}(x):= \begin{cases}\big(\textrm{tr}(P(x))\big)^{-\frac{1}{2}} & \textrm{if}\ P(x)\neq 0 \\
0 & \textrm{else}
\end{cases}
\end{align*}

\noindent and set $\nu_{P}:=\textrm{tr}(P(x))d\nu$.
\end{dfn}

\noindent If $A\in L^{1}([0,1],L^{1}(X,\mathcal{S}_{1}(H)))$, then $(\int_{0}^{1}A_{t}dt)(x)=\int_{0}^{1}A_{t}(x)dt$. If $m=1$, we have

\begin{align*}
M_{P}(F)(x)=\Big(\int_{0}^{1}P^{\alpha}FP^{1-\alpha}d\alpha\Big)(x)=\int_{0}^{1}P(x)^{\alpha}F(x)P(x)^{1-\alpha}d\alpha=M_{P(x)}(F(x))
\end{align*}

\noindent for each $P\in\mathcal{D}$ and $F\in L^{\infty}(X,\mathcal{B}(H))$. $G^{\alpha}(x)=G(x)^{\alpha}$ is immediate if $G\in L_{+}^{1}(X,\mathcal{S}_{1}(H))$ is already bounded. The unbounded case follows by approximating $G$ in $L^{\alpha}(X,\mathcal{S}_{1}(H))$ with $\min\{G,C_{i}\}(x):=\min\{G(x),C_{i}\}$, where $C_{i}\geq 0$ is a strictly increasing sequence. The tangent space inner product at $P$ is thus given by

\begin{align*}
\langle F,G\rangle_{P}&=\int_{X}\textrm{tr}(M_{P(x)}(\partial_{x}F(x)^{*})\partial_{x}G(x))d\nu\\
&=\int_{X}\textrm{tr}(M_{P(x)}^{\frac{1}{2}}(\partial_{x}F(x)^{*})M_{P(x)}^{\frac{1}{2}}(\partial_{x}G(x)))d\nu\\
&=\int_{X}\langle F(x),G(x)\rangle_{\theta_{P}^{2}(x)P(x)}d\nu_{P}.
\end{align*}

\noindent for each $F,G\in\mathfrak{A}$, where used $P(x)\in\mathcal{B}_{+}(H)$ a.e.~to ensure that $M_{P(x)}^{\frac{1}{2}}$ is defined. The case of general $m\in\mathbb{N}$ follows at once as the above describes the situation on each summand.\par
For all $P\in\mathcal{D}$ and all $F\in\mathfrak{A}$, set $(M_{P}^{\frac{1}{2}}\partial F)_{k}(x):=M_{P(x)}^{\frac{1}{2}}(\partial_{x})_{k}F(x)$. Then $M_{P}^{\frac{1}{2}}\partial F$ is strongly measurable by Lemma 7.5 in \cite{TakTOAI}, and lies in $L^{2}(X,\mathcal{H})$ by what we showed just above. Furthermore, we have 

\begin{align*}
||F||_{P}=||M_{P}^{\frac{1}{2}}\partial F||_{L^{2}(X,\mathcal{H})}
\end{align*}

\noindent and are therefore able to identify $T_{P}\mathcal{D}$ isometrically with a subspace of $L^{2}(X,\mathcal{H})$ in direct analogy to the bounded case. 

\begin{ntn}
For an admissible path in the setting of vertical gradients, we write $\mu_{t}=(P_{t},V_{t})=(P_{t},W_{t})$ instead of $\mu_{t}=(\rho_{t},v_{t})=(\rho_{t},w_{t})$. Here, $W_{t}$ is the unique vector in $L^{2}(X,\mathcal{H})$ associated to $V_{t}$ such that $||V_{t}||_{P_{t}}=||W_{t}||_{L^{2}(X,\mathcal{H})}$. 
\end{ntn}

\begin{prp}\label{PRP.Mass_Prsv_Unbd}
If $\mathcal{A}(P,Q)\neq\emptyset$, then $\textrm{tr}(P(x))=\textrm{tr}(P_{t}(x))$ for a.e.~$x\in X$ for each $t\in [0,1]$.
\end{prp}
\begin{proof}
Let $\eta_{i}\in\FinRk(H)$ be an approximate identity. Consider $f\odot\eta_{i}\in C_{c}(X)\odot\FinRk(H)$. Then for all $P\in\mathcal{D}$ and all $F\in\mathfrak{A}$, we have

\begin{align*}
\langle g\odot T,f\odot\eta_{i}\rangle_{P}=\sum_{k=1}^{m}\int_{X}f(x)(\langle F(x),\eta_{i}\rangle_{\theta_{P}^{2}(x)P(x)})_{k}d\nu_{P}
\end{align*}

\noindent and furthermore estimate

\begin{align*}
|(\langle F(x),\eta_{i}\rangle_{\theta_{P}^{2}(x)P(x)})_{k}|\leq ||\partial_{x}F(x)||_{\mathcal{B}(H)}||\partial_{x}\eta_{i}||_{\mathcal{B}(H)}\leq ||\partial_{x}F(x)||_{\mathcal{S}_{2}(H)}||\partial_{x}\eta_{i}||_{\mathcal{B}(H)}.
\end{align*}

\noindent We have $\sup_{x\in X}||\partial_{x}F(x)||_{\mathcal{S}_{2}(H)}<\infty$ by $1)$ and $\sup_{x\in\textrm{supp}\hspace{0.05cm}f,i\in\mathbb{N}}||\partial_{x}\eta_{i}||_{\mathcal{B}(H)}<\infty$ by $2)$ in Definition \ref{DFN.Vrt_Grd}. Since $P$ was fixed and Proposition \ref{PRP.Mass_Prsv} shows pointwise convergence to zero, we are able to apply dominated convergence to show convergence to zero of the integral above. Setting $g=f$ and $T=\eta_{i}$, we see $||f\odot\eta_{i}||_{P}$ to be bounded uniformly in $P$ and $i\in\mathbb{N}$. In particular, $f\odot\eta_{i}$ converges weakly to zero in each $T_{P}\mathcal{D}$ by density of $\mathfrak{A}$.\par
Let $\mu_{t}\in\mathcal{A}(P,Q)$. Observe that since $\eta_{i}$ converges to the identity in the $w^{*}$-topology, $||\eta_{i}||_{\mathcal{B}(H)}$ must be uniformly bounded. For all $t\in [0,1]$ and all $f\in C_{c}(X)$, we therefore have 

\begin{align*}
\int_{X}f(x)\textrm{tr}(P-P_{t})d\nu=\lim\int_{X}f(x)\textrm{tr}((P-P_{t})\eta_{i})d\nu=\lim\int_{0}^{t}\langle V_{s},f\odot\eta_{i}\rangle_{P_{s}}ds.
\end{align*}

\noindent The right-hand side is zero since $\sup_{i,P} ||\eta_{i}||_{P}<\infty$ and $||V_{t}||_{\rho_{t}}^{2}\in L^{1}([0,1])$. Since $f$ was arbitrary, $\textrm{tr}(P(x))=\textrm{tr}(P_{t}(x))$ almost everywhere.
\end{proof}

\begin{dfn}
Let $\mathcal{D}(X,\nu)$ be the set of densities on $(X,\nu)$. For $f\in\mathcal{D}(X,\nu)$, let $\mathcal{D}_{f}$ be the set of all $P\in\mathcal{D}$ with $\textrm{tr}(P(x))=f(x)$ almost everywhere. 
\end{dfn}

\begin{cor}\label{COR.Mass_Prsv_Unbd}\hspace{1cm}
\begin{itemize}
\item[1)] $(\mathcal{D},\mathcal{W}_{2})=\underset{f\in\mathcal{D}(X,\nu)}{\coprod}\ (\mathcal{D}_{f},\mathcal{W}_{2})$.
\item[2)] Admissible paths starting at an (un-)bounded density remain (un-)bounded.
\item[3)] $\mathcal{W}_{2}$ does not metrisise the $w^{*}$-topology.
\end{itemize}
\end{cor}
\begin{proof}
Any $P\in\mathcal{D}$ induces a density on $(X,\nu)$ by $f(x):=\textrm{tr}(P(x))$. Given $f\in \mathcal{D}(X,\nu)$, choose some density matrix $p\in\mathcal{S}_{1}(H)$ and set $P(x):=f(x)p$ to obtain a $P\in\mathcal{D}$ such that $\textrm{tr}(P(x))=f(x)$. This and \ref{PRP.Mass_Prsv_Unbd} imply the first and second statement. The third statement follows from the first, as $(\mathcal{D},w^{*})$ is connected.
\end{proof}

\begin{rem}\label{REM.Mass_Prsv_Unbd}
If $X$ is compact and $H$ finite-dimensional, any vertical gradient induces a $\mathcal{W}_{2}$ \textit{not} metrisising the $w^{*}$-topology on $\mathcal{D}$ even as its underlying $C^{*}$-algebra is unital. This is a departure from the commutative case.
\end{rem}

\begin{lem}\label{LEM.Msrbl_Unbd}
If $\mu_{t}$ is an admissible path, then $P_{t}\in L^{1}([0,1],L^{1}(X,\mathcal{S}_{1}(H)))$.
\end{lem}
\begin{proof}
When proving Lemma \ref{LEM.Msrbl_Bd}, we did not require boundedness of $\rho_{t}$. Hence our argument remains applicable to $P_{t}$ since $L^{1}(A,\tau)=L^{1}(X,\mathcal{S}_{1}(H))$ in our current setting.
\end{proof}

\subsection{Proving the disintegration theorem}

We first show all $\mu_{t}=(P_{t},W_{t})\in\mathcal{A}(P,Q)$ to be rectifiable, by which we mean the following. We fix a representative of $W_{t}\in L^{2}(X,\mathcal{H})$ for each $t\in [0,1]$, which we again denote by $W_{t}$. Then there exists a representative $P_{t}^{rct}$ of $P_{t}\in L^{1}([0,t]\times X,\mathcal{S}_{1}(H))$ such that 

\begin{align*}
\mu_{t}^{rct}(x):=(\theta_{P}^{2}(x)P_{t}^{rct}(x),\theta_{P}(x)W_{t}(x))\in\mathcal{A}(\theta_{P}^{2}(x)P(x),\theta_{P}^{2}(x)Q(x))
\end{align*}

\noindent for a.e.~$x\in X$. By construction, we have $\mu_{t}=(P_{t}^{rct},W_{t})$. This implies a mean energy representation of the energy functional.\par
Assuming continuous dependence of minimisers on start- and endpoints for a.e.~fibre, the second step is application of a measurable selection theorem. The latter is used to find an integrable collection of fibre-wise minimisers $(\xi_{t}(x))_{x\in X}\in\mathcal{A}(\theta_{P}^{2}(x)P(x),\theta_{P}^{2}(x)Q)$ integrating to an element of $\mathcal{A}(P,Q)$ after norming with $\theta_{P}^{-2}$. This yields a minimiser by the mean energy representation.

\begin{rem}
$L^{1}(X,\mathcal{S}_{1}(H))$ and $L^{2}(X,\mathcal{S}_{2}(H))$ are separable, see Remark \ref{REM.Prd_Tr_LP}.
\end{rem}

\begin{lem}\label{LEM.Rctf_NS}
Let $(X,\nu)$ and $(Y,\eta)$ be locally compact Hausdorff spaces equipped with Radon measures. Moreover, let $E$ be a Banach space. If $F$ is a representative of $0\in L^{1}(X\times Y,E)$, then $N_{x}:=\{y\in Y\ |\ F(x,y)\neq 0\}$ is a nullset for a.e~$x\in X$.
\end{lem}
\begin{proof}
For $x\in X$ fix, set $g_{x}(y):=||F(x,y)||_{E}$. Then $N_{x}=g_{x}^{-1}(\mathbb{R}\setminus {0})$, hence $N_{x}$ is measurable. We know $\int_{X}\int_{Y}||F(x,y)||_{E}\hspace{0.025cm}d\eta\hspace{0.025cm}d\nu=0$, hence $\int_{Y}g_{x}(y)=0$ for a.e.~$x\in X$. The claim follows from definiteness of the norm. 
\end{proof}

\begin{ntn}
If we fix a representative of $W_{t}\in L^{2}(X,\mathcal{H})$, we again denote it by $W_{t}$.
\end{ntn}

\begin{lem}\label{LEM.Rcft_Msrbl}
If $\mu_{t}$ is an admissible path and we fix a representative of $W_{t}\in L^{2}(X,\mathcal{H})$ for each $t\in [0,1]$, then 
\begin{itemize}
\item[1)] $\langle W_{s}(x),M_{P_{s}(x)}^{\frac{1}{2}}\partial_{x}F(x)\rangle_{\mathcal{H}}$ is measurable on $[0,1]\times X$ for each $F\in\mathfrak{A}$,
\item[2)] $T\longmapsto\int_{0}^{t}\langle W_{s}(x),M_{P_{s}(x)}^{\frac{1}{2}}\partial_{x}T\rangle_{\mathcal{H}}ds$ defines a unique $\tilde{P}_{t}(x)\in\mathcal{S}_{1}(H)$ for each $t\in [0,1]$ for a.e.~$x\in X$,
\item[3)] $t\longmapsto\tilde{P}_{t}(x)$ is $w^{*}$-continuous on $[0,1]$ for a.e.~$x\in X$,
\item[4)] $x\longmapsto\tilde{P}_{t}(x)\in\mathcal{S}_{1}(H)$ is strongly measurable w.r.t.~the $||.||_{\mathcal{S}_{1}(H)}$-topology for each $t\in [0,1]$. 
\end{itemize}
\end{lem}
\begin{proof}
From $W_{s}\in L^{2}(X,\mathcal{H})$, $P_{s}\in\mathcal{D}$ and existence of a separable function, we know 
\begin{align*}
x\longmapsto\langle W_{s}(x),M_{P_{s}(x)}^{\frac{1}{2}}\partial_{x}F(x)\rangle_{\mathcal{H}}
\end{align*}

\noindent to lie in $L^{1}(X)$ for each $s\in [0,1]$. Hence we need $s\longmapsto \langle W_{s}(\hspace{0.05cm}.\hspace{0.05cm}),M_{P_{s}(\hspace{0.05cm}.\hspace{0.05cm})}^{\frac{1}{2}}\partial_{\hspace{0.05cm}.\hspace{0.05cm}}F(\hspace{0.05cm}.\hspace{0.05cm})\rangle_{\mathcal{H}}$ to be strongly measurable w.r.t.~the $||.||_{L^{1}(X)}$-topology to obtain the first statement. To see this, we test on continuous bounded functions and extend to $L^{\infty}$-functions. For $g\in C_{b}(X)$ and $F\in\mathfrak{A}$, $gF\in\mathfrak{A}$. We thus have

\begin{align*}
\int_{X}g(x)\langle W_{s}(x),M_{P_{s}(x)}^{\frac{1}{2}}\partial_{x}F(x)\rangle_{\mathcal{H}}d\nu=\langle W_{s},M_{P_{s}}^{\frac{1}{2}}\partial (gF)\rangle_{L^{2}(X,\mathcal{H})}=\frac{d}{dt}_{|t=s}\tau(P_{t}gF)
\end{align*}

\noindent which is a measurable map on $[0,1]$. For $g\in L^{\infty}(X)$, we use density of $C_{b}(X)\subset L^{\infty}(X)$ to approximate pointwise by measurable maps. This proves the first statement.\par
For the second one, we assume $P_{t}(x)\geq 0$ for each $(t,x)\in [0,1]\times X$ without loss of generality by Corollary \ref{COR.L1_Pstv} and Lemma \ref{LEM.Msrbl_Unbd}. Using Proposition \ref{PRP.Mass_Prsv_Unbd} and letting $E=\mathbb{C}$ in Lemma \ref{LEM.Rctf_NS}, we know $\textrm{tr}(P_{t}(x))=\textrm{tr}(P_{0}(x))$ to hold for each $t\in [0,1]$ for a.e.~$x\in X$. Given arbitrary $T\in\FinRk(H)$ and $x\in X$, choose $f\in C_{c}(X)$ with $f(x)=1$. Then

\begin{align*}
s\longmapsto\langle W_{s}(x),M_{P_{s}(x)}^{\frac{1}{2}}\partial_{x}T\rangle_{\mathcal{H}}=\langle W_{s}(x),M_{P_{s}(x)}^{\frac{1}{2}}\partial_{x}(f\odot T)(x)\rangle_{\mathcal{H}}
\end{align*}

\noindent is measurable on $[0,1]$. Furthermore, we estimate

\begin{align*}
|\langle W_{s}(x),M_{P_{s}(x)}^{\frac{1}{2}}\partial_{x}T\rangle_{\mathcal{H}}|\leq ||W_{s}(x)||_{\mathcal{H}}\textrm{tr}(P_{0}(x))||\partial_{x}||\hspace{0.025cm}||T||_{\mathcal{K}(H)}
\end{align*}

\noindent for each $s\in [0,1]$. Furthermore, we have $||W_{\hspace{0.025cm}.\hspace{0.025cm}}(\hspace{0.025cm}.\hspace{0.025cm})||_{\mathcal{H}}\in L^{2}([0,1],L^{2}(X))=L^{2}([0,1]\times X)$ since $\mu_{t}$ is admissible. Thus $W_{t}(x)$ is measurable in $t$ on $[0,1]$ for a.e.~$x\in X$. Using this and our previous estimate, we see that

\begin{align*}
T\longmapsto\int_{0}^{t}\langle W_{s}(x),M_{P_{s}(x)}^{\frac{1}{2}}\partial_{x}T\rangle_{\mathcal{H}}ds
\end{align*}

\noindent defines a unique element in $\mathcal{K}(H)^{*}=\mathcal{S}_{1}(H)$ by density of $\FinRk(H)\subset\mathcal{K}(H)$. Continuity when tested on finite rank operators holds by construction, while the estimate just above shows $\sup_{t\in [0,1]}||P_{t}(x)||_{\mathcal{S}_{1}(H)}$ to be finite. From this, the third statement follows.\par
We turn to the last statement. Let $f_{i}\in C_{c}(X)$ be an approximation of $1_{X}$ in $C_{b}(X)$ and choose arbitrary $T\in\FinRk(H)$. Then $\textrm{tr}(\tilde{P}_{t}(x)T)=\lim_{i}\textrm{tr}(\tilde{P}_{t}(x)f_{i}(x)T)$ for each $x\in X$. This lemma's first two statements show 

\begin{align*}
x\longmapsto\textrm{tr}(\tilde{P}_{t}(x)f_{i}(x)T)=f_{i}(x)\int_{0}^{t}\langle W_{s}(x),M_{P_{s}(x)}^{\frac{1}{2}}\partial_{x}T\rangle_{\mathcal{H}}ds.
\end{align*}

\noindent In particular, the map above is measurable. Using this and density of $\FinRk(H)$ w.r.t.~the strong operator-topology, our last claim follows from pointwise approximation by measurable maps.
\end{proof}

\begin{lem}\label{LEM.Rctf}
If $\mu_{t}\in\mathcal{A}(P,Q)$ and we fix a representative of $W_{t}\in L^{2}(X,\mathcal{H})$ for each $t\in [0,1]$, then there exists a representative $P_{t}^{rct}$ in $P_{t}\in L^{1}([0,t]\times X,\mathcal{S}_{1}(H))$ such that																																					
\begin{itemize}
\item[1)]$P_{t}=\tilde{P}_{t}+P=:P_{t}^{rct}\in L^{1}([0,1]\times X,\mathcal{S}_{1}(H))$,
\item[2)]$\mu_{t}=(P_{t}^{rct},W_{t})$,
\item[3)]$\mu_{t}^{rct}(x):=(\theta_{P}^{2}(x)P_{t}^{rct}(x),\theta_{P}(x)W_{t}(x))\in\mathcal{A}(\theta_{P}^{2}(x)P(x),\theta_{P}^{2}(x)Q(x))$ for a.e.~$x\in X$.
\end{itemize}
\end{lem}
\begin{proof}
Let $\tilde{P}_{t}$ be as in Lemma \ref{LEM.Rcft_Msrbl}. For all $F\in\mathfrak{A}$, the same lemma implies

\begin{align*}
\tau((P_{t}-P)F)&=\int_{X}\textrm{tr}((P_{t}-P_{0})(x)F(x))d\nu\\ &=\int_{0}^{t}\int_{X}\langle W_{s}(x),M_{P_{s}(x)}^{\frac{1}{2}}\partial_{x}F(x)\rangle_{\mathcal{H}}d\nu\hspace{0.025cm}ds\\
&=\int_{X}\int_{0}^{t}\langle W_{s}(x),M_{P_{s}(x)}^{\frac{1}{2}}\partial_{x}F(x)\rangle_{\mathcal{H}}ds\hspace{0.02cm}d\nu\\
&=\int_{X}\textrm{tr}(\tilde{P}_{t}(x)F(x))d\nu.
\end{align*}

\noindent We were able to use Fubini-Tonelli in the third equality since the integrated function was shown to be measurable and $\mu_{t}$ has finite energy. By definition, $F=f\odot T$ for some $f\in C_{c}(X)$ and some $T\in\FinRk(H)$. We know $\textrm{tr}(\tilde{P}_{t}(x)T)\in L^{1}(X,d\nu)$ and $\int_{X}f(x)\textrm{tr}((P_{t}-P)(x)T)d\nu=\int_{X}f(x)\textrm{tr}(\tilde{P}_{t}(x)T)d\nu$. If we fix $t\in [0,1]$, this shows $\textrm{tr}((P_{t}-P)(x)T)=\textrm{tr}(\tilde{P}_{t}(x)T)$ for a.e.~$x\in X$. By density and countability of $\FinRk(H)$, we thus have $(P_{t}-P)(x)=\tilde{P}_{t}(x)$ for a.e.~$x\in X$ once we fixed a $t\in [0,1]$.\par
Using this, we know $P_{t}=\tilde{P}_{t}+P$ as measurable maps on $X$ modulo nullsets for each $t\in [0,1]$. By $4)$ in \ref{LEM.Rcft_Msrbl}
and $P_{t}\in\mathcal{D}$, $\tilde{P}_{t}\in L^{1}([0,1]\times X,\mathcal{S}_{1}(H))$. Applying Lemma \ref{LEM.Rctf_NS} with $E=\mathbb{C}$, we have $\tilde{P}_{t}(x)=(P_{t}-P)(x)$ for each $t\in [0,1]$ for a.e.~$x\in X$.  This and $3)$ in Lemma \ref{LEM.Rcft_Msrbl} shows $\textrm{tr}(\tilde{P}_{t})=0$ for each $t\in [0,1]$ for a.e.~$x\in X$. We set $P_{t}^{rct}(x):=\tilde{P}_{t}(x)+P(x)$, which is positive for all $t\in [0,1]$ for a.e.~$x\in X$ by $3)$ in Lemma \ref{LEM.Rcft_Msrbl} and $P_{t}^{rct}=P_{t}\in L^{1}([0,1]\times X,\mathcal{S}_{1}(H))$. Our remaining claims follow by construction of $\tilde{P}_{t}$. 
\end{proof}

\begin{cor}\label{COR.Rctf_DblInt}
If $\mu_{t}\in\mathcal{A}(P,Q)$, then $E(\mu_{t})=\int_{X}E(\mu_{t}^{rct}(x))dsd\nu_{P}$.
\end{cor}
\begin{proof}
\begin{align*}
E(\mu)=\frac{1}{2}\int_{0}^{1}\int_{X}||W_{t}(x)||^{2}_{\mathcal{H}}d\nu\hspace{0.025cm}dt=\int_{X}\frac{1}{2}\int_{0}^{1}||\theta_{P}(x)W_{t}(x)||_{\mathcal{H}}^{2}dt\hspace{0.025cm}d\nu_{P}=\int_{X}E(\mu_{t}^{rct}(x))dsd\nu_{P}.
\end{align*}
\end{proof}

We next focus on conditions allowing us to integrate a measurable selection of minimisers. Corollary \ref{COR.Rctf_DblInt}, i.e.~
the mean energy representation, ensures the resulting path to be a minimiser itself.

\begin{lem}\label{LEM.Pos_Appr}
Let $F:X\longrightarrow\mathcal{S}_{2}(H)_{+}$ be a strongly measurable function with $\textrm{tr}(F^{2}(x))=1$ for a.e.~$x\in X$. Then there exist simple functions $(S_{i})_{i\in\mathbb{N}}$ converging to $F^{2}$ pointwise a.e.~in the $||.||_{\mathcal{S}_{1}(H)}$-topology such that for all $i\in\mathbb{N}$ and a.e.~$x\in X$, $S_{i}(x)\in\mathcal{S}_{1}(H)_{+}$ and $\textrm{tr}(S_{i}(x))=1$. 
\end{lem}
\begin{proof}
As $F$ is strongly measurable, there exist simple functions $H_{i}$ converging to $F$ pointwise a.e.~in the $||.||_{\mathcal{S}_{2}(H)}$-topology. Setting $(H_{i})_{+}(x):=(H_{i}(x))_{+}$, we again obtain a simple function. By Lemma \ref{LEM.L2_Pst_Prj}, we know that $(H_{i})_{+}(x)$ is the metric projection of $H_{i}(x)$ onto the positive cone in $\mathcal{S}_{2}(H)$. Since $F(x)\geq 0$ a.e., this shows $(H_{i})_{+}$ to converge pointwise a.e.~to $F$ in the $||.||_{\mathcal{S}_{2}(H)}$-topology. Thus assume $H_{i}(x)\geq 0$ for a.e.~$x\in X$ and let $p\in\mathcal{S}_{1}(H)$ be a density matrix. We set $G_{i}(x):=H_{i}(x)+1_{X\setminus H_{i}^{-1}(0)}(x)p$ to obtain a sequence of simple functions $G_{i}$ that is non-zero for all $x\in X$. As $H_{i}$ converges pointwise a.e.~to $F$ and $F(x)\neq 0$ a.e., $1_{X\setminus H_{i}^{-1}(0)}(x)$ converges to zero for a.e.~$x\in X$.\par
We norm $G_{i}^{2}$ to obtain the required approximation. Consider the sequence of positive simple functions given by $S_{i}(x):=\textrm{tr}(G_{i}^{2}(x))^{-1}G_{i}^{2}(x)$ for each $x\in X$. This is well-defined because each $G_{i}$ has full support by construction. Pointwise convergence of $G_{i}(x)$ to $F(x)$ in $\mathcal{S}_{2}(H)$ for a.e.~$x\in X$ yields pointwise convergence of $G_{i}^{2}(x)$ to $F^{2}(x)$ in $\mathcal{S}_{1}(H)$. To see this, apply H\"older. From this, we have a.e.~convergence of $\sqrt{\textrm{tr}(G_{i}^{2}(x))}=||G_{i}(x)||_{\mathcal{S}_{2}(H)}$ to $||F(x)||_{\mathcal{S}_{2}(H)}=1$ and thus obtain a sequence $S_{i}$ as required. 
\end{proof}

We establish an appropriate setting for applying the measurable selection theorem. Recall our construction of the distance $d$ on $\mathcal{S}_{cl}(\mathcal{K}(H))$ and $D$ on $C([0,1],\mathcal{S}_{cl}(\mathcal{K}(H)))$ in Subsection 3.3. For $0<C<\infty$, let $\textrm{Lip}_{C}:=\{f:[0,1]\longrightarrow \mathcal{S}_{cl}(\mathcal{K}(H))\ |\ f\ d\textrm{-Lipschitz}\ \textrm{with}\ ||f||_{\textrm{Lip}}\leq C\}$ and equip it with the restriction of $D$. Then $\textrm{Lip}_{C}\subset C([0,1],\mathcal{S}_{cl}(\mathcal{K}(H))$ isometrically by construction. Arzel\'a-Ascoli immediately shows $(\textrm{Lip}_{C},D)$ to be a compact metric space. In particular, $\textrm{Lip}_{C}$ is closed, complete and separable. If $p,q$ are two density matrices with finite distance, then $\mathcal{M}(p,q)\subset\bigcup_{n=1}^{\infty}\textrm{Lip}_{n}$.\par
Let $\partial$ be a vertical gradient such that $\partial_{x}$ has continuous dependence of minimisers on start- and endpoints for a.e.~$x\in X$, and $P,Q\in\mathcal{D}$ with $\textrm{tr}(P(x))=\textrm{tr}(Q(x))$ for a.e.~$x\in X$. Apply Lemma \ref{LEM.Pos_Appr} for $F=\theta_{P}\sqrt{P}$ to find a sequence $P_{i}$ of simple functions converging to $\theta_{P}^{2}P$, and do the same to have a sequence $Q_{i}$ converging to $\theta_{P}^{2}Q$. For $x\in X$ and $\mu_{t}\in\mathcal{A}(\theta_{P}(x)^{2}P(x),\theta_{P}(x)^{2}Q(x))$, set

\begin{align*}
N_{\mu_{t}}:=\{\ (\mu_{t}^{i_{k}})_{k\in\mathbb{N}}\ |\ \mu_{t}^{i_{k}}\in\mathcal{M}(P_{i_{k}}(x),Q_{i_{k}}(x))\ \textrm{s.t.}\ \lim_{k\in\mathbb{N}} D(\mu_{t}^{i_{k}},\mu_{t})=0\}
\end{align*}

\noindent and define a multifunction from $X$ to $(C([0,1],\mathcal{S}_{cl}(\mathcal{K}(H))),D)$ by 

\begin{align*}
\psi_{P,Q}(x):=\{\mu_{t}\in\mathcal{M}(\theta_{P}(x)^{2}P(x),\theta_{P}(x)^{2}Q(x))\ |\ N_{\mu_{t}}\neq\emptyset\}.
\end{align*}

\noindent Continuous dependence of minimisers on start- and endpoints ensures $\psi_{P,Q}(x)\neq\emptyset$ a.e., while closedness of $\psi_{P,Q}(x)$ in $(C([0,1],\mathcal{S}_{cl}(\mathcal{K}(H))),D)$ follows by closedness of $\mathcal{M}(\theta_{P}(x)^{2}P(x),\theta_{P}(x)^{2}Q(x))$ and construction of $\psi_{P,Q}(x)$. We are in the setting of Theorem 6.9.3 in \cite{BK.Bog_MsrThry}. We therefore obtain a measurable selection of minimisers if for all open $U\subset C([0,1],\mathcal{S}_{cl}(\mathcal{K}(H))$, the sets 

\begin{align*}
\hat{\psi}_{P,Q}(U):=\{x\in X\ |\ \psi_{P,Q}(x)\cap U\neq\emptyset\}
\end{align*}

\noindent are measurable. 

\begin{rem}
Each $\psi_{P,Q}$ depends not only on $P$ and $Q$, but also $P^{i}$ and $Q^{i}$. These dependencies do not matter as we seek \textit{some} measurable selection for fixed $P$ and $Q$.
\end{rem}

\begin{lem}\label{LEM.Msrbl_Slct}
Let $\partial$ be a vertical gradient such that $\partial_{x}$ has continuous dependence of minimisers on start- and endpoints for a.e.~$x\in X$. If $P,Q\in\mathcal{D}$ with $\textrm{tr}(P(x))=\textrm{tr}(Q(x))$ for a.e.~$x\in X$, then $\hat{\psi}(U)$ is measurable for each open $U\subset C([0,1],\mathcal{S}_{cl}(\mathcal{K}(H))$.
\end{lem}
\begin{proof}
Without loss of generality, we replace 'almost everywhere' with 'everywhere' in this lemma's assumptions. Let $U\subset C([0,1],\mathcal{S}_{cl}(\mathcal{K}(H))$ be open. Then 

\begin{align*}
\psi_{P,Q}(x)\cap U=\bigcup_{n=1}^{\infty}\Big(\psi_{P,Q}(x)\cap U\cap\textrm{Lip}_{n}\Big)
\end{align*}

\noindent since each admissible path lies in some $\textrm{Lip}_{C}$. However, $U\cap\textrm{Lip}_{n}$ is open in the relative topology of $\textrm{Lip}_{n}$ because $\textrm{Lip}_{n}\subset C([0,1],\mathcal{S}_{cl}(\mathcal{K}(H))$ is closed. We thus have

\begin{align*}
\hat{\psi}_{P,Q}(U)=\bigcup_{n=1}^{\infty}\{x\in X\ |\ \psi_{P,Q}(x)\cap U\cap\textrm{Lip}_{n}\}
\end{align*}

\noindent and are left to check measurability of the sets on the right-hand side.\par
Reducing notational overhead, we consider an open set $U\subset\textrm{Lip}_{C}$ for some $0<C<\infty$. Since $\textrm{Lip}_{C}$ is separable, there exists a countable set of open balls covering $U$. Our statement follows if for all $f\in\textrm{Lip}_{C}$ and all $\varepsilon>0$, the set

\begin{align*}
\hat{\psi}_{P,Q,C}(B_{\varepsilon}(f)):=\{x\in X\ |\ \psi_{P,Q}(x)\cap B_{\varepsilon}(f)\}
\end{align*}  

\noindent is measurable. We claim that

\begin{align*}
\hat{\psi}_{P,Q,C}(B_{\varepsilon}(f))=\bigcup_{j=1}^{\infty}\bigcup_{k=1}^{\infty}\bigcap_{i=k}^{\infty}\{x\in X\ |\ \mathcal{M}(P_{i}(x),Q_{i}(x))\cap B_{\varepsilon-j^{-1}}(f)\neq\emptyset\}.
\end{align*}

Of course, $B_{\varepsilon-j^{-1}}(f)=\emptyset$ if $\varepsilon\leq j^{-1}$ holds. Let $x\in\hat{\psi}_{P,Q,C}(B_{\varepsilon}(f))$ and choose a $\mu_{t}\in\psi_{P,Q}(x)\cap B_{\varepsilon}(f)$. Pick a $j\in\mathbb{N}$ such that $\mu_{t}\in B_{\varepsilon-j^{-1}}(f)$. By definition of $\psi_{P,Q}(x)$ and the triangle inequality, there exist some $j_{0}\geq j$ and $i_{0}\in\mathbb{N}$ such that $\mu_{t}^{i}\in B_{\varepsilon-j_{0}^{-1}}(f)$ for all $i\geq i_{0}$. Hence

\begin{align*}
x\in\bigcap_{i=i_{0}}^{\infty}\{x\in X\ |\ \mathcal{M}(P_{i}(x),Q_{i}(x))\cap B_{\varepsilon-j_{0}^{-1}}(f)\neq\emptyset\}.
\end{align*}

\noindent showing one direction. For the converse, choose an arbitrary $x\in \bigcap_{i=k}^{\infty}\{x\in X\ |\ \mathcal{M}(P_{i}(x),Q_{i}(x))\cap B_{\varepsilon-j^{-1}}(f)\neq\emptyset\}$. By hypothesis, we have a sequence of minimisers $\mu_{t}^{i}\in\mathcal{M}(P_{i}(x),Q_{i}(x))\cap\textrm{Lip}_{C}$ such that $D(f,\mu_{t}^{i})<\varepsilon-j^{-1}$ for each $i\in\mathbb{N}$. As $\textrm{Lip}_{C}$ is compact, we extract a $D$-converging subsequence $\mu_{t}^{i_{k}}$ which we know must lie in $B_{\varepsilon-j^{-1}}(f)$. By $2)$ in Definition \ref{DFN.Cnt_Dpd}, we see the limit to be a $\mu_{t}\in\mathcal{M}(\theta_{P}(x)^{2}P(x),\theta_{P}(x)^{2}Q(x))$. Since $j>0$, $\mu_{t}\in\psi_{P,Q}(x)\cap B_{\varepsilon}(f)$ and therefore $x\in\hat{\psi}_{P,Q,C}(B_{\varepsilon}(f))$.\par
To conclude, we show each $\{x\in X\ |\ \mathcal{M}(P_{i}(x),Q_{i}(x))\cap B_{\varepsilon}(f)\neq\emptyset\}$ to be measurable. Write $P_{i}=\sum_{j=1}^{n_{i}}1_{A_{ij}}p_{ij}$, $Q_{i}=\sum_{j=1}^{m_{i}}1_{B_{ij}}q_{ij}$, and choose a finite sub-partition $C_{ij}$ of $X$ for the finite partitions $A_{ij}$ and $B_{ij}$ of $X$ such that each $C_{ij}$ is measurable. Then write

\begin{align*}
P_{i}=\sum_{j=1}^{k_{i}}1_{C_{ij}}\tilde{p}_{ij},\ Q_{i}=\sum_{j=1}^{k_{i}}1_{C_{ij}}\tilde{q}_{ij}
\end{align*}

\noindent where $\tilde{p}_{ij}$ or $\tilde{q}_{ij}$ might remain the same upon varying the $j$-variable. The latter is not relevant for this proof. From the representation above, we see $\mathcal{M}(P_{i}(x),Q_{i}(x))\cap B_{\varepsilon}(f)\neq\emptyset$ for $x\in C_{ij}$ if and only if it holds true for all $x\in C_{ij}$. Thus $\{x\in X\ |\ \mathcal{M}(P_{i}(x),Q_{i}(x))\cap B_{\varepsilon}(f)\neq\emptyset\}$ is given by a finite union of some $C_{ij}$, hence measurable.
\end{proof}

\begin{ntn}
For all $x\in X$, the $L^{2}$-Wasserstein distance on $\mathcal{D}_{b}\subset\mathcal{S}_{1}(H)_{+}$ arising from $\partial_{x}$ is $\mathcal{W}_{2,x}$.
\end{ntn}

\begin{thm}\label{THM.Disint}
Let $\partial$ be a vertical gradient such that $\partial_{x}$ has continuous dependence of minimisers on start- and endpoints for a.e.~$x\in X$. For all $P,Q\in\mathcal{D}$ with finite distance, we have 

\begin{align*}
\mathcal{W}_{2}^{2}(P,Q)=\int_{X}\mathcal{W}_{2,x}^{2}(\theta_{P}(x)^{2}P(x),\theta_{P}(x)^{2}Q(x))d\nu_{P}
\end{align*}

\noindent and there exists a minimiser $\mu_{t}$ of $\mathcal{W}_{2}(P,Q)$ such that $\theta_{P}(x)^{2}\mu_{t}(x)\in\mathcal{M}(\theta_{P}(x)^{2}P(x),\theta_{P}(x)^{2}Q(x))$ for a.e.~$x\in X$. 
\end{thm}
\begin{proof}
By finiteness of $\mathcal{W}_{2}(P,Q)$, we have $\textrm{tr}(P(x))=\textrm{tr}(Q(x))$ for a.e.~$x\in X$. Lemma \ref{LEM.Msrbl_Slct} yields a measurable selection of minimisers $\xi$. By construction, $\xi(x)\in\mathcal{M}(\theta_{P}(x)^{2}P(x),\theta_{P}(x)^{2}Q(x))$ for a.e.~$x\in X$. We claim $(t,x)\longmapsto\xi(x)(t)$ to be a strongly measurable map from $[0,1]\times X$ to $(\mathcal{S}_{1}(H),||.||_{\mathcal{S}_{1}(H)})$. For all $T_{k}$ as in the beginning of Subsection 3.3, evaluation at $T_{k}$ is a continuous map from $C([0,1],\mathcal{S}_{cl}(\mathcal{K}(H))$ to $C([0,1])$. Since $(T_{k})_{k\in\mathbb{N}}\subset\mathcal{B}(H)$ $w^{*}$-densely, pointwise approximation shows evaluation at each $T\in\mathcal{B}(H)$ to be strongly measurable. Moreover, $\xi$ is measurable w.r.t.~$\mathcal{B}(X)$ and $\mathcal{B}(C([0,1],\mathcal{S}_{cl}(\mathcal{K}(H))))$ by construction. Taken together, this implies measurability of $\textrm{tr}(\xi(x)(t)T)$ on $[0,1]\times X$ for each $T\in\mathcal{B}(H)$. As $\mathcal{S}_{1}(H)$ is separable, this proves the claim.\par
Furthermore, we proved strong measurability of $P_{t}(x):=\textrm{tr}(P(x))\xi(x)(t)$ as a map from $[0,1]\times X$ to $\mathcal{S}_{1}(H)$. By construction of $\xi$, we have $P_{t}(x)\geq 0$ with $||P_{t}(x)||_{\mathcal{S}_{1}(H)}=\textrm{tr}(P(x))$ for each $t\in [0,1]$ for a.e.~$x\in X$. Hence $P_{t}\in\mathcal{D}$ for each $t\in [0,1]$. Let $t\longmapsto w_{t}(x)$ be the vector field associated to the admissible path $\xi(x)$. By strong measurability of $\xi$, the maps

\begin{align*}
\textrm{tr}(P(x))f(x)\frac{d}{dt}\textrm{tr}(\xi(x)(t)T)=\langle\textrm{tr}(P(x))^{\frac{1}{2}}w_{t}(x),M_{P_{t}(x)}^{\frac{1}{2}}\partial_{x}(f\odot T)(x)\rangle_{\mathcal{H}}
\end{align*}

\noindent are measurable on $[0,1]\times X$ for each $f\odot T\in\mathfrak{A}$. Furthermore, $M_{\xi_{t}(x)}$ and $\partial_{x}T$ are measurable on $[0,1]\times X$ for each $T\in\FinRk(H)$. Hence each $||T||_{\xi_{t}(x)}$ is measurable on $[0,1]\times X$ as well. This in turn implies 

\begin{align*}
||w_{t}(x)||_{\mathcal{H}}=\sup_{T\in\FinRk(H)}||T||_{\xi_{t}(x)}^{-1}\langle w_{t}(x),M_{\xi(t)(x)}^{\frac{1}{2}}\partial_{x}T\rangle_{H}
\end{align*}

\noindent Yet $\FinRk(H)$ is an extension algebra for each fibre gradient, see Remark \ref{REM.Curveball}. As a pointwise limit of measurable maps on $[0,1]\times X$, $||w_{t}(x)||_{\mathcal{H}}$ is therefore measurable on $[0,1]\times X$.\par
In particular, $\mathcal{W}_{2,x}^{2}(\theta_{P}(x)^{2}P(x),\theta_{P}(x)^{2}Q(x))=\int_{0}^{1}||w_{s}(x)||_{\mathcal{H}}^{2}ds$ is measurable on $X$. We therefore know

\begin{align*}
\int_{X}\mathcal{W}_{2,x}^{2}(\theta_{P}(x)^{2}P(x),\theta_{P}(x)^{2}Q(x))d\nu_{P}\leq\mathcal{W}_{2}^{2}(P,Q)<\infty
\end{align*}

\noindent by Corollary \ref{COR.Rctf_DblInt} and construction of $\xi$. This implies integrability of $\textrm{tr}(P(x))||w_{t}(x)||_{\mathcal{H}}^{2}$ on $[0,1]\times X$. We set $W_{t}(x):=\textrm{tr}(P(x))^{\frac{1}{2}}w_{t}(x)$, and $||W_{t}(x)||_{\mathcal{H}}^{2}$ is integrable by what we just showed. By the above, integrability of $||W_{t}(x)||_{\mathcal{H}}^{2}$, and the separating function $g(f\odot T)=||\partial f\otimes T||_{\infty}$, we have

\begin{align*}
\int_{X}\textrm{tr}((P_{t}-P)(x)(f\odot T)(x))d\nu&=\int_{X}\int_{0}^{t}\langle W_{s}(x),M_{P_{s}(x)}^{\frac{1}{2}}\partial_{x}(f\odot T)(x)\rangle_{\mathcal{H}}dsd\nu\\
&=\int_{0}^{t}\int_{X}\langle W_{s}(x),M_{P_{s}(x)}^{\frac{1}{2}}\partial_{x}(f\odot T)(x)\rangle_{\mathcal{H}}d\nu ds
\end{align*}

\noindent for each $f\odot T\in\mathfrak{A}$. If $R_{t}$ is the projection onto $T_{P_{t}}\mathcal{D}\subset L^{2}(X,\mathcal{S}_{2}(H))$, then $\mu_{t}=(P_{t},R_{t}(W_{t}))$ is an admissible path by what we showed just now. On the other hand, we have

\begin{align*}
\int_{0}^{1}||R_{s}(W_{s})||_{L^{2}(X,\mathcal{S}_{2}(H))}^{2}ds\leq \int_{0}^{1}||W_{s}||_{L^{2}(X,\mathcal{S}_{2}(H))}^{2}=\int_{X}\mathcal{W}_{2,x}^{2}(\theta_{P}(x)^{2}P(x),\theta_{P}(x)^{2}Q(x))d\nu_{P}
\end{align*}

\noindent and we see $\mu_{t}$ to be a minimiser as required. The statement follows.
\end{proof}

\begin{cor}\label{COR.Disint}
Let $\partial$ be a vertical gradient and $H$ finite-dimensional. For all $P,Q\in\mathcal{D}$ with $\textrm{tr}(P(x))=\textrm{tr}(Q(x))$ for a.e.~$x\in X$, we have 

\begin{align*}
\mathcal{W}_{2}^{2}(P,Q)=\int_{X}\mathcal{W}_{2,x}^{2}(\theta_{P}(x)^{2}P(x),\theta_{P}(x)^{2}Q(x))d\nu_{P}
\end{align*}

\noindent and there exists a minimiser $\mu_{t}$ of $\mathcal{W}_{2}(P,Q)$ such that $\theta_{P}(x)^{2}\mu_{t}(x)\in\mathcal{M}(\theta_{P}(x)^{2}P(x),\theta_{P}(x)^{2}Q(x))$ for a.e.~$x\in X$. 
\end{cor}
\begin{proof}
Using Lemma \ref{LEM.FinDim_CntDpd_II}, we see each $\partial_{x}$ to have continuous dependence of minimisers on start- and endpoints. By Proposition \ref{PRP.FinDim_WkMetr}, each $\mathcal{W}_{2,x}$ has finite diameter. This implies $\mathcal{W}_{2}(P,Q)<\infty$ if $\textrm{tr}(P(x))=\textrm{tr}(Q(x))$ for a.e.~$x\in X$ since $\textrm{tr}(P(x))$ is integrable. We apply Theorem \ref{THM.Disint} to conclude.
\end{proof}

As an end to this subsection, we give a toy application of Theorem \ref{THM.Disint}. Let $H$ be finite-dimensional and $\nu$ a probability measure. We view density matrices as modelling a system's states and assume to be given a vertical gradient $\partial$. Minimisers for $\mathcal{W}_{2,x}(p,q)$ describe all possible ways for the system to evolve from $p$ to $q$ under the cost, hence geometry, determined by $\partial_{x}$.\par
We are in the following situation: if the system changes states, then it evolves along a minimiser determined by \textit{some} $\partial_{x}$. However, we are unable to say which $\partial_{x}$ is chosen for any particular state change. We hope to find an average evolution. For $p,q\in\mathcal{S}_{1}(H)_{+}$ density matrices, set $P(x):=p$ and $Q(x):=q$. Both $P,Q\in\mathcal{D}$ with $\textrm{tr}(P(x))=\textrm{tr}(Q(x))=1$, hence $d\nu_{P}=d\nu$. By Corollary \ref{COR.Disint}, there exists a minimiser $\mu_{t}\in\mathcal{M}(P,Q)$ and we have

\begin{align*}
\mathcal{W}_{2}^{2}(P,Q)=\int_{X}\mathcal{W}_{2,x}^{2}(p,q)\hspace{0.025cm}d\nu.
\end{align*}

\noindent As $\mu_{t}(x)\in\mathcal{M}(p,q)$, we consider $\mu_{t}$ to be an average evolution (see FIG. \ref{FIG.1} and FIG. \ref{FIG.2}).\par
Another application of Theorem \ref{THM.Disint} and its proof will be presented in the next subsection in form of mean entropic curvature bounds.

\begin{figure}
\begin{center}
\begin{tikzpicture}
\node at (1.81,0.583) {$p$};
\node at (14.125,-0.445) {$q$};
\draw [black,fill=black] (1.935,0.483) circle (0.03cm);
\draw [black,fill=black] (14,-0.305) circle (0.03cm);
\draw [dashed] (2,0.5) .. node[above, yshift=-0.05cm] (ref) {\small{$\mu_{t}(x_{1})$}} controls (8,3.25) .. (13.935,-0.278);
\draw [dashed] (2,0.5) .. node[above, yshift=-0.05cm] {\small{$\mu_{t}(x_{2})$}} controls (8,2.375) .. (13.935,-0.278);
\draw [dashed] (2,0.5) .. node[above, yshift=-0.05cm] {\small{$\mu_{t}(x_{3})$}} node[below] {$\vdots$} controls (8,1.5) .. (13.935,-0.278);
\node [above = 0cm of ref] {$\vdots$};
\end{tikzpicture}
\caption{}
\label{FIG.1}
\end{center}
\end{figure}

\begin{figure}
\begin{center}
\begin{tikzpicture}
\node at (1.81,0.583) {$p$};
\node at (14.125,-0.445) {$q$};
\draw [black,fill=black] (1.935,0.483) circle (0.03cm);
\draw [black,fill=black] (14,-0.305) circle (0.03cm);
\draw [fill=gray!40, draw=none] (2,0.5) .. node[above, yshift=-0.05cm] (ref){} controls (8,3.25) .. (13.935,-0.278);
\draw [line width=0.025cm] (2,0.5) .. node[above] {$\mu_{t}$} controls (8,1.875) .. (13.935,-0.278);
\draw [fill=white, draw=none] (2,0.5) .. node[below] {$\vdots$} controls (8,1.5) .. (13.935,-0.278);
\node [above = 0cm of ref] {$\vdots$};
\end{tikzpicture}
\caption{}
\label{FIG.2}
\end{center}
\end{figure}

\subsection{Mean entropic curvature bounds}

Let $H$ be finite-dimensional and $\partial$ a symmetric gradient for $(M_{n}(\mathbb{C}),\textrm{tr})$. In the classical setting, Sturm introduced entropic curvature bounds for metric measure spaces \cite{SturmGMMSI}. Theorem \ref{THM.Disint} leads us to consider a mean relative entropy, as well as mean entropic curvature bounds. For the latter, we prove a local to global theorem. The next definition uses Proposition \ref{PRP.Entrp_Rpr}, i.e. $\textrm{Ent}(p|\tau)=\tau(p\log p)$ for all density matrices.

\begin{dfn}\label{DFN.Crv_Fibre}
We say that $(M_{n}(\mathbb{C}),\textrm{tr},\partial)$ has curvature $\geq K\in\mathbb{R}$ if for all $p,q\in\mathcal{D}_{b}$, there exists some $\mu_{t}\in\mathcal{M}(p,q)$ such that

\begin{align}
\textrm{tr}(\rho_{t}\log\rho_{t})\leq (1-t)\textrm{tr}(p\log p)+t\textrm{tr}(q\log q)-\frac{K}{2}t(1-t)\mathcal{W}_{2}^{2}(p,q)
\end{align}

\noindent for each $t\in [0,1]$. We set $\curv(A,\tau,\partial):=\sup\{K\in\mathbb{R}\ |\ (M_{n}(\mathbb{C}),\textrm{tr},\partial)\ \textrm{has}\ \textrm{curvature} \geq K\}$, where $\sup\emptyset=-\infty$ as usual.
\end{dfn}

\begin{rem}
By definition, each $(M_{n}(\mathbb{C}),\textrm{tr},\partial)$ has curvature $\geq\curv(M_{n}(\mathbb{C}),\textrm{tr},\partial)$.
\end{rem}

\begin{dfn}
For all $p,q\in\mathcal{D}_{b}$, set $\mathcal{M}(p,q,K):=\{\mu_{t}\in\mathcal{M}(p,q)\ |\ \mu_{t}\ \textrm{satisfies}\ (1)\ \textrm{for}\ K\}$.
\end{dfn}

\begin{lem}\label{LEM.FinDim_CntDpd_III}
Let $\curv(M_{n}(\mathbb{C}),\textrm{tr},\partial)\geq K$. For all $p,q\in\mathcal{D}_{b}$ and all $(p_{i})_{i\in\mathbb{N}},(q_{i})_{i\in\mathbb{N}}\subset\mathcal{D}_{b}$ with $p_{i}\longrightarrow p$, resp.~$q_{i}\longrightarrow q$ in the $||.||_{\mathcal{S}_{1}(H)}$-topology, then

\begin{itemize}
\item[1)] there exist $\mu_{t}\in\mathcal{M}(p,q,K)$ and $\mu_{t}^{i_{k}}\in\mathcal{M}(p_{i_{k}},q_{i_{k}},K)$ with $\lim_{k\in\mathbb{N}}D(\mu_{t}^{i_{k}},\mu_{t})=0$,
\item[2)] the limit of each $D$-converging sequence of $\mu_{t}^{i}\in\mathcal{M}(p,q,K)$ lies in $\mathcal{M}(p,q,K)$. 
\end{itemize}
\end{lem}
\begin{proof}
Convergence w.r.t~the $||.||_{\mathcal{S}_{1}}$-topology implies $\lim\tau(p_{i}\log p_{i})=\tau(p\log p)$ and Proposition \ref{PRP.FinDim_WkMetr} shows convergence of $\mathcal{W}_{2}(p_{i},q_{i})$ to $\mathcal{W}_{2}(p,q)$. Hence $(1)$ is a closed condition, and we argue analogously to our proof of Lemma \ref{LEM.FinDim_CntDpd_II}.
\end{proof}

\begin{prp}\label{PRP.Cnt_Dpd_II}
If $\curv(M_{n}(\mathbb{C}),\textrm{tr},\partial)\geq K$, then $\mathcal{M}(p,q,K)$ is non-empty and closed w.r.t.~$D$ for each $p,q\in\mathcal{D}_{b}$.
\end{prp}
\begin{proof}
By Definition of $\curv(M_{n}(\mathbb{C}),\textrm{tr},\partial)$, $\mathcal{M}(p,q,K)$ is non-empty. Closedness follows from $\mathcal{M}(p,q)$ being closed and $(1)$ being a closed condition. 
\end{proof}

We define the mean relative entropy of a density, as well as an associated synthetic curvature bound condition. The latter is called the mean entropic curvature bound. Lemma \ref{LEM.FinDim_CntDpd_III} and Proposition \ref{PRP.Cnt_Dpd_II} allow us to argue analogously to Lemma  \ref{LEM.Msrbl_Slct}. This yields a theorem similar in spirit to the disintegration theorem. Note that for all $P\in\mathcal{D}$, $x\longmapsto\textrm{tr}(P(x)\log P(x))$ is measurable. 

\begin{dfn}
Let $\partial$ be a vertical gradient for $(C_{0}(X,M_{n}(\mathbb{C})),\nu\otimes\textrm{tr})$. 

\begin{itemize}
\item[1)] For all $P\in\mathcal{D}$, $\textrm{Ent}_{m}(P|\nu\otimes\textrm{tr})=\int_{X}\textrm{tr}(P(x)\log P(x))d\nu\in\mathbb{R}\cup\{\pm\infty\}$ is the mean relative entropy.
\item[2)] We say that $(C_{0}(X,M_{n}(\mathbb{C})),\nu\otimes\textrm{tr},\partial)$ has mean curvature $\geq K\in\mathbb{R}$ if for all $f\in\mathcal{D}(X,\nu)$ and for all $P,Q\in\mathcal{D}_{f}$, there exists a minimiser $\mu_{t}\in\mathcal{A}(P,Q)$ such that 

\begin{align}
\textrm{Ent}_{m}(P_{t}|\nu\otimes\textrm{tr})\leq (1-t)\textrm{Ent}_{m}(P|\nu\otimes\textrm{tr})+t\textrm{Ent}_{m}(Q|\nu\otimes\textrm{tr})-\frac{K}{2}t(1-t)\mathcal{W}_{2}^{2}(P,Q)
\end{align}

\noindent for each $t\in [0,1]$. We set 

\begin{align*}
\mcurv(\nu\otimes\textrm{tr},\partial):=\sup\{K\in\mathbb{R}\ |\ (C_{0}(X,M_{n}(\mathbb{C})),\nu\otimes\textrm{tr},\partial)\ \textrm{has}\ \textrm{mean}\ \textrm{curvature} \geq K\}.
\end{align*}
\end{itemize}
\end{dfn}

\begin{prp}
The mean relative entropy is a convex function on $\mathcal{D}$. If $X$ is compact and $P\in\mathcal{D}_{b}$, then $\textrm{Ent}(P|\nu\otimes\textrm{tr})=\textrm{Ent}_{m}(P|\nu\otimes\textrm{tr})$.
\end{prp}
\begin{proof}
The first statement follows immediately from convexity of the noncommtuative relative entropy. For the second one, use Proposition \ref{PRP.Entrp_Rpr} to see the first equality in

\begin{align*}
\textrm{Ent}(P|\nu\otimes\textrm{tr})=\tau(P\log P)=\int_{X}\textrm{tr}(P(x)\log P(x))d\nu=\textrm{Ent}_{m}(P|\nu\otimes\textrm{tr}).
\end{align*}
\end{proof}

\begin{thm}\label{THM.NC_MCrv}
If $\partial$ is a vertical gradient for $(C_{0}(X,M_{n}(\mathbb{C})),\nu\otimes\textrm{tr})$, then 

\begin{align*}
\mcurv(\nu\otimes\textrm{tr},\partial)\geq \essinf_{x\in X}\curv(M_{n}(\mathbb{C}),\textrm{tr},\partial_{x}).
\end{align*}
\end{thm}
\begin{proof}
If the right-hand side is $-\infty$, there is nothing to show. We therefore assume

\begin{align*}
\essinf_{x\in X}\curv(M_{n}(\mathbb{C}),\textrm{tr},\partial_{x})\geq K
\end{align*}

\noindent for some $K\in\mathbb{R}$. For $f\in\mathcal{D}(X,\nu)$, let $P,Q\in\mathcal{D}_{f}$. By definition, $\textrm{tr}(P(x))=\textrm{tr}(Q(x))$ for a.e.~$x\in X$. By Lemma \ref{LEM.FinDim_CntDpd_III} and Proposition \ref{PRP.Cnt_Dpd_II}, each $\mathcal{M}(p,q,K)$ has the same properties we required of $\mathcal{M}(p,q)$ when proving $\ref{LEM.Msrbl_Slct}$. Arguing as in Theorem \ref{THM.Disint}, we obtain a minimiser $\mu_{t}\in\mathcal{A}(P,Q)$. Our choice of $\mathcal{M}(p,q,K)$ instead of $\mathcal{M}(p,q)$ ensures $\theta_{P}(x)^{2}\mu_{t}(x)\in\mathcal{M}(\theta_{P}(x)^{2}P(x),\theta_{P}(x)^{2}Q(x),K)$.\par
For all density matrices $p$ and $C>0$, we have $\textrm{tr}(Cp\log Cp)=\textrm{tr}(Cp\log C)+\textrm{tr}(Cp\log p)$ by functional calculus and unitality of $M_{n}(\mathbb{C})$. Using this, we obtain

\begin{align*}
\textrm{Ent}_{m}(P_{t}|\nu\otimes\textrm{tr})-\int_{X}\textrm{tr}(P_{t}(x)\log\textrm{tr}(P(x)))d\nu=\int_{X}\textrm{tr}(\theta_{P}(x)^{2}P_{t}(x)\log\theta_{P}(x)^{2}P_{t}(x))d\nu_{P}.
\end{align*}

\noindent Since $\theta_{P}(x)^{2}\mu_{t}(x)$ satisfies $(1)$ with $K$ for a.e.~$x\in X$, the right-hand side of the equation just above is less or equal to

\begin{align*}
&\ (1-t)\int_{X}\textrm{tr}(\theta_{P}(x)^{2}P(x)\log\theta_{P}(x)^{2}P(x))d\nu_{P}\\
+&\ t\int_{X}\textrm{tr}(\theta_{P}(x)^{2}P(x)\log\theta_{P}(x)^{2}Q(x))d\nu_{P}\\
-&\ \frac{K}{2}t(1-t)\int_{X}\mathcal{W}_{2,x}^{2}(\theta_{P}(x)^{2}P(x)),\theta_{P}(x)^{2}Q(x)))d\nu_{P}.
\end{align*}

\noindent The last summand equals $-\frac{K}{2}t(1-t)\mathcal{W}_{2}^{2}(P,Q)$ by Theorem \ref{THM.Disint}. Knowing this, it suffices to add

\begin{align*}
\int_{X}\textrm{tr}(P_{t}(x)\log\textrm{tr}(P(x)))d\nu=(1-t)\int_{X}\textrm{tr}(P_{t}(x)\log\textrm{tr}(P(x)))d\nu+t\int_{X}\textrm{tr}(P_{t}(x)\log\textrm{tr}(P(x)))d\nu
\end{align*}

\noindent to both sides of the esimate we just proved and to use $\textrm{tr}(Cp\log Cp)=\textrm{tr}(Cp\log C)+\textrm{tr}(Cp\log p)$ to show that $\mu_{t}$ satisfies $(2)$ for $\essinf_{x\in X}\curv(M_{n}(\mathbb{C}),\textrm{tr},\partial_{x})$.
\end{proof}

\section{Disintegrating $L^{2}$-Wasserstein distances}

We extend the notion of vertical gradients to $\mathcal{K}(H)$-bundles, introduce the disintegration problem for unital $C^{*}$-algebras and give sufficient conditions for solving it. Finally, we outline plans to achieve disintegration for more general fields of $C^{*}$-algebras.\par
As in the last section, let $X$ be a locally compact Hausdorff space, $\mathcal{B}(X)$ its Borel $\sigma$-algebra, $(X,\mathcal{B}(X))$ a separable measure space and $H$ a separable Hilbert space. Furthermore, let $X$ have a continuous partition of unity for each locally finite open cover. We say that $X$ has sufficiently many continuous partitions of unity. Normal spaces have sufficiently many continuous partitions of unity, see \cite{BK.Quer_MngthTop}. Compact Hausdorff spaces and paracompact topological manifolds are examples of normal spaces.

\subsection{Vertical gradients for $\mathcal{K}(H)$-bundles}

Let $E$ be an hermitian vector bundle over $X$ with fibres $H$. Since $E$ is hermitian, all structure maps of $\textrm{End}(E)$ must be unitary. In particular, $||.||_{\mathcal{B}(H)}$ changes appropriately under structure maps and we are able to define bounded sections accordingly. Finally, structure maps send compact operators to compact operators. We obtain the compact endomorphism bundle $\textrm{End}_{\mathcal{K}}(E)$ with fibres given by $\mathcal{K}(H)$, its structure maps being $^{*}$-homomorphisms. $\Gamma_{0}(\textrm{End}_{\mathcal{K}}(E))$ is the $C^{*}$-algebra of continuous sections vanishing at infinity.

\begin{ntn}
Write $\mathcal{T}(E)$ for the set of trivialising open subsets. If we pick a continuous partition of unity $(\varphi_{i})_{i\in I}$, we demand each $\textrm{supp}\hspace{0.025cm}\varphi_{i}$ to be a subset of some $U_{i}\in\mathcal{T}(E)$.
\end{ntn}

Let $\tau$ be a trace on $\Gamma_{0}(\textrm{End}_{\mathcal{K}}(E))$. Given a continuous partition of unity $(\varphi_{i})_{i\in I}$, surjectivity of the restriction map implies that each $F\longmapsto \tau(\varphi_{i}F_{|U_{i}})$ defines a trace $\tau_{|U_{i}}$ on $\Gamma_{0}(\textrm{End}_{\mathcal{K}}(E_{|U_{i}}))$. This allows a general notion of product trace.

\begin{dfn}
We call a trace $\tau$ on $\Gamma_{0}(\textrm{End}_{\mathcal{K}}(E))$ a product trace if there exists a continuous partition of unity $(\varphi_{i})_{i\in I}$ such that each $\tau_{|U_{i}}$ is a product trace.
\end{dfn}

\begin{prp}\label{PRP.Prd_Tr_All_Part}
If $\tau$ is a product trace, then $\tau_{|U_{i}}$ is a product trace for all continuous partitions of unity $(\varphi_{i})_{i\in I}$.
\end{prp}
\begin{proof}
Choose a continuous parition of unity $(\eta_{i})_{j\in J}$ such that $\tau_{|U_{j}}$ is a product trace. Write $\varphi_{i}=\sum_{j\in J}\eta_{j}\varphi_{i}$ for an arbitrary continuous partition of unity $(\varphi_{i})_{i\in I}$ and set $\chi_{i,j}:=\eta_{j}\varphi_{i}$. Each $(\chi_{i,j})_{(i,j)\in I\times J}$ is itself a continuous partition of unity and each $\tau_{|U_{j}}$ a product trace. Therefore $\tau_{|U_{(i,j)}}$ must be a product trace as well. From this, the statement follows at once.
\end{proof}

\begin{cor}\label{COR.Prd_Tr_Bdl}
If $H$ is finite-dimensional and there exists a continuous partition of unity $(\varphi_{i})_{i\in I}$ such that each $\tau_{|U_{i}}$ is finite, $\tau$ is a product trace.
\end{cor}
\begin{proof}
This follows from the above proposition and Corollary \ref{COR.Prd_Tr} by finite-dimensionality.
\end{proof}

Given a continuous partition of unity $(\varphi_{i})_{i\in I}$ and an $F\in\Gamma_{0}(\textrm{End}_{\mathcal{K}}(E))$, we write

\begin{align}
\tau(F)=\sum_{i\in I}\tau(\varphi_{i}F)=\sum_{i\in I}\int_{U_{i}}\varphi_{i}(x)\textrm{tr}(F_{|U_{i}}(x))d\nu_{i}
\end{align}

\noindent where the right-hand side is independent of our choices since the left-hand side already is. Let $\Gamma_{c}(\textrm{End}_{\mathcal{K}}(E))$ denote the space of continuous section with compact support. 

\begin{dfn}
If $p\in [1,\infty)$, then $\Gamma^{p}(\textrm{End}_{\mathcal{K}}(E),\tau)$ is defined as the Hausdorff-completion of $\Gamma_{c}(\textrm{End}_{\mathcal{K}}(E))$ w.r.t.~the semi-norm

\begin{align*}
||F||_{p}:=\sum_{i\in I}\int_{U_{i}}\varphi_{i}(x)\textrm{tr}(F_{|U_{i}}(x))d\nu_{i}.
\end{align*}

\noindent For $p=\infty$, let $\Gamma^{\infty}(\textrm{End}_{\mathcal{K}}(E),\tau)$ be the space of bounded measurable sections modulo nullsets with norm

\begin{align*}
||F||_{\infty}:=\esssup_{x\in X} ||F(x)||_{\mathcal{B}(H)}.
\end{align*}
\end{dfn}

\begin{rem}
All $\Gamma^{p}(\textrm{End}_{\mathcal{K}}(E),\tau)$ are Banach spaces. $\Gamma^{2}(\textrm{End}_{\mathcal{K}}(E),\tau)$ is a Hilbert space with the obvious inner product, and we represent $\Gamma^{\infty}(\textrm{End}_{\mathcal{K}}(E),\tau)$ canonically over $\Gamma^{2}(\textrm{End}_{\mathcal{K}}(E),\tau)$. This representation trivialises to the usual one used in the fourth section. The definition of $||.||_{\infty}$ makes sense since structure maps are unitary.
\end{rem}

\begin{ntn}
From now on, we suppress $\tau$ in the notation of the above $L^{p}$-spaces.
\end{ntn}

\begin{prp}
For all $p\in [0,\infty]$, we have $L^{p}(\Gamma_{0}(E),\tau)=\Gamma^{p}(\End_{\mathcal{K}}(E))$.
\end{prp}
\begin{proof}
Use $(3)$ and Proposition \ref{PRP.Prd_Tr_LP}.
\end{proof}

Let $\bigoplus_{k=1}^{m}E$ be the $m$-th Whitney sum of $E$. Next, fix a product trace $\tau$ and consider the $\Gamma^{\infty}(\textrm{End}_{\mathcal{K}}(E))$-bimodule $\bigoplus_{k=1}^{m}\Gamma^{2}(\textrm{End}_{\mathcal{K}}(E))=\Gamma^{2}(\textrm{End}_{\mathcal{K}}(\bigoplus_{k=1}^{m}E))$. By construction, the latter trivialises to $L^{2}(U,\bigoplus_{k=1}^{m}\mathcal{S}_{2}(H),\tau_{|U})$ for each $U\in\mathcal{T}(E)$. For a continuous partition of unity $(\varphi_{i})_{i\in I}$, we have

\begin{align*}
FG=\sum_{i\in I}\varphi_{i}F_{|U_{i}}G_{|U_{i}},\ GF=\sum_{i\in I}\varphi_{i}G_{|U_{i}}F_{|U_{i}}
\end{align*}

\noindent for each $F\in\Gamma_{\infty}(\textrm{End}_{\mathcal{K}}(E))$ and $G\in\Gamma^{2}(\textrm{End}_{\mathcal{K}}(E))$. As in the trivial case, we obtain a canonical symmetric bimodule structure. Moreover, we see that

\begin{align}
M_{P}(G)=\sum_{i\in I}\varphi_{i}M_{P_{|U_{i}}}(G_{|U_{i}})
\end{align}

\noindent for each $P\in\mathcal{D}_{b}$ and $G\in\bigoplus_{k=1}^{m}\Gamma^{2}(\textrm{End}_{\mathcal{K}}(E))\cap\bigoplus_{k=1}^{m}\Gamma_{\textrm{loc}}^{\infty}(\textrm{End}_{\mathcal{K}}(E))$. Here, $\Gamma_{\textrm{loc}}^{\infty}(\textrm{End}_{\mathcal{K}}(E))$ denotes the locally bounded sections modulo nullsets.\par
Consider an unbounded $C_{c}(X)$-module map $\Phi$. Given a continuous partition of unity $(\varphi_{i})_{i\in I}$, we define linear maps from $\Gamma^{2}(\textrm{End}_{\mathcal{K}}(E_{|U_{i}}))$ to $\bigoplus_{k=1}^{m}\Gamma^{2}(\textrm{End}_{\mathcal{K}}(E_{|U_{i}}))$ by setting $\Phi_{|U_{i}}(F):=\Phi(\varphi F_{|U_{i}})=\varphi_{i}\Phi (F)$.

\begin{rem}
In the next definition, we could choose to replace $\bigoplus_{k=1}^{m}\Gamma^{2}(\textrm{End}_{\mathcal{K}}(E))$ by a symmetric Hilbert $\Gamma^{\infty}(\textrm{End}_{\mathcal{K}}(E))$-subbimodule which is furthermore a subsheaf. However, composing $\partial$ with the subsheaf inclusion would then yield a vertical gradient in the sense of Definition \ref{DFN.Vrt_Grd_Bdl}.
\end{rem}

\begin{dfn}\label{DFN.Vrt_Grd_Bdl}
An unbounded $C_{c}(X)$-module map $\partial:\Gamma^{2}(\textrm{End}_{\mathcal{K}}(E))\longrightarrow\bigoplus_{k=1}^{m}\Gamma^{2}(\textrm{End}_{\mathcal{K}}(E))$ is a vertical gradient if there exists a continuous partition of unity $(\varphi_{i})_{i\in I}$ such that each $\partial_{|U_{i}}$ is a vertical gradient in the sense of Definition \ref{DFN.Vrt_Grd}.
\end{dfn}

\begin{rem}
Definition \ref{DFN.Vrt_Grd_Bdl} is consistent with Definition \ref{DFN.Vrt_Grd}. Arguing as in Proposition \ref{PRP.Prd_Tr_All_Part}, an unbounded $C_{c}(X)$-module map $\partial$ as above is a vertical gradient if and only if all $\partial_{|U_{i}}$ are vertical gradients for each continuous partition of unity $(\varphi_{i})_{i\in I}$. 
\end{rem}

Set $\mathfrak{A}:=\{F\in \Gamma_{0}(\textrm{End}_{\mathcal{K}}(E))\cap\Gamma^{2}(\textrm{End}_{\mathcal{K}}(E))\ |\ \partial F\in\bigoplus_{k=1}^{n}\Gamma_{\textrm{loc}}^{\infty}(\textrm{End}_{\mathcal{K}}(E))\}$. Then $\mathfrak{A}$ is an extension algebra by $(4)$ and $D_{cb}(\partial_{|U})=C_{c}(U)\otimes \FinRk(H)$ for each $U\in\mathcal{T}(E)$. We proceed as in Subsection 4 to obtain a disintegration theorem for general $\mathcal{K}(H)$-bundles.

\setcounter{section}{4}
\setcounter{thm}{0}

\begin{thm}\label{THM.Bdl_Disint}
Let $\partial$ be a vertical gradient such that $(\partial_{|U})_{x}$ has continuous dependence of minimisers on start- and endpoints for a.e.~$x\in U$ for each $U\in\mathcal{T}(E)$. For all $P,Q\in\mathcal{D}$ with finite distance and all partitions of unity $(\varphi_{i})_{i\in\mathbb{N}}$, we have 

\begin{align*}
\mathcal{W}_{2}^{2}(P,Q)=\sum_{i}\int_{U_{i}}\varphi_{i}(x)\mathcal{W}_{2}^{2}(\theta_{P_{|U_{i}}}^{2}P_{|U_{i}}(x),\theta_{P_{|U_{i}}}^{2}Q_{|U_{i}}(x))d\nu_{P_{|U_{i}}}
\end{align*}

\noindent and there exists a minimiser $\mu_{t}$ of $\mathcal{W}_{2}(P,Q)$ such that $\theta_{P}(x)^{2}\mu_{t}(x)$ is a fibre-wise minimiser a.e.
\end{thm}

\setcounter{section}{5}
\setcounter{thm}{0}

\subsection{Sufficient conditions involving Morita equivalence}

We begin by giving sufficient conditions in case $A$ and $\tau$ are already of the required form, i.e. as in the previous subsection. We thus concern ourselves with $\partial$ only. Moreover, we restrict to the finite-dimensional case.

\begin{ntn}
An unbounded linear map trivialising to an unbounded linear map for each $U\in\mathcal{T}(E)$ is called an unbounded bundle map.
\end{ntn}

\begin{prp}\label{PRP.Grd_Dcp_I}
Let $H$ be finite-dimensional, $\tau$ a product trace and $\partial$ is a symmetric gradient for $(\Gamma_{0}(\End_{\mathcal{K}}(E)),\tau)$ mapping to $\bigoplus_{k=1}^{m}\Gamma^{2}(\End_{\mathcal{K}}(E))$. Then $\partial$ is a vertical gradient if for all $U\in\mathcal{T}(E)$, we know that

\begin{itemize}
\item[1)] $\partial_{|U}$ is bounded,
\item[2)] $C_{c}(U)\subset\ker\partial_{|U}$.
\end{itemize}
\end{prp}
\begin{proof}
Let $\textrm{dim}_{\mathbb{C}}(H)=n$ and $U\in\mathcal{T}(E)$ such that $\tau_{|U}(1_{U}\otimes 1_{M_{n}(\mathbb{C})})<\infty$. By $1)$, we then have $L^{\infty}(U)\otimes M_{n}(\mathbb{C})\subset D(\partial)$. Thus for all $g\in C_{0}(U)$ and $F\in C_{0}(U,M_{n}(\mathbb{C}))$, we have $\partial_{|U}(gF)=g\partial_{|U}(F)$ by the Leibniz-rule and $2)$. Since $\partial_{|U}$ is bounded, $\partial$ commutes with $C_{0}(U)$ and hence $L^{\infty}(U)$. Applying Theorem IV.7.10 in \cite{TakTOAI} shows that $\partial_{|U}$ decomposes into bounded linear operators $\partial_{x}$ from $M_{n}(\mathbb{C})$ to $\bigoplus_{k=1}^{m}M_{n}(\mathbb{C})$. As $\partial$ is furthermore an unbounded bundle map, it therefore is an unbounded $C_{0}(X)$-module map. By finite-dimensionality, each $\partial_{x}$ is a fibre-gradient. Because $\partial_{|U}$ was assumed to be a symmetric gradient, $\partial_{x}$ must be a symmetric gradient for almost every $x\in U$. Boundedness of each $\partial_{|U}$ and finite-dimensionality of $H$ ensure all remaining conditions in Definition \ref{DFN.Vrt_Grd} to be met.
\end{proof}

Consider $(A,\tau,\partial)$ for unital $A$ being Morita equivalent to a $C(X)$, $X$ compact Hausdorff. By compactness, $X$ has sufficiently many partitions of unity. Assume $\partial$ to map into $\bigoplus_{k=1}^{m}L^{2}(A,\tau)$. By Morita equivalence, we have an isomorphism $\Phi$ from $A$ to a $C(X,\textrm{End}_{\mathcal{K}}(E))$ as unitality ensures any Hilbert module implementing the equivalence to be finitely projective \cite{KhalBasicNCG}. $E$ is a finite-dimensional hermitian vector bundle. Moreover, we have isomorphisms from $L^{p}(A,\tau)$ to $\Gamma^{p}(\textrm{End}_{\mathcal{K}}(E),\Phi_{*}\tau)$ for each $p\in [0,\infty]$. By Corollary \ref{COR.Prd_Tr_Bdl}, finiteness of $\tau$ implies $\Phi_{*}\tau$ to be a product trace.

\begin{dfn}\label{DFN.Disint}
Let $(A,\tau,\partial)$ such that $A$ is unital and Morita equivalent to $C(X)$, $X$ compact. Furthermore, assume $\tau$ is finite and that $\partial$ maps to $\bigoplus_{k=1}^{m}L^{2}(A,\tau)$, $m\in\mathbb{N}$. We say that $\mathcal{W}_{2}$ disintegrates if there exists an isomorphism $\Phi$ such that $\Phi_{*}\partial$ is a vertical gradient.
\end{dfn}

\begin{cor}\label{COR.Grd_Dcp_I}
If we are in the setting of Definition \ref{DFN.Disint}, then $\mathcal{W}_{2}$ disintegrates if $\Phi_{*}\partial$ satisfies the conditions in Proposition \ref{PRP.Grd_Dcp_I}.
\end{cor}

It is clear that the sufficient conditions presented here are too strong for easy application. For example, we require better conditions for choosing isomorphisms such that $\phi_{*}\partial$ is at least an unbounded bundle map. Nevertheless, they give a first tentative attempt to reduce general problems to the vertical gradient case.\par
There are two main avenues we wish to explore. Firstly, we require conditions for having continuous dependence of minimisers on start- and endpoints of symmetric gradients for $(\mathcal{K}(H),\textrm{tr})$ and for the hyperfinite type $\textrm{II}_{1}$ factor $R$ equipped with its canonical trace $\tau_{0}$. Secondly, we consider more general direct integrals than $L^{2}(X,H)$ since we seek to understand gradients after disintegrating $L^{\infty}(A,\tau)$ into its factors. A natural point of departure are fields of elementary $C^{*}$-algebras. However, even if a continuous field of elementary $C^{*}$-algebras satisfies Fell's condition it need not equal the induced field of elementary $C^{*}$-algebras associated to its direct integral of Hilbert spaces, cmpr.~Theorem 10.7.15 in \cite{DixC*Alg}. In our setting, this is necessary for having $L^{\infty}(X,\mathcal{B}(H))=L^{\infty}(A,\tau)$. Thus not all direct integrals are immediately suitable to our purposes.\par
Once we have determined a class of direct integrals and generalised the notion of vertical gradient, we hope to apply results of form \cite{MathUnbdDecomp} in order to decompose unbounded gradients between direct integrals. With the outlined approach, we aim to cover a large number of $C^{*}$-algebras whose $L^{\infty}$-space is isomorphic to that induced by a direct integral whose fibres are given by some $\mathcal{S}_{2}(H)$ or $L^{2}(R,\tau_{0})$.

\end{document}